\documentclass[leqno,12pt]{amsart}

\allowdisplaybreaks[1]

\usepackage{amssymb,amsfonts,amsmath,amscd,latexsym,amsxtra}
\usepackage[mathscr]{eucal}
\theoremstyle{plain}
\newtheorem{theorem}{Theorem}[section]
\newtheorem{corollary}[theorem]{Corollary}

\newtheorem{lemma}[theorem]{Lemma}
\newtheorem{proposition}[theorem]{Proposition}

\theoremstyle{example}

\theoremstyle{definition}
\newtheorem{definition}[theorem]{Definition}

\theoremstyle{remark}
\newtheorem{remark}[theorem]{Remark}
\theoremstyle{notation}

\theoremstyle{e}

\numberwithin{equation}{section}

\newcommand{\Spec}{\operatorname{Spec}}
\newcommand{\Hom}{\operatorname{Hom}}

\newcommand{\id}{\operatorname{id}}
\newcommand{\GL}{\operatorname{GL}}
\newcommand{\SL}{\operatorname{SL}}
\newcommand{\lsl}{\mathfrak{s}\mathfrak{l}}
\newcommand{\V}{\mathcal{O}}
\newcommand{\Ur}{\mathbb U^n_r}
\newcommand {\Lrn}{\mathbb Lat_r^n(K)}

\newcommand{\W}{\mathfrak{W}}

\newcommand{\bu}{\boldsymbol{u}}
\newcommand{\Cr}{\widetilde{\Gamma}_r}
\newcommand{\lie}[1]{\ensuremath{\mathfrak{L}(#1)}}
\newcommand{\ws}[1]{\ensuremath{\mathfrak{w}\mathfrak{s}\mathfrak{l}(n,#1 \mathcal{O})}}
\newcommand{\wg}[1]{\ensuremath{\mathfrak{w}\mathfrak{g}\mathfrak{l}(n,#1 \mathcal{O})}}

\newcommand{\bw}{\boldsymbol{w}}

\newcommand{\X}{\mathbb{X}}
\newcommand{\pr}{\mathbb{P}}
\newcommand{\x}{\boldsymbol{x}}
\newcommand{\y}{\boldsymbol{y}}

\newcommand{\be}{\boldsymbol{\epsilon}}

\newcommand{\Or}{\mathfrak{O}}
\newcommand{\F}{\mathfrak{F}}
\newcommand{\fu}{\mathfrak{u}}
\newcommand{\thetao}{\overline{\theta}}
\newcommand{\OOr}{\overline{\mathfrak{O}}}
\newcommand{\OF}{\overline{\mathfrak{F}}}

\newcommand{\lienr}{\mathcal{L}_{n,r}}

\newcommand{\fm}{\mathfrak{m}}
\newcommand{\ke}{k_{\epsilon}}
\begin{document}
\title[Varieties of Special Lattices]{The Normality of Certain Varieties of Special Lattices}
\author[W. J. Haboush]{William J. Haboush}
\address{Department of Mathematics \\ 273 Altgeld Hall \\ 1409 W. Green St.\\ The University of Illinois at Urbana Champaign \\ Urbana, Illinois \\ 61801}
\email{haboush@math.uiuc.edu}
\author[A. Sano]{Akira Sano}
\address{Center for Advanced Studies in Mathematics. Ben-Gurion University of the Negev, Be'er-Sheva 84108, Israel}
\email{sano@math.bgu.ac.il}
\thanks{The second author was supported by the Center for Advanced Studies in
Mathematics, Ben-Gurion University of the Negev, Be'er-Sheva, Israel}
\subjclass{20G25, 20G99, 14L15}
\keywords{Witt Vectors, Lattices, Lattice Varieties, Normal Varieties}
\begin{abstract}
We begin with a short exposition of the theory of lattice varieties. This includes a description of their orbit structure and smooth locus. We construct a flat cover of the lattice variety and show that it is a complete intersection. We show that the lattice variety is locally a complete intersection and nonsingular in codimension one and hence normal. We then prove a comparison theorem showing that this theory becomes parallel to the function field case if linear algebra is replaced by $p$-linear algebra. We then compute the Lie algebra of the special linear group over the truncated Witt vectors. We conclude by applying these results to show how to describe the canonical bundle on the lattice variety and we use the description to show that lattice varieties are not isomorphic to the analogous objects in the affine Grassmannian.  
\end{abstract}
\maketitle
\tableofcontents

\section{Introduction}\label{S:Int}
In two recent papers, the senior author of this paper constructed varieties parametrizing certain families of lattices over a certain discrete valuation ring. The direct limit of these varieties is viewed as a mixed characteristic analogue of the infinite Grassmann variety. Let $k$ be the algebraic closure of $\mathbb{F}_p,$ the field with $p$ elements, let $\V =\W(k)$ denote the ring of Witt vectors of $k$ (see \cite[p.41 \&  ff]{cl}) and let $K$ denote its field of fractions. Let
$\xi:k\rightarrow \V$ denote the extended multiplicative representative map, that is, the classical multiplicative representative on $k^*$ extended to $k$ by setting $\xi(0)$ equal to $0$. Then elements of $\V$ can all be uniquely represented as sums $\bw(x_0, x_1, \dots )=\sum_{i=0}^{\infty}\xi(x_i)^{p^{-i}}p^i$. The exponents are unavoidable if the multiplication is to be a polynomial map. The product of two Witt vectors with only one nonvanishing component is given by the formula:
\[
\xi(x)^{p^{-s}}p^s\cdot \xi(y)^{p^{-r}}p^r=\xi(x^{p^r}y^{p^s})^{p^{-r-s}}p^{r+s}.
\]
In particular, multiplication by $p$ is a purely inseparable map. This also means that if $r<0$, one must use a $p^{-r}$'th root of $x$ to describe the multiplication.   For this reason $K$, the field of fractions of $\V,$ does not carry an algebraic multiplication. The fraction field does carry a structure as an inductive limit of group schemes for the additive structure and this additive structure is defined over $\mathbb{F}_p$ (but not over $\mathbb{Z}).$ One may treat $\GL(n,\V)$ as a transformation group acting on the module $\V^n$. On the $n$-dimensional vector space $K^n$ it acts at most as a quasialgebraic group (see \cite{grP} for the definition). This is one of the central problems of this paper as well as the papers \cite{me} and \cite{me2}. 

Choose a free $\V$-submodule of $K^n$ which we write $F=\V^n$. A lattice is a free rank-$n$ $\V$-submodule of $K^n$. It is called \emph{special} if its $n$'th exterior power is equal to the $n$'th exterior power of $F$. Let $\Lrn$ denote the set of special lattices contained in $p^{-r}F$. The most significant results of \cite{me} concern this object. It is shown to be projective, reduced and irreducible over $k$ and it is shown to have an algebraic action of the group $\SL (n, \V)$. The orbit structure of $\Lrn$ under this group and under an Iwahori subgroup as well are described and the tangent space to a lattice in $\Lrn$ is computed. Section 2 of this paper consists of a review of these results.

It is interesting to contrast the constructions in this paper and those in \cite{me} and \cite{me2} with those used by various authors in studying the infinite Grassmannian. This theory, which originates in works of H. Garland \cite{hg} and of G. Segal \cite{ps} on the loop Grassmannian,
found its algebraic expression in such work as \cite{kl} or \cite{george}. There one considers points in classical Grassmann varieties fixed under a nilpotent endomorphism $T.$ The $T$-fixed subspaces of fixed codimension can be viewed as $k[[T]]$-modules residually of a given dimension and of a fixed length. Since one is working with a Grassmannian, points naturally correspond to vector subspaces and it is $T$-fixedness which is a nontrivial condition.  To construct the lattice varieties one must consider $k^*$-stable subgroups of the unipotent commutative group scheme $p^{-r}F/p^{(n-1)r}F$ of codimension $nr$ that are $\V$-submodules.
This means they must be fixed under the uniformizing parameter $p$ which is automatic since $p\in \mathbb{Z}$, and fixed under $k^*$ which is nontrivial. This is exactly opposite to the situation in the infinite Grassmannian. The $k^*$-closedness amounts to requiring that the subgroups correspond to projective subschemes of the quotient of $p^{-r}F/p^{(n-1)r}F\setminus \{ 0\}$ by the $k^*$-action. Hence one must work with a Hilbert scheme and in consequence one obtains a flat universal family but at the expense of not having a very explicit construction. The universal family of lattices is birational to the affine cone over the universal family associated to the Hilbert scheme of $(p^{-r}F/p^{(n-1)r}F\setminus \{ 0\})/k^*.$ The flatness of this family, which is automatic on the complement of the null section, is an important result in \cite{me} (Lemma 8 and Proposition 6, p.87 and Theorem 4, p.89). It is intimately related to the flatness of the universal family on the Hilbert scheme. All of the issues in this  paragraph and the one above are discussed at length in \cite{me}.

This paper contains several important new results. The first is a proof that $\Lrn $ is normal and locally a complete intersection. This is done in two steps. First in Section~\ref{bfd}, a ``matrix" cover of the lattice variety is constructed as a subset of an $n$-fold relative Cartesian power of the universal family of lattices, and it is shown to be a complete intersection. In Section~\ref{Smt}, the proof of normality is given. The first version of the proof, which was influenced by \cite{kp}, was given in \cite{ast}.  In Section \ref{CC} we extend a result of \cite{me2} to show that $\Lrn$ is naturally isomorphic to a scheme of lattices over a noncommutative power series ring (the Ore ring). Thus we can think of $\Lrn$ as either a mixed characteristic analogue of the function field infinite Grassmannian or a noncommutative analogue of the same scheme. In section \ref{pln} we use this second characterization to show that $\mathbb{L}at_1^n(K)$ contains an open subset equivariantly isomorphic to a set of $p$-linear nilpotent transformations. The equations for this scheme are rather different than those of the nilpotent cone.   In Section~\ref{OC} we show that the smooth locus of the lattice variety is a fiber space over $\pr_k^{n-1}$ with fibers that are affine spaces with a $p$-adic linear fractional modular action.

In \cite{me2} it is proven that the Picard group of the smooth locus of $\Lrn$ is isomorphic to $\mathbb{Z}$ with an ample generator. Moreover, as we already mentioned the smooth locus consists of a single $\SL(n,\V)$-orbit and its complement is of codimension two. Hence any finite reflexive sheaf and in particular any locally free sheaf is determined by its restriction to the smooth locus which is a homogeneous space. Hence by normality line bundles may be computed on the smooth locus and they are uniquely determined. Since the smooth locus is a homogeneous space, a line bundle is determined by a character on the stabilizer of a point.  In section 8 we carry out the computations necessary to describe the canonical bundle. We also observe that if the same calculation were carried out on the affine Grassmannian a different integer would occur and so the two schemes are certainly not isomorphic.

Finally we conclude the paper with two appendices which give some further details on some of the proofs in \cite{me} and \cite{me2}. We had not thought these details necessary but persistent questions from colleagues have led the senior author to believe that it is appropriate to expand upon them.

\section{Lattice Varieties}\label{lv}
Write $e_1,\dots ,e_n$ for the standard $K$-basis of $K^n$ so that it is an $\V$-basis of $F$. If $L$ is any lattice in $K^n,$ then $\bigwedge _{\V}L=p^{d_L}\bigwedge_{\V}F$ for some $d_L\in \mathbb{Z}$. Then $d_L$ is the \emph{discriminant} of $L$ with respect to $F$ and it is equal to $\ell (F/(L\cap F))-\ell (L/(L\cap F)) $ where, for $M$ an $\V$-module of finite length, $\ell (M)$ denotes its length. Special lattices are lattices of discriminant $0$. Then the fundamental existence result is:
\begin{proposition}\cite[\S 3, Proposition 6]{me}
Let $q, m$ and $N$ be three integers such that $q<m$ and $N<n(m-q)$. Then there is a $k$-scheme $\mathbb{L}(q,m;N)$ which is projective and of finite type over $k$ and a flat commutative $\mathbb{L}(q,m;N)$-group scheme $\mathbb{U}(q,m; N)\subseteq p^qF/p^mF\times _k\mathbb{L}(q,m;N)$ which is a universal family of flat module subschemes of $p^qF/p^mF$ of dimension $N$ parametrized by $\mathbb{L}(q,m;N)$. That is, for any $k$-scheme $Y$ and any flat $Y$-group scheme $\mathcal{M}$ whose fibers are $\V$-submodules of $p^qF/p^mF$ of length $N$, there is a unique map $f:Y\rightarrow \mathbb{L}(q,m;N)$ so that $\mathcal{M}\simeq Y\times_{\mathbb{L}(q,m;N)}\mathbb{U}(q,m;N)$.
\end{proposition}

A remark is in order. If a lattice $L$ is in $p^qF$ and it is of codimension $d$, then $L\supseteq p^{q+d}F$. This means that for a given codimension $d$, the scheme $\mathbb{L}(q,m;n(m-q)-d)$ is independent of $m$ so long as $m\geq q+d$. This is exactly the condition that guarantees that every lattice of codimension $d$ in $p^qF$ contains $p^mF$.

By definition, $\Lrn$ is $\mathbb{L}(-r,(n-1)r; n(n-1)r)$,
and its universal family is $\mathbb{U}_r^n$. By means of a suitable power of an $n$-fold product of the verschiebung, $\Lrn$ may be shown to be isomorphic to $\mathbb{L}(0,nr;n(n-1)r)$. In what follows we will view $\Lrn$ as this scheme. This can be thought of as the set of lattices in $F$ of discriminant $nr$. The proalgebraic group $\SL (n,\V )$ operates canonically on $F$ and it preserves the discriminant, hence it operates algebraically on $\Lrn$ \cite[\S 3.4, Lemma~12, Definition~23 and following]{me}. \emph{ Note that in }\cite{me} \emph{ we do not identify the two schemes. There they are written as }$\mathbb{L}_r(F)$ \emph{ and $\Lrn .$}

\section{Bases of Fixed Discriminant}\label{bfd}
The results of this section are auxiliary to those of the next. We construct the schemes and maps which will be used to prove our major results. The first such object is the scheme of bases.

Each lattice of discriminant $nr$ in $F$ contains $p^{nr+1}F$. Hence if $L$ is of discriminant $nr$ in $F$, it is uniquely determined by $L/p^{nr+1}F$. If $\overline{L}$ is a $k^*$-stable subgroup of $F/p^{nr+1}F$ of codimension $nr$, its inverse image in $F$ is a lattice of discriminant $nr$ (see the proof of proposition 6 of \S3.4 of \cite{me}, the proposition quoted above). The quotient $F/p^{nr+1}F$ can be viewed as a free module over the truncated Witt vectors $\W _{nr+1}(k)=\V /p^{nr+1}\V$. If $u$ is a truncated Witt vector, we may write it uniquely as $\sum_{i=0}^{nr}\xi(x_i)^{p^{-i}}\overline{p}^i$ where the bar denotes residue class mod $p^{nr+1}$. Hence if $\boldsymbol{u}_i$ is a vector in $F/p^{nr+1}F$, it has coordinates $\{u_{i,j}:\,j=1,\dots ,n\}$ each of which may be written as  $u_{i,j}=\sum_{s=0}^{nr}\xi(x_{i,j,s})^{p^{-s}}\overline{p}^s$. Hence an ordered set of $n$ elements of $F/p^{nr+1}F$, which we write $\boldsymbol{u}_1, \dots ,\boldsymbol{u}_n$, can be thought of as a point with the $n^2(nr+1)$ coordinates $x_{i,j,s}$, where $\bu _i=(u_{i,1},\dots , u_{i,n})$ and $u_{i,j}=\sum_{s=0}^{nr}\xi(x_{i,j,s})^{p^{-s}}\overline{p}^s$. Thus the space of $n$-tuples of vectors in $F/p^{nr+1}F$ is affine $n^2(nr+1)$-space with the coordinates $x_{i,j,s}$. Write $S$ for this set of indeterminates. In subsequent sections we shall omit the bar on $p$ when it is a residue class.

If $\bu _1, \dots ,\bu _n$ is such an $n$-tuple of vectors, then $\bu _1\wedge \bu _2\wedge \dots \wedge \bu _n=\Delta e_1\wedge \dots \wedge e_n$ where $\Delta$ is the determinant of the matrix with entries $u_{i,j}$. Since $\Delta$ is a polynomial in a set of Witt vectors, we may write: 
\begin{equation}\label{del}
\Delta=\sum_{i=0}^{nr}\xi(\delta_i(S))^{p^{-i}}p^i
\end{equation}
Here each of the $\delta_i$ is a polynomial in the set $S$ of indeterminates $x_{i,j,s}$. To require that the determinant $\Delta$ be of $p$-adic valuation $nr$ is exactly to require that $\delta_i=0,\,i<nr,\; \delta_{nr}\neq 0$. Hence the set of ordered $n$-tuples of vectors $\{\bu _1, \dots ,\bu _n\}$, whose exterior product $\bu _1\wedge \bu _2\wedge \dots \wedge \bu _n$ is of valuation $nr$, is the spectrum of the ring $k[S]_{\delta_{nr}}/I$ where $I$ is the ideal generated by $\delta_0, \dots ,\delta_{nr-1}$. That is, it is an affine subscheme of a closed subscheme of the affine $n^2(nr+1)$-space of codimension at most $nr$. In particular, if it can be shown that it is a variety of dimension $n^2(nr+1)-nr$, it is a complete intersection scheme. It suffices that it be shown to be irreducible and generically reduced.

\begin{definition}\label{base}
The scheme of bases in $F$ of discriminant $nr$ is the scheme $\Spec (k[S]_{\delta'_{nr}}/I)$. We will write it $\mathcal{B}_{nr}(F)$.
\end{definition}

This is a functorially defined scheme. That is, by defining it by means of the particular equations we have specified for it we may say that for any $k$-scheme, $Z,$ the functor of points, $\Hom_k(Z, \mathcal{B}_{nr}(F)),$ can be viewed as the set of equivalence classes modulo $p^{rn+1}F$ of ordered sets of $n$ sections of $F\times_kZ$ which are globally linearly independent over the ring of global Witt vectors over $Z$ and which generate a submodule with the specified discriminant. On a more explicit level it is a scheme of matrices over $\V/p^{nr+1}\V.$ If $\{ u_1,\dots ,u_n\}$ is any base of discriminant $nr$ we may think of it as a matrix in $\mathbb{M}_n(\V/p^{nr+1}\V)$with columns $u_1, \dots ,u_n.$ It will be a matrix with determinant a unit multiple of $p^{nr}.$ Conversely given any such matrix, its columns are a base of discriminant $nr.$

\begin{lemma}\label{irr} The scheme of bases of discriminant $nr$ is irreducible.
\end{lemma}

\begin{proof}Let $G=\SL(n, \V/p^{nr+1}\V).$ Then $G\times_kG$ operates on the scheme of bases of discriminant $nr.$ If $U$ is a base of discriminant $nr$ and $(\sigma, \tau)\in G\times G,$ then $\sigma U \tau^{-1}$ is again a base of discriminant $nr.$ As a consequence of the Iwahori decomposition every base is in the orbit of a base of the form $p^{r_1}e_1, \dots , p^{r_n}e_n$ where the $r_i$ may be taken in nonincreasing order. If $r_i>r_j$ we construct a path of bases parametrized by the affine line $\Spec (k[t])$ which is in the orbit of $p^{r_1}e_1,\dots ,p^{r_i+1}e_i, \dots ,p^{r_j-1}e_j,\dots p^{r_n}e_n$ for $t$ invertible and which is just the given base for $t=0.$

Just use the base which is the same as the given one for $q\neq j$ but in which $p^{r_j}e_j$ is replaced by $u_j=p^{r_j}e_j+\xi(t)p^{r_i-r_j+1}e_i.$ For $t=0$ one clearly obtains the given base. When $t$ is invertible leave all but the $i$'th and $j$'th columns alone. Notice that the vectors $p^{r_i}e_i$ and $u_j$ can be replaced by $\xi(t)(p^{r_i}e_i-\xi(t)^{-1}p^{r_i-r_j+1})$ and $\xi(t)^{-1}u_j.$ Thus the $i$'th vector is now $p^{r_i+1}e_j$ while the $i$'th entry in $\xi(t)^{-1}u_j$ is $p^{r_j-1}.$ After multiplying by appropriate permutations this means that the base $\{ p^{r_1}e_1,\dots ,p^{r_i}e_i, \dots ,p^{r_j}e_j,\dots p^{r_n}e_n\}$ is in the closure of the orbit of the base $\{ p^{r_1}e_1,\dots ,p^{r_i+1}e_i, \dots ,p^{r_j-1}e_j,\dots p^{r_n}e_n\}.$ By a sequence of such paths one sees that every base is in the closure of the orbit of the base $\{ p^{nr}e_1,e_2,  \dots , e_n\}.$ Hence that orbit is dense. Irreducibility follows. 
\end{proof}

\begin{remark} If one crudely uses the family of bases described above to span a family of lattices over the family of bases one finds that the fiber at $0$ is not reduced. The path described above can be thought of as a path in a Cartesian power universal family of lattices. As such it has a vertical tangent at $t=0$ and so in the lattice variety it gives rise to a singular path and a nonflat family. To obtain the proper deformation of lattices with a reduced fiber at $0$ one must take the total space of this deformation and replace it by its normalization. Alternatively one may construct explicit deformations that show that any lattice of co-rank $nr$ can be obtained from the maximal lattice by a sequence of flat deformations. We will give such an explicit construction in an appendix. 
\end{remark}  

Suppose that $\{ \bu_1, \dots ,\bu_n \}$ is a base of discriminant $nr$ and that $\boldsymbol{v}$ is another vector in $\V^n.$ Let $L$ be the lattice spanned by the $\bu_i.$ Then it is an exercise in the theory of elementary divisors to prove that $\boldsymbol{v}\in L$ if and only if $\V \boldsymbol{v}+L$ is also of discriminant $nr.$ Write $L'=\V\boldsymbol{v}+L.$ To say that $L'$ is of discriminant $nr$ is to say that $\bigwedge^n_{\V}L'=p^{nr}\bigwedge^n_{\V}{F}.$ 

Let $\Delta_j$ be the determinant all of whose columns but the $j$'th are the $\bu_i$ and whose $j$'th column is $\boldsymbol{v}.$
Now $\bigwedge^n_{\V}L'$ is generated by $\bu_1\wedge \dots \wedge \bu_n$ and the $n$ products, $\bu_1\wedge \dots \wedge \bu_{j-1}\wedge \boldsymbol{v}\wedge \bu_{j+1} \dots \wedge \bu_n.$ Thus the discriminant of $L'$ is the colength of the ideal generated by $\Delta$ and the $n$ determinants $\Delta_j.$  Write $\boldsymbol{v}=(v_1, \dots ,v_n)$ and $v_j=\sum_{s=0}^{nr}\xi(v_{i,s})^{p^{-s}}p^s.$ Write $T$ for the set of indeterminants: $$S\bigcup \{v_{i,s}: 1\leq i\leq n;\; 0\leq s\leq nr\}.$$ Reasoning as in the paragraph preceding \eqref{del}, we may write:
\begin{equation}\label{mdel}
\Delta_j =\sum_{i=0}^{nr}\xi(\eta_{j,i}(T))^{p^{-i}}p^i
\end{equation}

\begin{definition}The \emph{spanning family} over $\mathcal{B}_{nr}(F)$ is the closed subscheme of $\Spec(k[T]_{\delta_{nr}})$ defined by the polynomials $\delta_j,\;0\leq j\leq nr-1$ and $\eta_{j,i}\quad j=1,\dots ,n,\;0\leq i\leq nr-1$ with its induced reduced structure. This scheme will be written $\mathbb{U}(\mathcal{B}).$
\end{definition}

There is a natural homomorphism from the coordinate ring of $\mathcal{B}_{nr}(F)$ to the coordinate ring of $\mathbb{U}(\mathcal{B}).$ Just note that $k[S]_{\delta_{nr}}\subseteq k[T]_{\delta_{nr}}$ and that the ideal defining $\mathcal{B}_{nr}(F)$ is a subset of the set of polynomials defining $\mathbb{U}(\mathcal{B}).$ The associated morphism of schemes will be written $f:\mathbb{U}(\mathcal{B})\rightarrow \mathcal{B}_{nr}(F).$

\begin{lemma}\label{gpt} The geometric points of $\mathbb{U}(\mathcal{B})$ are the $(n+1)$-tuples \linebreak $(\boldsymbol{v}, \bu_1, \dots ,\bu_n )$ where $\boldsymbol{v}$ and the $\bu_i$ are $n$-vectors in $\overline{\V}$ such that \linebreak $( \bu_1, \dots ,\bu_n )$ is a base of discriminant $nr$ and $\boldsymbol{v}$ is in the $\V$-span of the $\bu_i.$ The map $f$ sends the geometric point $(\boldsymbol{v}, \bu_1, \dots ,\bu_n )$ to the base $(\bu_1, \dots ,\bu_n ).$
\end{lemma}

\begin{proof}   
Then the $n+1$-tuple $(\boldsymbol{v}, \bu_1, \dots ,\bu_n )$ will consist of a base of discriminant $nr$ and a vector $\boldsymbol{v}$ such that $\boldsymbol{v}$ is in the lattice spanned by the $\bu_i$ if and only if $\delta_{nr}$ is invertible and $\delta_j,\;0\leq j \leq nr-1,$ and the polynomials $\eta_{j,i},\; j=1,\dots ,n,\;0\leq i\leq nr-1$ are all $0.$ One implication is clear. To see that the vanishing of the $n(nr+1)$ polynomials and the invertibility of $\delta_{nr}$ imply that the point in question defines a base and a vector in the span of that base, just lift each of the vectors $\bu_i$ and $\boldsymbol{v}$ to vectors in $\V^n.$ Denote the corresponding vectors $(\boldsymbol{v}^*,\bu_1^*, \dots, \bu_n^*).$ Call the determinants of the matrices with these vectors as columns  like those above in the discussion preceding (3.2)  $\Delta^*$ and $\Delta_j^*.$ These are quantities  in $\V$ congruent modulo $p^{nr+1}$ to $\Delta$ and $\Delta_j\;j=1,\dots, n.$  Then Cramers rule states that the solution to the matrix equation  $(\bu_1^*,\dots, \bu_n^*)\left(\begin{smallmatrix}x_1^*&\\ \vdots&\\ x_n^*\end{smallmatrix}\right)=\boldsymbol{v}^*$ is given by $x_i^*=\Delta^*_i(\Delta^*)^{-1}.$ Now the polynomial equations given and the condition $\delta_{nr}$ is invertible together imply that the value of $\Delta^*$ is exactly $nr$ while the values of the $\Delta^*_i$ are all greater than or equal to $nr.$ Hence the $x_i^*$ are all in $\V$ and so their residue classes which we write $x_i$ satisfy $\boldsymbol{v}=\sum_1^nx_i\bu_i.$ The assertion concerning $f$ is clear. 
\end{proof}

\begin{remark}We warn the reader that this process does not produce a polynomial map. To divide by $(\Delta^*)$ is to divide by a unit multiple of $p^{nr}.$ Inverting a unit in $\V^*$ results in a Witt vector whose entries are polynomials in its coordinates and the inverse of its $0$-component. (This is a somewhat nontrivial exercise.) Inverting $p^{nr}$ however involves taking $p^{nr}$'th roots of the coordinates of the Witt vector in question. Since $k$ is perfect the solution described above is fine for geometric points but if one wishes to solve the problem for $R$-points the solution $\boldsymbol{v}$ will be an $R^{p^{-nr}}$-point.
\end{remark}

We comment here on what we are proving in this section. We establish two main results. The first is a proof, rather direct and explicit, that $\mathcal{B}_{nr}(F)$ is a complete intersection. Then we must show that the map which assigns to each base of discriminant $nr$ the lattice it generates is a faithfully flat morphism of schemes. In this section we show that it is a smooth morphism of schemes. Since there is no explicit description of the scheme $\Lrn,$ this can only be done by considering the functor which it represents. Maps to $\Lrn$ correspond to flat families of lattices and so showing the existence of a map is the same as constructing a flat family. The flat family over $\mathcal{B}_{nr}$ in this instance is the spanning family $\mathbb{U}(\mathcal{B}).$ 

\begin{lemma}\label{rv} The scheme $\mathcal{B}_{nr}(F)$ is generically rational of dimension $(n^2-1)(nr+1)+1=n^2(nr+1)-nr.$ That is it contains a dense open rational subset. \end{lemma}

\begin{proof}Recall Theorem 6, p. 113 of \cite{me}. A lattice $L\subseteq F$ is a smooth point of $\Lrn$ if it can be written in the form $L'\oplus \V p^{nr}v$  where $L'$ is a rank $n-1$ direct summand of $F$ and $F=\V v\oplus L'$. Clearly, $\{\bu _1, \dots ,\bu _n\}$ is a basis for such a lattice if and only if some $(n-1)\times (n-1)$ minor of the matrix with the $\bu_i$ as columns has unit determinant. For simplicity we consider those bases for which the lower right $(n-1)\times (n-1)$ minor is a unit, whence the last $n-1$ columns span a direct summand.

Write the matrix $U=(\bu_1,\dots ,\bu_n)$ in the form:
\begin{equation*}
\left(
\begin{matrix}
u_{1,1}  & U_{1,2}\\
U_{2,1}  & U_{2,2}
\end{matrix}
\right).
\end{equation*}
In this expression, $u_{1,1}\in \V/p^{nr+1}\V$, $U_{1,2}$ is a row of length $n-1$ with entries in the same ring and the lower entries are matrices of the obvious dimensions. Then $\gamma =\operatorname{det}(U_{2,2})$ is a function on $\mathcal{B}_{nr}(F)$ with value in $\V/p^{nr+1}\V$ and so we may write:
\[
\gamma =\sum_{i=0}^{nr}\xi(\gamma_i)^{p^{-i}}p^i.
\]
In this expression, the $\gamma_i$ are global functions on the scheme $\mathcal{B}_{nr}(F)$. Then $D(\gamma_0)$, the principal open set defined by $\gamma_0$, is exactly the set on which $\operatorname{det}U_{2,2}$ is invertible in $ \V/p^{nr+1}\V$.  The determinant of the full basis matrix $U=(\bu_1, \dots, \bu_n)$ as an element of $\V/p^{nr+1}\V$ is of the form $\xi(t)^{p^{-nr}}p^{nr} \in \V/p^{nr+1}\V$. Write $d_{i,j}$ for the determinant of the minor of $U$ obtained by deleting the $i$'th column and the $j$'th row. Note that $d_{1,1}=\gamma$. Then $\sum_{j=1}^n (-1)^{1+j}u_{1,j}d_{1,j}=\xi(t)^{p^{-nr}}p^{nr}$. Since $d_{1,1}=\gamma$, we may write:
\begin{equation}\label{E:par}
u_{1,1}=\gamma^{-1}(\xi(t)^{p^{-nr}}p^{nr}- \sum_{j=2}^nu_{1,j}(-1)^{1+j}d_{1,j}).
\end{equation}
On $D(\gamma_0)$ this expression can be realized as a truncated Witt vector whose Witt components are polynomials in regular functions on $D(\gamma_0)$. This is true of the points of $\mathcal{B}_{nr}(F)$ in the affine scheme $\Spec (R)$ for any $k$-algebra $R$.

Thus for any such $R$ and $(i,j)\neq (1,1)$, we may let the $u_{i,j}$ vary over the $R$-points of a dense open subspace of $ (\V/p^{nr+1}\V)^{n^2-1}$ and choose $t\in R^*$ arbitrarily and take $u_{1,1}$ given by formula \eqref{E:par}. That is, we have parametrized the $R$-points of $D(\gamma_0)$ as the $R$-points of a dense open subset of $(\V/p^{nr+1}\V)^{n^2-1}\times k^*$. That is, for each $k$-algebra $R$, the $R$-points of $D(\gamma_0)$ are functorially equal to the $R$-points of a dense open subset of $(\V/p^{nr+1}\V)^{n^2-1}\times k^*$. It is a fact well known and extensively applied in the theory of group schemes (the discussion in \cite{rag} and \cite{dg} is based on it) that, if the $R$-points of a $k$-scheme $X$ are functorially isomorphic to the $R$-points of a $k$-scheme $Y$ for each $k$-algebra $R$, then $X$ and $Y$ are isomorphic.  Hence $D(\gamma_0)$ is isomorphic to a dense open subset of an affine space of dimension $(n^2-1)(nr+1)+1=n^2(nr+1)-nr$, and so it is rational and hence integral as well.
\end{proof}

\begin{proposition}\label{cin}The scheme $\mathcal{B}_{nr}(F)$ is an affine complete intersection.
\end{proposition}

\begin{proof}By Lemma \ref{rv} and lemma \ref{irr}, $\mathcal{B}_{nr}(F)$ is irreducible, generically reduced and of dimension $n^2(nr+1)-nr.$ On the other hand it is defined by the $nr$ equations $\delta_0, \dots , \delta_{nr-1}$ in the principal open subset of affine space $\Spec(k[S]_{\delta_{nr}}.$ The set of indeterminates $S$ consists of $n^2(nr+1)$ indeterminates and so $\Spec(k[S])$ is affine $n^2(nr+1)$ space. The result is immediate.
\end{proof}

We now construct a map from $\mathcal{B}_{nr}(F)$ to $\Lrn.$ That is we construct a flat family. It is the map which sends a base to the lattice it generates. Hence we must construct a flat family parametrized by $\mathcal{B}_{nr}(F)$ whose fiber at a base is the lattice spanned by that base.

Write $\mathcal{O}_{nr+1}^n$ for the $n$-fold product of $\mathcal{O}/p^{nr+1}\mathcal{O}$  with itself over $\Spec{k}$. Then there is a natural map of schemes
\[
\alpha:\mathcal{O}_{nr+1}^n\times_k\mathcal{B}_{nr}(F) \to F/p^{nr+1}F\times_k\mathcal{B}_{nr}(F).
\]
For a $k$-scheme $T$ let $a_i$ be in $\Hom_k(T,\V_{nr+1})$ and let 
$(\bu _1, \dots ,\bu _n)$ be in $\Hom_k(T, \mathcal{B}_{nr}(F)).$ There is a natural map which sends $(a_1,\dots , a_n),\; a_i\in \Hom_k(T,\V_{nr+1}),\; i=1,\dots ,n$   to the point  $(\sum_{i=1}^na_i\bu _i,\bu _1,\dots ,\bu_n)$ lying in $\Hom_k(T,F/p^{nr+1}F\times_k\mathcal{B}_{nr}(F))$. This is an additive map of $T$ valued points of $\mathcal{B}_{nr}(F)$-group schemes. It is functorial in $T$ and so it corresponds, under Yoneda's lemma to a map of $\mathcal{B}_{nr}(F)$-group schemes. 

\begin{lemma} The image of $\alpha$ is the scheme $\mathbb{U}(\mathcal{B}).$\end{lemma}

\begin{proof} Consider the point set image of closed points in $\V_{nr+1}^n\times _k\mathcal{B}_{nr}(F).$ Since all the schemes under consideration are schemes of finite type over the algebraically closed field $k,$ closed points are the same thing as $k$-points. A $k$-point in the source is a $2n$-tuple $(a_1,\dots ,a_n; \bu_1,\dots ,\bu_n)$ where the $a_i$ are in $\V_{nr+1}$ and the $\bu_i$ constitute a base of discriminant $nr.$ By definition $\alpha(a_1,\dots ,a_n; \bu_1,\dots ,\bu_n)=(\sum_{i=1}^na_i\bu_i , \bu_1,\dots ,\bu_n).$ This is evidently a point in $\mathbb{U}(\mathcal{B}).$ By Lemma \ref{gpt}, since closed points are $k$-points, the map is evidently surjective. Thus $\alpha$ is a morphism of $\mathcal{B}_{nr}(F)$-group schemes from $\V_{nr+1}^n\times_k\mathcal{B}_{nr}(F)$ to $F/p^{nr+1}F\times_k\mathcal{B}_{nr}(F)$ whose image is the closed subscheme $\mathbb{U}(\mathcal{B}).$ Now the subscheme $\mathbb{U}(\mathcal{B})$ is the scheme defined by certain equations but with the reduced subscheme structure. On the other hand $\alpha$ is a map from a variety. In fact $\V_{nr+1}^n\times_k\mathcal{B}_{nr}(F)$ is the product over $k$ of a rational variety with an affine space and so it is a variety. The image of an integral scheme is integral and so reduced. It follows that the image of $\alpha$ is exactly $\mathbb{U}(\mathcal{B}).$
\end{proof}

\begin{lemma}\label{flg} The group scheme $\mathbb{U}(\mathcal{B})$ is smooth over $\mathcal{B}_{nr}(F).$
\end{lemma}
\begin{proof}First consider the fiber of $\mathbb{U}(\mathcal{B})$ over the base $U=(\bu_1, \dots ,\bu_n).$ By Lemma \ref{gpt}, the geometric points of the fiber are quantities $\sum_ia_i\bu_i$ in $F/p^{nr+1}F.$ On the other hand, base extending the surjective morphism $\alpha$ by the geometric point $U:\Spec(k)\rightarrow \mathcal{B}_{nr}(F)$ yields the surjective map $\V_{nr+1}^n\rightarrow f^{-1}(U).$ That is the fiber over $U$ is a group scheme over $k$ which is the surjective image of an integral groupscheme. Hence it is integral and connected and so smooth. Each fiber is a subgroup of $F/p^{nr+1}F$ of discriminant $nr$ and so of dimension $n(nr+1)-nr.$ That is $\mathbb{U}(\mathcal{B})$ is a group scheme of finite type over the integral scheme $\mathcal{B}_{nr}(F)$ which is itself of finite type over $k$ and its fibers are smooth and equidimensional. It follows that it is a smooth group scheme over $\mathcal{B}_{nr}(F).$
\end{proof}

 What we have shown is that $\mathbb{U}(\mathcal{B})$ is a flat scheme of lattices over $\mathcal{B}_{nr}(F).$ Since the lattices are of codimension $nr$ in $F$ this determines a morphism $\beta:\mathcal{B}_{nr}(F)\rightarrow \Lrn$ so that $\mathbb{U}(\mathcal{B})=\Ur \times_{\Lrn}\mathcal{B}_{nr}(F)$ and $\beta$ is unique up to multiplication by a unit of $\V.$

\begin{definition}\label{gen}The map $\beta:\mathcal{B}_{nr}(F)\rightarrow \Lrn$ determined by the flat family of lattices $\mathbb{U}(\mathcal{B})$ will be referred to as the generating map.
\end{definition}

\begin{remark}\label{sgs} Let $X$ be an integral separated scheme of finite type over $k$ and let $f:H\rightarrow X$ be a flat group scheme separated and of finite type over $X$ with smooth connected fibers over all points. Suppose further that for each $x\in X$ the fiber $f^{-1}(x)$ is equidimensional of dimension $N$ and that for any point $h\in H$ the equation $N+\operatorname{dim}(X)=\operatorname{dim}_h(H)$ holds. Then $H$ is smooth over $X.$
\end{remark}

This follows for example from [EGA] IV, part 4, Theorem  17.5.1 which asserts that a morphism locally of finite presentation is smooth if and only if it is flat and has smooth fibers. It will be used below.

\section{The Main Theorem}\label{Smt}

The main substance of the proof of the Main Theorem in this section is a demonstration that the generating morphism $\beta:\mathcal{B}_{nr}(F)\rightarrow \Lrn$ is a smooth morphism. We do this by embedding it as a dense open subset of a smooth group scheme over $\Lrn.$ Once this is established the main theorem follows easily.

Consider the universal family $\Ur.$ It is a closed subscheme of \linebreak $F/p^{nr+1}F\times_k \Lrn.$ It is flat by construction with smooth equidimensional fibers and $\Lrn$ is integral (\cite{me}, Theorem 4). It thus meets the conditions of Remark \ref{sgs} and is hence a smooth group scheme over $\Lrn$.  Form the $n$-fold relative fiber product $\Ur \times_{\Lrn}\Ur \times_{\Lrn} \dots \times_{\Lrn}\Ur$, and write it $\Ur(n)$.  Then $\Ur (n)$ is closed in $(F/p^{nr+1}F)^n\times_k\Lrn$ and moreover it is a smooth group scheme over $\Lrn.$

\begin{theorem}\label{brn} There is an open embedding $g$ of $\mathcal{B}_{nr}(F)$ onto a principal open subset of $\Ur (n)$.   The generating morphism 

 \[\beta:\mathcal{B}_{nr}(F)\rightarrow \Lrn \] 
is a smooth morphism which induces an isomorphism from $\mathbb{U}(\mathcal{B})$ to the fiber product $\Ur \times_{\Lrn}\mathcal{B}_{nr}(F).$ 
\end{theorem}
\begin{proof}Let $h:\Ur(n)\rightarrow (F/p^{nr+1}F)^n$ be the composite of the closed embedding $\Ur(n)\hookrightarrow (F/p^{nr+1}F)^n\times_k \Lrn$ and the projection on the first factor. Write $\pi:\Ur \rightarrow \Lrn$ for the natural map and write $\pi^n:\Ur(n)\rightarrow \Lrn$ for its Cartesian exponentiation to $\Ur(n)$.

 Both $\mathcal{B}_{nr}(F)$ and  $\Ur(n)$ are varieties of finite type over the algebraically closed field $k.$ Hence if we specify algebraic maps between $\mathcal{B}_{nr}(F)$ and some open subset of $\Ur(n)$  which are inverse to each other on geometric points they are inverse as maps of schemes.

We define functions on  $\Ur(n)$ which we call $\nu_0, \nu_1, \dots ,\nu_{nr}$ by taking $\nu_j=\delta_j\circ h$ where $\delta_j$ is the function defined by equation \ref{del}. The closed points of $\Ur(n)$ are ordered $n$-tuples of vectors with entries in $\V/p^{nr+1}\V$ together with a lattice of discriminant $nr$ in which they lie.  On the $k$-point $(L, \bu_1,\dots ,\bu_n),\; \nu_j$ is the Witt $j$'th component of the determinant of the matrix $(\bu_1, \dots ,\bu_n)$ (Recall that the $\bu_j$ are columns with entries in $\V_{nr+1}$). Since all of the vectors occurring in any point of $\Ur(n)$ lie in a lattice of discriminant $nr$, namely $L,$ the functions $\nu_j$ are all identically zero for $j<nr$. It follows that the image of $h$ lies in the vanishing set of $\delta_0,\dots ,\delta_{nr-1}\text{in}\quad (F/p^{nr+1}F)^n$, that is, the Zariski closure of $\mathcal{B}_{nr}(F)$ in $(F/p^{nr+1}F)^n$. The condition $\nu_{nr}\neq 0$ defines a principal open subscheme of $\Ur(n)$ in which all of the $n$-tuples obtained by excluding the lattice component are ordered bases of discriminant $nr.$ Call this principal open subset $Z.$ We have just shown that $h$ is a map of schemes from $Z$ to $\mathcal{B}_{nr}(F).$ It is trivially surjective.

Now we construct a map inverse to $h.$ The variety $\mathcal{B}_{nr}(F)$ is a subscheme of $(F/p^{nr+1}F)^n.$ Let $\eta:\mathcal{B}_{nr}(F)\rightarrow (F/p^{nr+1})^n$ be the natural embedding. It is a closed embedding into a principal open subset. Let $g$ be the map $(\eta, \beta)$ into $(F/p^{nr+1}F)\times_k \Lrn.$ Then the image of $g$ is evidently in $\Ur(n).$ Since the points of $\mathcal{B}_{nr}(F)$ are bases, this image is contained in $Z.$ The maps $g$ and $h$ are clearly inverse to each other on geometric points. Hence $g$ is an isomorphism from $\mathcal{B}_{nr}(F)$ onto $Z.$ Now $Z$ is a principal open subset of $\Ur(n)$ which is a smooth group scheme over $\Lrn.$ Hence $\beta,$ which is the composition of the open embedding $g$ and the smooth projection $\pi^n$ is also smooth. That is the main assertion of the theorem.

For the last assertion notice that $\beta$ is defined as the map to $\Lrn$ (Equation \ref{gen})corresponding to the flat family $\mathbb{U}(\mathcal{B}).$ In \cite{me} $\Lrn$ is defined as the scheme representing the functor ``isomorphism classes of flat families of lattices" and $\Ur$ is the universal family making it a fine moduli scheme. Thus it is a matter of definition that $\mathbb{U}(\mathcal{B})$ is isomorphic to $\Ur \times_{\Lrn}\mathcal{B}_{nr}(F).$
\end{proof}

We may now turn to the main theorem which is actually little more than a corollary to Theorem~\ref{brn}.  The argument is parallel to the arguments of \cite{kp}. The  first proof of this was given by the junior author in \cite{ast}.

\begin{theorem}[\bf{Main Theorem}]\label{M:th}The lattice variety $\Lrn$ is a normal, projective algebraic variety which is locally a complete intersection.
\end{theorem}
\begin{proof}Consider $\mathcal{B}_{nr}(F)$. There is an open subset of an affine space in which it is a complete intersection. Moreover by Theorem~\ref{brn}, there is a smooth, surjective morphism $\beta:\mathcal{B}_{nr}(F)\rightarrow \Lrn$. Hence by a theorem of L. Avramov (see \cite{lav} or the remark following Theorem~23.6 on p.~182 in \cite{mat}), $\Lrn$ is locally a complete intersection. Now by Theorem 6 on p.~113 of \cite{me}, the singular locus of $\Lrn$ is of codimension two. Consequently by the criterion of Serre, $\Lrn$ is normal.
\end{proof}

We recall the definition of a ``quotient morphism" in \cite{bo}. Designed to cover the most simple quotients over fields of positive characteristic, this definition is a parallel to the corresponding notion in \cite{git}.  A classical quotient morphism is a surjective open map $\pi:V\rightarrow W$ such that $\V_W(U)$ is the set of functions on $\pi^{-1}(U)$ constant on the fibers of $\pi$.  The following was established in \cite{ast}.
\begin{corollary}The map $\beta:\mathcal{B}_{nr}(F)\rightarrow \Lrn$ is a classical quotient morphism. Moreover the fibers of $\beta$ are $\GL(n,\V/p^{nr+1}\V)$-orbits.
\end{corollary}
\begin{proof}By Lemma 6.2 on p.95 of \cite{bo}, this will be true if $\beta$ is a surjective open separable map from an irreducible variety onto a normal variety. Now $\beta$ is a smooth surjective morphism by Theorem~\ref{brn}. Hence it is open and separable. The scheme $\Lrn$ is a normal variety by Theorem~\ref{M:th} and by the Lemma \ref{rv} $\mathcal{B}_{nr}(F)$ is reduced and irreducible. The conditions of the lemma are satisfied and so the map is a classical quotient map.

To see that the fibers of $\beta$ are orbits, let us consider the fiber over the lattice $L$.  It consists of the set of $n$-tuples of elements $(\overline{\bu}_1, \dots ,\overline{\bu}_n)\in (L/p^{nr+1}F)^n$ which can be represented as classes of the column spaces of matrices in $M_n(\V)$ with determinant of the form $\sum_{j\geq nr}\xi(\delta_j)^{p^{-j}}p^{j}$ such that $\delta_{nr} \neq 0$. Over $\V$ any two such matrices differ by multiplication by an element of $\GL(n,\V)$. This certainly implies the result.
\end{proof}
Note that $\mathcal{B}_{nr}(F)$ is not a principal bundle over $\Lrn$. The infinite dimensional version of the same construction is a principal bundle for the proalgebraic group $\GL(n,\V)$, but in passing to the finite dimensional situation of $\Ur$ viewed as a subscheme of $\Lrn \times_kF/p^{nr+1}F$ one could at best hope to use a nonconstant group scheme such as the scheme of module automophisms of $\Ur$. For a principal $G$-bundle $\pi:X\rightarrow Y$, the group of $G$-automorphisms of each fiber of $\pi$ is isomorphic to $G$. In the case at hand the automorphisms of the fiber over $L$ is the finite dimensional algebraic group $\operatorname{Aut}_{\V}(L/p^{nr+1}F)$ and this changes as the elementary divisors of $L$ change.

\section{Comparison to the Classical Case}\label{CC}
The results and methods of the previous section raise the question of whether there are greater similarities between lattice varieties and the corresponding Schubert cells of infinite Grassmann varieties. In this section as well as the next we will show that there is an intriguing analogy. More specifically, we show that many constructions are virtually identical provided one replaces linear algebra with $p$-linear algebra. We begin by recalling certain results of \cite{me}and \cite{me2}. In \cite{me2} the senior author constructed  a certain canonical sequence of compatible projective embeddings of the schemes $\Lrn$. This was done through a careful study of the Lie algebra functor on $\V$-lattices of rank $n$.  For any algebraic $k$-group $H$, write $\lie H$ for its Lie algebra. Recall that, over a field of positive characteristic, the Lie algebra of an algebraic group is a restricted Lie algebra, that is, a Lie algebra with a formal $p$'th power satisfying the Jacobson identity. The general case is discussed below in Section~\ref{RLA} but the Lie algebra of a commutative group scheme is just a vector space with a $p$-linear endomorphism. The $p$-linear endomorphism is nothing but the $p$'th power operation on left invariant vector fields.

Write $\Or$ for the set of formal series $\sum_{i=0}^{\infty}a_i \theta^i,\; a_i\in k$. These formal series form a ring subject to the multiplication $\theta a=a^p\theta.$ We call this the complete Ore ring on $k.$ It is not a $k$-algebra; it is easy to see that its center is just $\mathbb{F}_p.$ We shall write it unadorned to signify the left module over itself and we shall write $\Or '$ for $\Or$ viewed as a right $\Or$-module. In \cite{me2} or in \cite{cl}, it is shown that the Lie algebra of the additive proalgebraic group $\V$ is just $\Or$ and that the Lie algebra of $\V/p^s\V$ is $\Or/\theta^s\Or.$ Further, since the Lie algebra of a product is the product of the Lie algebras, $\lie F=\Or^n.$ Write $\mathfrak{F}$ for $\Or^n$, particularly when it is thought of as $\lie{F}.$ If $M\subseteq \Or^n$ is a left $\Or$-submodule of finite vector space codimension $q$ we shall call it an $\Or$-lattice of codimension $q$ in $\Or^n.$

It is known that all ideals in $\Or$ are principal and that since $k$ is perfect, they are two-sided and Noetherianness arguments apply. Moreover every finite left (respectively right) $\Or$-module is a direct sum of cyclics and that each is of the form $\Or/\Or\theta^s$ (respectively $\Or/\theta^s\Or$) if it is a torsion module (\cite{spr} pp. 49-51 or \cite{ffa}, 1.7 for a fuller treatment).  If $\mathfrak{N}$ is an $\Or$-lattice in $\F,$ just as is the case over a discrete valuation ring, there is a basis for $\F, \{\fu_1, \dots ,\fu_n\}$ and integers $s_1, \dots, s_n,$ so that $\{\theta^{s_1}\fu_1, \dots , \theta^{s_n}\fu_n\}$ is a basis for $\mathfrak{N}$ and $s_1+\dots +s_n=\operatorname{codim}\mathfrak{N}$.  The $s_i$ are uniquely determined.

The $n\times n$ matrices with entries in $\Or$ are a ring (they are not a $k$-algebra) and the invertible elements in this ring form a group which we will denote $\GL(n, \Or)$.  We caution the reader that this group has a finite center, to wit, scalar matrices with entries in $\mathbb{F}_p^*$. Notice that if $\Or$ is thought of as a left $\Or$ module then $\Hom_{\Or}(\Or, \Or)= \Or$ where the last $\Or$ acts as right multiplications. 

Write the ring of matrices with $\Or$-entries $M_n(\Or)$. Then recall that \linebreak $\Hom_{\Lambda}(\coprod_iM_i,\coprod_jN_j)=\coprod_{i,j}\Hom_{\Lambda}(M_i,N_j)$ and these homomorphisms can be thought of as matrices with $i,j$ entry in $\Hom_{\Lambda}(M_i,N_j)$. Then thinking of the elements of $\F$ as rows let the matrix $A=(a_{i,j})$ operate on the row, $B=(b_i)$ by the formula, $BA=(c_j),\: c_j=\sum_i b_ia_{i,j}$. This makes $\F$ a right $M_n(\Or)$-module and it induces an isomorphism $\Hom_{\Or}(\F,\F)\simeq M_n(\Or)$. It is clear that $\GL(n,\Or)$ operates on $F$ on the right. (Note that transpose is not an antihomomorphism.) Since $\GL(n,\Or)$ consists entirely of module automorphisms each element $A$ induces isomorphisms $\F/\mathfrak{N}\simeq \F/\mathfrak{N}A$ for each $\Or$-lattice $\mathfrak{N}$. That is, the elements of $\GL(n,\Or)$ preserve codimension. Now the lattices of codimension $q$ all contain $\F\theta^q$. It is also true that $\theta^q\Or=\Or\theta^q$ for perfect $k$ and so $\GL(n,\Or/\theta^q\Or)$ acts algebraically on $\F/\F\theta^q$ for each $q$.

Elements of $M_n(\Or)$ can be written as sums $A=\sum_i A_i\theta^i$ where $A_i$ is an $n\times n$ matrix with entries in $k$. Such an element is invertible if and only if $A_0$ is and so there is a polynomial function $c_0(A)=\operatorname{det}A_0$, so that $A$ is invertible if and only if $c_0(A)\neq 0$.  The same remarks and the same function apply to $M_n(\Or/\theta^q\Or)$.

The fact that each finite $\Or$-module is uniquely a direct sum of cyclic modules means that all bases of $\F$ have just $n$ elements and that for any two bases, $\{\fu_1, \dots, \fu_n\}$ and $\{\mathfrak{v}_1, \dots ,\mathfrak{v}_n\}$, the map $\alpha(\sum_ia_i\fu_i)=\sum_ia_i\mathfrak{v}_i$ is an $\Or$-automorphism whence there is a matrix $A\in M_n(\Or)$ such that $\fu_iA=\mathfrak{v}_i$.  If $\mathfrak{L}$ and $\mathfrak{N}$ are two $\Or$-lattices in $\F$ such that $\F/\mathfrak{L}\simeq \F/\mathfrak{N}$, then there are integers $s_1,\dots ,s_n$ and bases for $\F$, $\{\fu_1, \dots ,\fu_n\}$ and $\{ \mathfrak{v}_1,\dots ,\mathfrak{v}_n\}$, so that $\mathfrak{L}$ has a basis $\{\theta^{s_i}\fu_i:\, i=1,\dots,n\}$ and $\mathfrak{N}$ has a basis $\{\theta^{s_i}\mathfrak{v}_i:\, i=1,\dots,n\}$. This means that there is an element $A$ of $\GL(n,\Or)$ so that $\mathfrak{L}A=\mathfrak{N}.$ This means that the lattices $\mathfrak{L}$ such that $\F/\mathfrak{L}$ lie in a single isomorphism class form a single orbit under the action of $\GL(n,\Or)$ and that each such orbit is uniquely characterized by the sequence $(s_1, \dots, s_n)$ arranged in nonincreasing order.

It is easy to prove that the restricted Lie subalgebras of a restricted commutative nilpotent Lie algebra $T$ of a fixed codimension form a projective scheme. (See \S 5 of   \cite{me2} They are simply subspaces fixed under the $p$'th power operation.) It is clear that the restricted Lie subalgebras of $\F$ of codimension $q$ can be identified with the restricted Lie subalgebras of $\F/\theta^q\F$ of codimension $q,$ that is to say with its $\Or$ submodules. To prove the next results it is necessary to review the computations of the tangent spaces to a point of $\Lrn$ as well as the tangent to a certain scheme of Lie subalgebras.  

It is a consequence of Proposition 4 of \cite{me2} and of the proof of Theorem 1 of the same paper that the map which sends a lattice of codimension $q$ to its Lie algebra is an algebraic morphism from the scheme of lattices in $F$ of codimension $q$ to the scheme of restricted Lie subalgebras of $\F$ of codimension $q.$ We wish to describe the differential of that map.  For a finite dimensional restricted Lie algebra $T$, write $\mathcal{L}_q(T)$ for the scheme of restricted Lie subalgebras of $T$ of codimension $q.$

First we describe the tangent space to a lattice in $\Lrn.$ The relevant results were first proven on page 107 and following in \cite{me} but we have included an appendix with proofs specifically adapted to the needs of this paper. The proofs in the appendix are complete. Let $F_{nr}=F/p^{nr}F$ and let $M\subset F_{nr}$ denote a lattice of corank $nr$ contained in it. A lattice is nothing more than a $\V$-stable smooth subgroup of the smooth $k$-groupscheme $F$ of finite corank. Since every lattice of corank $nr$ contains $p^{nr}F$ lattices of corank $nr$ correspond to $\V$-submodules of $F_{nr}$ of corank $nr$and henceforth we shall view these objects interchangeably. Let $A$ be the coordinate ring of $F_{nr}$ and let $I_M$ be the ideal defining $M$ in $A.$ The tangent space to $M$ in $\Lrn$ will be the set of $\V$-stable $k[\epsilon]$ subgroup schemes of $F_{nr}(\epsilon)=\Spec (k[\epsilon])\times_kF_{nr}$ which admit $M$ as the fiber over the closed point in $\Spec (k[\epsilon]).$ These correspond to Hopf ideals in $A[\epsilon]$ reducing to $I_M$ modulo $\epsilon$ which are also stable with respect to the coaction describing scalar multiplication by $\V$ on $M$. Such ideals are classified in the appendix (Theorem A.6, Theorem A.7). They correspond to homomorphisms of schemes of $\V$-modules from $M/pM$ to $\mathfrak{F}/\mathfrak{M}$ where $\mathfrak{F}$ is the restricted Lie algebra of $F_{nr}$ and $\mathfrak{M}$ is the Lie algebra of $M$.

 By general deformation theory,the ideal defining such a deformation is determined by a map $\delta \in \Hom_{A/I_M}(I_M/I_M^2,A/I_M)\}$. Ideals lifting $I_M$ correspond bijectively to these maps by letting $\tilde{I}_M(\delta)$be the ideal in $A[\epsilon]$ generated by $\{ a+b\epsilon:a\in I_M,\quad b+I_M=\delta(a+I_M^2)\}$ and it defines a $k[\epsilon]$ subgroup scheme of $F_{nr}(\epsilon)$ if it is contained in the identity ideal, it is stable under $s_F$ (the coinverse) and if $\alpha(\tilde{I}_M(\delta)\subseteq \tilde{I}_M(\delta)\otimes A[\epsilon]+A[\epsilon]\otimes \tilde{I}_M(\delta).$ It will define an $\V$-submodule if $\mu( \delta(I/I^2))\subseteq I/I^2\otimes_{\ke}\ke$. The map alpha is the coaddition on $F$ and $\mu$ is the coaction $\mu: k[F]\rightarrow k[F]\otimes k[\V]$ corresponding to multiplication by elements of $\V$ on $F$.

A brief digression on the conormal bundle to a smooth subgroup-scheme is in order. Let $\alpha:A\rightarrow A\otimes_kA$ be the coaddition on $F_{nr}.$ Then $\alpha(I_M)\subseteq (I_M\otimes A+A\otimes I_M).$ Then $(I_M\otimes A+A\otimes I_M)/(I_M\otimes A+A\otimes I_M)^2=I_M/I_M^2\otimes A/I_M \coprod A/I_M\otimes I_M/I_M^2.$ Now $\alpha$ induces an action of $M$ on its normal bundle in $F_{nr}$ and it can be described by the map induced by $\alpha$ first from $I_M/I_M^2$ to $(I_M\otimes A+A\otimes I_M)/(I_M\otimes A+A\otimes I_M)^2$ and then by projection onto $I_M/I_M^2\otimes_kA/I_M.$ This is a coaction making the coherent sheaf associated to $I_M/I_M^2$ into a homogeneous bundle.The second projection onto $A/I_M\otimes I_M/I_M^2$ is the same map twisted because of the commutativity of $F$ and $M.$ Denote this map $\alpha_N.$ It is a map into a direct sum with first component the right coaction on the normal bundle and with second component the left coaction. In particular we may write the equation:
\begin{equation}\label{nma}
\alpha^N=\alpha_R\oplus \alpha_L
\end{equation}

For some $u\in I_M$ consider $u+v\epsilon$ where $v$ is congruent to $\delta(u)$ modulo $I_M.$ The coaddition carries this to $\alpha(u)+\alpha(v)\epsilon.$ Consider $\tilde{I}_M(\delta)\otimes A[\epsilon)+A[\epsilon]\otimes \tilde{I}_M(\delta).$ This is an ideal in $A[\epsilon]\otimes A[\epsilon]$  lying over $I_M\otimes A+A\otimes I_M$ and so it is determined by a map $\Delta:(I_M\otimes A+A\otimes I_M)/(I_M\otimes A+A\otimes I_M)^2\rightarrow A/I_M\otimes A/I_M$ just as $\tilde{I}_M(\delta)$ was. It is not difficult to see that the map would have to be $\delta \otimes \operatorname{id}\oplus \operatorname{id}\otimes \delta.$ Hence $\alpha(u)+\alpha(v)\epsilon \in \tilde{I}_M(\delta)\otimes A[\epsilon]+A[\epsilon]\otimes \tilde{I}_M(\delta)$ if and only if $\alpha(v)$ is congruent to $(\delta \otimes \operatorname{id}\oplus \operatorname{id}\otimes \delta)(\alpha(v)).$ (The congruence is modulo the square of the ideal of $M\times_kM$ in $F_{nr}\times F_{nr}$.) 

There is, in the case of a lattice, yet another coaction. The scalar multiplication of $\V_r$ on $F_{nr}$ namely $\V_r\times_kF_{nr}\rightarrow F_{nr}$ is given by a coaction $\mu:k[F_{nr}]\rightarrow k[F_{nr}]\otimes k[\V_r]$. If the subgroup $M\subseteq F_{nr}$ is a lattice then the restriction map $k[F_{nr}]\rightarrow k[M]$ must induce a corresponding coaction on $k[M]$. This will be so if and only if $\mu(I_M)\subseteq I_M\otimes k[\V_r]$. If this is so then $\mu(I_M^2)\subseteq I_M^2\otimes k[\V_r]$ and so there is an induced conormal coaction: 

\begin{equation}
\mu^N:I_M/I_M^2\rightarrow I_M/I_M^2\otimes k[\V_r]
\end{equation}

 Now the ideal $\tilde{I}_M(\delta)$ defines a $\ke$-scheme of lattices if and only if $\mu(\tilde{I}_M(\delta)\subseteq \tilde{I}_M(\delta)\otimes \ke[F_{\epsilon}]$. Now $\tilde{I}_M(\delta)\otimes \ke[F_{\epsilon}]$ corresponds to an infinitesimal deformation of $I\otimes k[F]$ and by simpe flatness it is easy to see that the classifying map of $\tilde{I}_M(\delta)\otimes \ke[F_{\epsilon}]$ is $\delta \otimes \id$. A complete discussion is to be found in the appendix. We summarize:   

\begin{lemma}\label{tlat}Let $M\subseteq F_{nr}$ be a lattice of codimension $nr$ in $F_{nr}$ with defining ideal $I_M.$ Then the $k[\epsilon]$ sublattices of $F_{nr}(\epsilon)$ of codimension $nr$ are all uniquely determined by maps $\delta \in \Hom_{A/I_M}(I_M/I_M^2, A/I_M)$ such that:
\begin{enumerate}

\item  $\delta$ is $s$ stable where $s$ is the coinverse.
\item  $\alpha \circ \delta= (\delta \otimes \id + \id \otimes \delta)\circ \alpha_N$
\item $\mu \circ \delta =(\delta \otimes \id)\circ \mu^N$
\end{enumerate}
This bijective correspondence assigns to the map $\delta$ the ideal in $A[\epsilon]$ generated by $\{u+v\epsilon :u\in I_M,\; v\equiv \delta(u) \operatorname{mod}I_M^2\}.$
\end{lemma}

Such maps have a natural interpretation. Since $I/I^2$ is a homogeneous bundle on the group $M$ it can be written $I/I^2=\mathfrak{n}_{M/F}\otimes_kk[M]$ where $\mathfrak{n}_{M/F}$ denotes the $k$-vector space of invariant sections in $I/I^2$. Then $\mathfrak{n}_{M/F}$ can be canonically identified with the linear dual of $\mathfrak{F}/\mathfrak{M}$ the Lie algebra of the algebraic $\V$-module $F/M$. Moreover $\mathfrak{n}_{M/F}$ is preserved by the coactions $\alpha^N$ and $\mu^N$. That is for $u\in \mathfrak{n}_{M/F},\; \alpha^N(u)=u\otimes 1+1\otimes u$ and $\mu^N(\mathfrak{n}_{M/F}\subseteq \mathfrak{n}_{M/F}\otimes k[\V]$. This is worked out in detail in Proposition A.5 and Theorem A.6 of the appendix. Let $k[\mathfrak{n}_{M/F}]$ denote the symmetric algebra on the invariant sections. It can be thought of as the coordinate ring of $\mathfrak{F}/\mathfrak{M}$ under addition and then (2) and (3) of Lemma \ref{tlat} are equivalent to the statement that $\delta$ induces a morphism of schemes of modules from $M$ to $\mathfrak{F}/\mathfrak{M}$. Theorem A.6 of the appendix shows that this induces an isomorphism between the tangent space to the lattice $M$ in the scheme of lattices in $F$ of corank $nr$ and $\Hom_{\V}(M/pM, \mathfrak{F}/\mathfrak{M})$.

 We will write $\mathcal{T}_F(M)$ for the tangent space to the lattice $M$ viewed as a point in $\Lrn$ which is just the scheme of lattices of codimension $nr$ in $F_{nr}.$ The following and its proof can be found in in \cite[Lemmas 18, 19 pages 111--113]{me} but a more complete version specifically adopted to the needs of this paper is to be found in the appendix (Theorem A.7). 

\begin{proposition}\label{tanl}Let $M$ be a lattice of codimension $nr$ in $F_{nr}.$  Let $\mathfrak{n}=(\mathfrak{F}/\mathfrak{M})^*$ denote the dual of the quotient of the Lie algebra of $F_{nr}$ by the Lie algebra of $M.$ Then $\mathcal{T}_F(M)$ is canonically isomorphic to $Hom_{\V}(M/pM, \mathfrak{f}/\mathfrak{m})$ where the morphisms are in the category of schemes of modules over schemes of rings.
\end{proposition}

\begin{proof}This is just Theorem A.7 of the appendix. 
\end{proof} 

We now turn to the corresponding calculation for restricted commutative nil Lie algebras.

It is convenient to call the lowest power of $\theta$ occurring in an element of $\Or$ its order. For $a\in \Or$ we will write it $\operatorname{ord}(a)$. Recall that $\mathcal{L}_q(T)$ denotes the scheme of restricted Lie subalgebras of codimension $q$ in the restricted Lie algebra $T.$ For the balance of this discussion write $\mathcal{L}_{n,r}$ for $\mathcal{L}_{nr}(\mathfrak{F}/\theta^{nr}\mathfrak{F})$ and write $\F_r$ for $\mathfrak{F}/\theta^{nr}\mathfrak{F}.$ Let $\mathfrak{m}$ be a restricted Lie subalgebra of $\F_r$ thought of as a point of $\lienr.$ The tangent space to $\lienr$ at $\mathfrak{m}$ is the set of $k[\epsilon]$ restricted Lie subalgebras of $\F_r\otimes_kk[\epsilon]=\F_r(\epsilon)$ which are free $k[\epsilon]$ submodules of rank $n(n-1)r$ and which reduce to $\mathfrak{m}$ modulo $\epsilon.$ 
\begin{lemma}\label{tlie}Let $\mathfrak{m}$ be a restricted Lie subalgebra of $\F_r$ of codimension $nr.$ Then the tangent space to $\mathcal{L}_{n,r}$ at $\mathfrak{m}$ is isomorphic to $$\Hom_{\Or}(\mathfrak{m}/\theta \mathfrak{m}, \F_r/\mathfrak{m}).$$  

\end{lemma}

\begin{proof} Let $\tilde{\mathfrak{m}}$ be a $k[\epsilon]$ subalgebra of $\F_r(\epsilon)$ lifting $\mathfrak{m}.$ Then there is a $k$-vector space section to the map $\tilde{\mathfrak{m}}\mapsto \tilde{\mathfrak{m}}/\epsilon \tilde{\mathfrak{m}}=\mathfrak{m}.$ Choose one and call it $\tilde{\sigma}.$ Then there is a map $\sigma_1 \mathfrak{m}\mapsto \F_r$ so that $\tilde{\sigma}(x)=x+\sigma_1(x)\epsilon.$ Then $k[\epsilon]$-freeness of $\tilde{\mathfrak{m}}$ means that $\tilde{\sigma}$ applied to a $k$-basis will yeild a $k[\epsilon]$-basis. In particular $\tilde{\sigma}(\mathfrak{m})$ and $\epsilon \mathfrak{m}$ span $\tilde{\mathfrak{m}}$ over $k$ and the class of $\sigma_1$ modulo $\mathfrak{m}$ is independent of the choice of $\tilde{\sigma}.$ Write $\sigma$ for the composition of $\sigma_1$ with the surjection $\F_r\mapsto \F_r/\mathfrak{m}.$ Notice that since $\theta$ is $p$-semilinear, $\theta(\epsilon y)=(\epsilon)^p\theta  y=0.$ It follows that $\sigma$ vanishes on $\theta \mathfrak{m}.$ Conversely any $k$-linear map, $\sigma:\mathfrak{m}/\theta \mathfrak{m}\rightarrow \F_r/\mathfrak{m}$ may be used to construct a section $\tilde{\sigma}$ and a restricted Lie subalgebra of $\F_r(\epsilon).$    
\end{proof}

Now by Proposition 5.3 The tangent space to the lattice $M$ in $F_{nr}$ is the set $\Hom_{\V}(M/pM, \mathfrak{f}/\mathfrak{m})$. Let $J$ be the ideal defining $M$. In the appendix it is proven that this is the same as the set of $k[M]$-maps $\delta: J/J^2\rightarrow k[M]$ and that these are in turn uniquely determined by their restrictions to $\mathfrak{n}_{F/M}$ and consequently $\mathcal{T}_F(M)$ can be thought of as the vector space of maps $\delta:\mathfrak{n}_{F/M}\rightarrow k[M]$ satisfying (2) and (3) of Lemma 5.1 (A.9 and A.10 of the appendix). Write $\mathcal{T}^0$ for the space of maps in $\Hom_k(\mathfrak{n}_{F/M}, k[L])$ satisfying A.9 and A.10. These maps will be referred to as classifying maps and they will be said to classify the deformation class corresponding to them.

Suppose that $G=\Spec (A)$ is a smooth connected group scheme flat over $\Spec (R).$ Let $e:A\mapsto R$ be the morphism of rings corresponding to the identity section and let $\fm_G=\ker (e).$ Let $H$ be a closed subgroupscheme smooth over $\Spec (R)$ and suppose that $J$ is the ideal defining it. Then the Lie algebra of $G$ is the linear dual of $\fm_G/\fm_G^2.$ If $\fm_H$ is the ideal defining the identity section in $H$ then the Lie algebra of $H$ is the dual of $\fm_H/\fm_H^2.$ On the other hand $\fm_H=\fm_G/J$ and so $\fm_H/\fm_H^2=\fm_G/(\fm_G^2+J).$ There is an exact sequence:
\begin{equation}\label{neq}
0\mapsto J/J\cap \fm_G^2 \mapsto \fm_G/\fm_G^2 \mapsto \fm_H/(\fm_H^2)\mapsto 0. 
\end{equation}

Hence the element $\eta \in  (\fm_G/\fm_G^2)^*$ is in the Lie algebra of $H$ if and only if $\eta(J/J\cap \fm_G^2)=0.$ Finally we may take the linear dual (over $R$) of the exact sequence \ref{neq}. The duals of the last two terms are isomorphic to the Lie algebras of $G$ and $H$ respectively and so we see that $J/J\cap \fm_G^2$ is canonically isomorphic to $(\mathfrak{g}/\mathfrak{h})^*$ the linear dual of the quotient of the Lie algebra of $G$ by the Lie algebra of $H.$ That is $J/J\cap \fm_G^2$ is canonically isomorphic to the space of invariant sections in the normal bundle $J/J^2.$ This is the space that was written $\mathfrak{n}$ in a previous discussion (see e.g. Proposition \ref{tanl}). When $H$ is normal and connected, as is the case for a lattice in $F$,then the coordinate ring of $G/H$ is, via composition with the projection a subalgebra of $k[G]$ and so if $\mathfrak{m}_0$ is the ideal defining the identity in $G/H$ there is a natural inclusion $\mathfrak{m}_0/\mathfrak{m}_0^2\subseteq J/J^2$ and in fact $\mathfrak{m}_0$ generates $J$. This inclusion identifies $\mathfrak{m}_0/\mathfrak{m}_0^2$ with the invariant sections in $J/J^2$. Thus  $\mathfrak{m}_0/\mathfrak{m}_0^2$ is naturally isomorphic to the invariant sections in $J/J^2$ and also to $(\mathfrak{g}/\mathfrak{h})^*$ the linear dual of the Lie algebra of $G/H$. We will henceforth routinely identify these three vector spaces.

We now compute the differential of the algebraic morphism corresponding to the Lie algebra functor. For the definition of the classifying map of a deformation see Proposition A.5 of the appendix. It is a map from $J/J^2$ to $k[H]$. Suppose $\Delta \in \mathfrak{h}$ the Lie algebra of $H$. Regard it as a tangent vector at the origin. Then consider the composition $\Delta \circ \delta.$ This is a map from $J/J^2$ to $k$ and it is uniquely determined by its restriction to $\mathfrak{m}_0/\mathfrak{m}_0^2\subseteq J/J^2$. Hence $\Delta \circ \delta$ may be viewed as an element of $(\mathfrak{m}_0/\mathfrak{m}_0^2)^*$. This however is the Lie algebra of $G/H$. Hence the assignment $\Delta \mapsto \Delta \circ \delta $ sends the $\Delta \in \mathfrak{h}$ to an element of $\mathfrak{g}/\mathfrak{h}$. We apply this to the case $G=F_{nr},\; H=M$ where $M$ is a lattice of corank $nr$ in $F_{nr}$. In this case write $\mathfrak{n}_{F/M}$ for the space of invariant sections in $J/J^2$.

\begin{lemma}Let $M\subseteq F_{nr}$ be a lattice of codimension $nr$. Suppose that $J$ is the ideal defining $M$. Let $\tilde{J}$ be an ideal determining $\widetilde{M}$ the infinitesimal deformation of the lattice $M$ determined by the classifying map $\delta:J/J^2\rightarrow k[M]$. Then if $\Delta \circ \delta=0$ for each $\Delta \in \mathfrak{m}$ the Lie algebra of $M$, the infinitesimal deformation determined by the ideal $\tilde{J}$ is the trivial one. 
\end{lemma}
\begin{proof}To say that $\Delta \circ\delta=0$ for each element of $\mathfrak{m}$ means that $\delta(\mathfrak{n}_{F/M}\subseteq \mathfrak{m}_0^2$ where $\mathfrak{m}_0$ is the ideal defining $0$ in $M$. On the other hand the $\delta(\mathfrak{n}_{F/M})$ consists of additive $\V$-module maps from $M$ to $k$. The set of such maps which lie in $\mathfrak{m}_0^2$ is the $0$ vector space. But then by the proof of Proposition A.1, this means that $\widetilde{J}$ is generated by $J$ in $\ke[\widetilde{F}_{nr}]$. That is it is the trivial deformation. 
\end{proof}

\begin{proposition}Let $\delta \in \mathcal{T}^0$ be the classifying map of the lattice deformation $\widetilde{M}$ defined by the ideal $\widetilde{J}_{\delta}$. (See Proposition A.1, the appendix.) Then the $k[\epsilon]$-Lie algebra of $\widetilde{M}$ is a $k[\epsilon]$-free restricted Lie subalgebra of $\F_r$ of codimension $nr$ lying over $\mathfrak{m}$ the Lie algebra of $M$ and corresponding under Lemma 5.4 to the map $d\delta:\mathfrak{m}\rightarrow \F_r/\mathfrak{m}$ given by the formula $d\delta (\Delta)=-  \Delta \circ \delta$. Furthermore the assignment sending $\delta$ to $d\delta$ is injective. 
\end{proposition}

\begin{proof} Suppose that $\delta: J/J^2 \mapsto k[F_{nr}]/J$ is the classifying map determining $\widetilde{M}$. Then $\delta$ satisfies equations (2) and (3) of Lemma 5.1. An element of the Lie algebra of $\widetilde{F}_{nr}$ is a derivation of $\ke [\widetilde{F}_{nr}]$ in $\ke$. A standard computation shows that this is determined by a derivation of $k[F_{nr}]$ in $k[\epsilon]$. Such a derivation is a map $\widetilde{\Delta}(a)=\Delta_1(a)+\Delta_2(a)\epsilon$ where $\Delta_1$ and $\Delta_2$ are derivations of $k[F_{nr}]$ in $k$ and $k[\epsilon]$ is understood as the fiber of the structure sheaf of $\widetilde{F}_{  nr}$ at $e$. Then $\widetilde{\Delta}$ will be in the tangent space to $\widetilde{M}$ only if $\widetilde{\Delta}(\tilde{J})=0$. Now $\widetilde{\Delta}(a+b\epsilon)=\Delta_1(a)+(\Delta_2(a)+\delta_1(b))\epsilon$. Now $\Delta_1(J)=0$ because $\Delta_1$ is in the Lie algebra of $M$. The elements $a+b\epsilon$ with $a\in J \text{and}\; b\equiv \delta(a) \operatorname{mod}J$ generate $\tilde{J}$. Hence $\widetilde{\Delta}(a+b\epsilon)$ must be zero for each $a+b\epsilon$ such that $b\equiv \delta(a) \operatorname{mod}J$. This means that $\delta_1(a)=0$ for all $a\in J$ and that $\Delta_2(a)+\Delta_1(b)=0$. Now if $b$ is replaced by $b+x$ with $x\in J$ then $\Delta_1(b+x)=\Delta_1(b)+\Delta_1(x)=\Delta_1(b)$ since $\Delta_1$ is in the Lie algebra of $M$. Hence $\Delta_1(b)$ depends only on the residue class of $b$ modulo $J$, that is to say on $\delta(a)$ if $a+b\epsilon\in \tilde{J}$. That is $\Delta_1+\Delta_2\epsilon$ vanishes on $\tilde{J}$ only if $\Delta_2(a)=-\Delta_1(\delta(a))$. That is to say $\Delta_1+\Delta_2\epsilon$ is in the Lie algebra of $\widetilde{M}$ only if $\Delta_2=-\Delta_1\circ \delta$. Write $\sigma(\Delta_1)=d\delta(\Delta_1)=\Delta \circ \delta$. Then the Lie subalgebra of $\F_r(\epsilon)$ corresponding to $\sigma$ under Lemma 5.3 is just precisely the same subalgebra of $\F_r(\epsilon)$. This means that $d:\mathcal{T}^0\rightarrow \Hom_{\Or}(\mathfrak{m}, \F_r/\mathfrak{m})$ is the map of tangent spaces realizing the map of tangent spaces arising from the Lie algebra functor from $\Lrn$ to $\mathfrak{L}_{n,r}$. Two points remain. The first is whether $d\delta$ vanishes on $\theta \mathfrak{m}$. The second is the issue of injectivity. The latter is just Lemma 5.4. Hence all we must show is that $d\delta (\theta \mathfrak{m})=(0)$. This just follows from the restricted Lie algebra structure. Suppose $m\in \mathfrak{m}$. Suppose that $m+n\epsilon$ lies in $\widetilde{m}$a $\ke$ Lie algebra lying above $\mathfrak{m}$. Then $\theta(m+n\epsilon)$ must lie in $\mathfrak{m}$ as well. But $\theta(m+n\epsilon)=\theta(m)+\theta(n\epsilon)=\theta{m}+\epsilon^p\theta(n)=m$. Thus $\theta m\in\widetilde{m}$ for each $m\in \mathfrak{m}$. But this means that the map $\sigma$ of Lemma 5.3 vanishes on $\theta \mathfrak{m}$.     

\end{proof}

\begin{lemma}\label{orb}Let $(s_1, \dots ,s_n)$ be a nonincreasing sequence of nonnegative integers. Let $\mathfrak{L}$ be the $\Or$-lattice spanned by the vectors
\[\{(\theta^{s_1}, 0, \dots, 0),(0, \theta^{s_2},0, \dots ,0),\dots ,(0,\dots ,0,\theta^{s_n})\}.\] Then $A=(a_{i,j})\in \GL(n,\Or)$ is a matrix such that $\mathfrak{L}A=\mathfrak{L}$ if and only if $\operatorname{ord}(a_{i,j})\geq s_j-s_i$ for each pair $i,j$ such that $s_j>s_i.$ 
\end{lemma}
\begin{proof} Write the basis in the statement $\{\theta^{s_1}e_1, \theta^{s_2}e_2, \dots ,\theta^{s_n}e_n\}$. Then $\mathfrak{L}A$ has basis $\{ \theta^{s_i}e_i A:\, i=1,\dots,n\}$ and this must be     equal to $C\Theta$ where  $C\in M_n(\Or)$ and $\Theta$ is the diagonal matrix with entries $\theta^{s_i}.$ One obtains $\theta^{s_i}a_{i,j}= c_{i,j}\theta^{s_j}$. Taking orders one obtains	           $\operatorname{ord}(a_{i,j})=\operatorname{ord}(c_{i,j})+s_j-s_i.$ Since the $c_{i,j}$ can be arbitrary elements of nonnegative order,this is only a nonvacuous condition when $j>i$ and the difference is positive and it is clearly necessary and sufficient when both matrices are assumed to be in $\GL(n,\Or).$
\end{proof}
The action of $\GL(n,\Or)$ on $\F$ is linear and so it reduces to an algebraic action of the finite dimensional algebraic group $\GL(n,\Or/\theta^{nr}\Or)$ on the vector space $\F/\theta^{nr}\F.$ The $\Or$-lattices of codimension $nr$ in $\F$ may be regarded as the $\theta$-fixed subspaces of $\F/\theta^{nr}\F$ of codimension $nr.$ By our observations concerning the modules $\F/\mathfrak{L},$ each of these orbits is the orbit of a lattice, which we will write $\mathfrak{L}(r_1, \dots ,r_n),$ that is spanned by $\theta^{r_1}e_1,\dots ,\theta^{r_n}e_n$ where $r_1\geq r_2 \geq \dots \geq r_n$ and $\sum_i r_i=nr.$ Now the lemma above means that we may compute the dimension of $\GL(n,\Or/\theta^{nr}\Or)$-orbits exactly as in the case of lattice varieties or infinite Grassmannians. Hence without further proof we may state the:
\begin{proposition}\label{dim}
Let $r_1,\dots ,r_n$ be a nonincreasing sequence of natural numbers such that $\sum_i r_i=nr.$ The $\GL(n,\Or/\theta^{nr}\Or)$-stabilizer of $\mathfrak{L}(r_1, \dots ,r_n)$ in $\mathcal{L}_{nr}(\F/\theta^{nr}\F)$ is of dimension $n^3r-\sum_{i<j}(r_i-r_j)$ and the dimension of its $\GL(n,\Or/\theta^{nr}\Or)$-orbit is $\sum_{i<j}(r_i-r_j)$.
\end{proposition}
Recall that in \cite{me2} the morphism from $\Lrn$ to $\mathcal{L}_{nr}(\F/\theta^{nr}\F)$ which sends a lattice to its Lie algebra is called the canonical morphism. Theorem 1 on p.88 of \cite{me2} is equivalent to the fact that it is an algebraic map which is injective on points and tangent vectors. Write $\lambda_r$ for this map. Notice also that scalar matrices with entries in $\V^*$ act trivially on lattices and so the $\SL(n,\V)$ orbits are the same as $\GL(n,\V)$ orbits in $\Lrn.$

We propose to show that $\lambda_r$ induces an isomorphism of tangent spaces. We will do it by computing the tangent space to $\mathcal{L}(nr)$ and then computing the differential of $\lambda_r.$ First we will examine $\mathcal{L}_{nr}$ more precisely. We have seen that a restricted Lie algebra is nothng more than an $\Or$ module and that the restricted commutative Lie subalgebras of the restricted Lie algebra $\mathfrak{F}/\theta^{nr}\mathfrak{F}$ of codimension $nr$ are a projetive algebraic variety. Write $\mathfrak{F}_r$ for $\mathfrak{F}/\theta^{nr}\mathfrak{F}.$ Let us compute the tangent space to the subLie algebra $\mathfrak{M}$ of codimension $nr.$ The functor represented by $\mathcal{L}_{nr}$ assigns to each commutative $k$-algebra $B$ the set of commutative $B$ subLie algebras of $B\otimes _k\mathfrak{F}_r$ which are locally free $B$-submodules of rank $n(n-1)r$ over $\Or.$

Let $\mathfrak{M}$ be such an $\Or$-submodule of $\mathfrak{F}_r.$ That is it is of codimension $nr$ in $\mathfrak{F}_r.$ The tangent space to $\mathfrak{M}$ in $\mathcal{L}_{nr}(\mathfrak{F})$ consists of free rank $n(n-1)r \quad k[\epsilon]$ submodules of $(\mathfrak{F}/\theta^{nr}\mathfrak{F})\otimes_kk[\epsilon]$ which are $\Or$-submodules and which reduce to $\mathfrak{M}$ modulo $\epsilon \mathfrak/\theta^{nr}\mathfrak{F}.$ Let $\tilde{\mathfrak{M}}$ be such a module.

\begin{theorem}
The canonical morphism from $\Lrn$ to $\mathcal{L}_{nr}(\F/\theta^{nr}\F)$ is an isomorphism. The algebraic action of $\SL(n,\V/p^{nr}\V)$ on $F/p^{nr}F$ defines an algebraic morphism, $$\phi_r: \SL(n,\V/p^{nr}\V)\rightarrow \GL(n,\Or/\theta^{nr}\Or)$$ so that the canonical morphism is equivariant.
\end{theorem}
\begin{proof}
The canonical morphism, being a natural transformation of \allowbreak functors, by Yoneda's lemma yields an algebraic morphism. Now the group $\SL(n,\V/p^{nr}\V)$ acts algebraically by automorphisms on $F/p^{nr}F.$ Applying the Lie algebra functor gives a natural map from the group $\SL(n,\V/p^{nr}\V)$ to the automorphisms of the Lie algebra of $F/p^{nr}F.$ These automorphisms are restricted Abelian Lie algebra morphisms. The Lie algebra is $\F/\F\theta^{nr}\simeq (\Or/\Or\theta^{nr})^n$ and the restricted structure is the left $\Or$-module structure.
The left module endomorphisms of $(\Or/\Or\theta^{nr})^n$ are the elements of the matrix algebra $M_n(\Or/\theta^{nr}\Or)$, and the automorphism group is $\GL(n,\Or/\theta^{nr}\Or).$ Hence there is an induced homomorphism:
\[
\phi_r:\SL(n,\V/p^{nr}\V)\rightarrow \GL(n,\Or/\theta^{nr}\Or).
\]
The morphism is induced by the Lie algebra functor. Suppose that $\alpha \in \SL(n,\V/p^{nr}\V)$ is an element such that $\alpha(L)=L'.$ That is, $\phi_r$ defines an action of $\SL(n,\V/p^{nr}\V)$ on $\mathcal{L}_{nr} (\F/\theta^{nr}\F)$ so that the canonical morphism is equivariant. Notice that since it is injective the $\SL(n\V/p^{nr}\V)$-stabilizer of $\lie L$ is equal to the stabilizer of $L$. Consequently, the dimension of the $\GL(n,\Or/\theta^{nr}\Or)$-orbit is greater than or equal to the dimension of the $\SL(n, \V/p^{nr}\V)$-orbit. Observe that both of these orbits are varieties and that the $\GL(n,\Or/\theta^{nr}\Or)$-orbit contains the $\SL(n, \V/p^{nr}\V)$-orbit. We shall attempt to compare the dimensions of these orbits.

Now $F,L$ and $F/L$ are all smooth algebraic groups and so $\lie{F/L}=\lie F/\lie L.$ Then $F/L\simeq \bigoplus_j \V/p^{s_j}\V$ for some nonincreasing sequence $(s_1,\dots,s_n)$. Since the Lie algebra of a product is the product of the corresponding Lie algebras, $\F/\lie L\simeq \bigoplus_j \lie {\V/p^{s_j}\V}.$ But $\lie {\V/p^{s_j}\V}=\Or/\theta^s_j\Or.$ Consequently, $\lie{F/L}=\bigoplus_j\Or/\theta^{s_j}\Or.$
By Proposition~\ref{dim}, the dimension of the $\GL(n,\Or/\theta^{nr}\Or)$-orbit of $\lie L$ is $\sum_{i<j}(s_i-s_j).$ This is the same as the $\SL(n,\V/p^{nr}\V)$-orbit dimension of $L.$ Thus the two orbits are varieties, one contained in the other and both of the same dimension. To complete the proof just note that 
$\phi_r$ is an algebraic morphism of algebraic groups. It follows that its image is a closed subgroup of $\GL(n,\Or).$ Hence the $\SL(n,\V/p^{nr}\V)$-orbit of a point is a closed subset of its $\GL(n,\Or/\theta^{nr}\Or)$-orbit. The one orbit is hence a closed subvariety of the other and of the same dimension and so they are equal. In particular the maximal dimensional orbits are isomorphic. The theorem follows immediately.
\end{proof}
\section{$p$-Linear Nilpotents}\label{pln}
In \cite{george}, Lusztig established an isomorphism between the nilpotent cone under the conjugating action and an open subset of the variety of $k[[t]]$-lattices of codimension $n$ in $k[[t]]^n.$ Since then that parametrization has been understood much more fully. P. Magyar generalized to lattices of codimension $nr$ (see \cite{pm}). Later authors generalized further replacing the nilpotent cone with constructions involving opposite parahorics. (See for example the discussion in \cite{mov}.) In this section we construct a parallel theory. There is a generalization involving representations of $\boldsymbol{\alpha}_{p^n}$ but that would carry us too far afield. We stay with the most elementary case in which the parallel is clear and which shows that $p$-linear algebra must replace ordinary linear algebra. Our construction is dual to the one in \cite{george} and it establishes a more direct correspondence. At the end of the section we give an indication of how to show that this construction is exactly dual to Lusztig's construction. If $V$ is an arbitrary $k$-vector space of dimension $n$, a \emph{$p$-linear} map $\phi:V\rightarrow V$ is an additive map such that $\phi(av)=a^p\phi(v).$ If we fix a basis $\{ e_1, \dots ,e_n\}$ then a $p$-linear map is determined by its values on the basis. That is, $\phi(\sum a_ie_i)=\sum a_i^p\phi(e_i).$ If $\phi(e_i)=\sum b_{i,j}e_j$ then $\phi(\sum a_ie_i)=\sum_i \sum_j a_i^pb_{i,j}e_j.$ Thinking in terms of matrices, $\phi$ has the matrix $(b_{i,j})$ and its action is described by $(a_i)(b_{i,j})=(\sum_i a_i^pb_{i,j}).$ Thus define a twisted action of matrices by the equation:
\begin{equation}\label{pac}
\left(
\begin{matrix}
a_1, \dots, a_n
\end{matrix}
\right)
\circ 
\left(
\begin{matrix}
b_{1,1}& \dots & b_{1,n} \\
\vdots & \ddots & \vdots \\
b_{n,1}& \dots & b_{n,n}
\end{matrix}
\right)
=
\left(
\begin{matrix}
\sum_i a_i^p b_{i,1}, \dots ,\sum_i a_i^p b_{i,n}
\end{matrix}
\right).
\end{equation}
This twisted action identifies each matrix with a $p$-linear map and, as we observed, this represents each $p$-linear map uniquely. Notice that the representation of $\phi$ by a matrix is dependent on the choice of a basis and that the rules for change of basis are not as in the commutative case.

Let $\OOr$ be the  quotient of $\Or$ by the two sided ideal $\theta^n\Or$ and let $\thetao$ denote the residue class of $\theta$ in this ring. Let $\OF$ denote the quotient $\F/{\thetao}^n\F.$ 
If $A\in M_n(k)$ let $A^{[p^j]}$ denote the matrix obtained by raising the entries of $A$ to the $p^j$'th power. If $\phi$ is a $p$-linear map then so is the conjugate $A\phi A^{-1}.$ If $\phi$ is represented by the matrix $B$ then $A\phi A^{-1}$ corresponds to the matrix $A^{[p]}BA^{-1}.$ The action here is ordinary matrix multiplication and not the twisted action above.
Notice that the twisted action and ordinary matrix multiplication act according to the formula;
\[
vA\circ B=v\circ A^{[p]}B.
\]
It is also true that the matrix $B$ represents a $p$-nilpotent transformation if and only if 
\begin{equation}\label{nil}
B^{[p^{n-1}]}B^{[p^{n-2}]}\dots B=0.
\end{equation}
Equation \eqref{nil} defines a subscheme of the space of $n\times n$ matrices. This cone of $p$-nilpotents has a very different set of equations than the set of equations defining the classical nilpotent cone.

Consider the Grassmann variety of $n^2-n$-planes in $\OF.$ The set of $n^2-n$-planes closed under multiplication by $\thetao$ is a closed subscheme of this Grassmannian. Write $Y$ for the variety of $n^2-n$-planes in $\F$ stable under left multiplication by $\thetao.$ Let $Y_0$ denote the open subscheme of $Y$ consisting of $n^2-n$-planes transversal to $V.$ 

Suppose that $\phi:V\rightarrow V$ is a $p$-nilpotent transformation. Then letting $\thetao u=\phi(u)$ one determines an $\OOr$-module structure on $V$ and, because of the nilpotence on a space of dimension $n,$ an $\OOr$-structure. Write $V_{\phi}$ for this $\OOr$-module. Consider the set of elements: 

\begin{equation}\label{gen}
R(\phi)=\{{\thetao}^{a+b}\otimes \phi^c(m)-{\thetao}^a\otimes \phi^{b+c}(m):m\in V,\; a,b,c\in \mathbb{Z},\;a,b,c\geq 0\}.
\end{equation}

Then let $\mathfrak{L}(\phi)$ be the $k$-span of $R(\phi).$

Recall that $\GL(n,\Or)$ operates on $\Or \otimes_kV=\F$ on the right by regarding it as a free two-sided module. Notice also that the usual change of basis rules apply to this situation. That is, the formula for the action changes by conjugation by a change of basis matrix. Notice that $\Or$-basis changes correspond to choices of the isomorphism between the lattice variety and the variety of $\Or$ subspaces of codimension $n$ in $\F/\theta^{nr}\F.$ Now $\OF$ is the quotient of $\F$ by the two-sided submodule ${\thetao}^n\F$  and so $\GL(n,\OOr)$ operates on $\OF$ by right multiplications taking submodules with a fixed codimension and elementary divisor type into modules of the same codimension and elementary divisor type. Note however that the image of $\GL(n,k)$ in $\GL(n,\OOr)$ is not independent of the choice of $V$ and the isomorphism $\Or\otimes V\simeq \F$ and that the conjugate of a $p$-linear map is again $p$-linear only for elements of $\GL(n,k).$ (We caution the reader that the center of this group is very small, only $\mathbb{F}_p^*.$ The conjugate of a $p$-linear map is $p$-linear only for the  appropriate conjugate of $k$.)
\begin{theorem}
For each $p$-nilpotent transformation $\phi:V\rightarrow V,$ the $\OOr$-module $\lie \phi$ is an $\OOr$-submodule of $\OF$ of codimension $n$ transversal to $V.$ The correspondence which associates $\phi$ to $\lie \phi$ is a bijective correspondence between the set of $p$-nilpotents on $V$ and the set of $\OOr$-submodules of $\OF$ of codimension $n$ transversal to $V.$ For a fixed choice of basis, this correspondence is a $\GL(n,k)$-equivariant algebraic isomorphism from the scheme of $p$-nilpotents to $Y_0$ under the natural action. \end{theorem}

\begin{proof} Consider $R(\phi)$ (see \eqref{gen}). It is closed under multiplication by $\thetao$ and so $\lie \phi$ is an $\OOr$-submodule of $\OF.$ Then note that the expression ${\thetao}^{a+b}\otimes \phi^c(m)-{\thetao}^a\otimes \phi^{b+c}(m)$ is $p^{a+b+c}$-linear in $m$ for fixed $a, b$ and $c.$ It follows that, for any basis $\{v_1,\dots ,v_n\},$ the $k$-span of $R(\phi)$ is the same as the $k$-span of $\{{\thetao}^{a+b}\otimes \phi^c(v_i)-{\thetao}^a\otimes \phi^{b+c}(v_i):i=1, \dots , n,\; a,b,c\in \mathbb{Z},\;a,b,c\geq 0\}.$ 

Let $\phi:V\rightarrow V$ be a $p$-nilpotent. Then consider $V_{\phi}$ the $\OOr$-module described above. 
Define a map $\lambda_{\phi}:\OOr\otimes V\rightarrow V_{\phi}$ by the equation $\lambda_{\phi}({\thetao}^s\otimes m)=\phi^s(m).$ This is a surjective left $\OOr$-module map and since $V_{\phi}$ is of dimension $n,$ its kernel is an $\OOr$-submodule of $\OF$ of codimension $n.$ We propose to prove that it is transversal to $V$ and that it is equal to $\lie \phi.$ Consider the element ${\thetao}^{a+b}\otimes \phi^c(m)-{\thetao}^a\otimes \phi^{b+c}(m)$ in $R(\phi).$ It is clear that $\lambda_{\phi}$ takes such an element to $0$ and so $\lie \phi \subseteq \operatorname{Ker}(\lambda_{\phi}).$

Let $M=\OF/\lie \phi.$ It is clear that there is a natural surjection $\overline{\lambda}:M\rightarrow V_{\phi}.$ Since expressions $({\thetao}^s\otimes m-1\otimes \phi^{s}(m))$ all lie in $\lie \phi,$ each element of $M$ is congruent to an element of the image of $V$ in $M.$ Hence $M$ is of dimension at most $n.$ On the other hand it maps surjectively onto $V_{\phi}$ and so it is of dimension at least $n.$ Thus $M$ is isomorphic to $V_{\phi}$ and so $\lie \phi$ is an $\OOr$-submodule of $\OF$ of codimension $n.$ Furthermore, by definition, $\lambda_{\phi}$ carries the element $1\otimes v$ to $v$ in $V_{\phi}.$ That is, $\lambda_{\phi}|V$ is an isomorphism and so $\lie \phi$ is transversal to $V.$

Suppose that $\mathfrak{L}$ is an $\OOr$-submodule of $\OF$ transversal to $V.$ Let $\pi:\OF \rightarrow \OF/\mathfrak{L}$ be the natural projection. Then, by transversality, $\pi|V$ is an isomorphism. Define $\phi_{\mathfrak{L}}$ by the equation $\phi_{\mathfrak{L}}(m)=(\pi|V)^{-1}(\thetao \pi(m)).$ This is a $p$-nilpotent on $V.$ It is the simplest triviality to show that the correspondences $\phi \mapsto \lie \phi$ and $\mathfrak{L}\mapsto \phi_{\mathfrak{L}}$ are inverse to one another.

We now turn to the equivariance. For this we must emphasize that a basis $\{e_1, \dots , e_n\}$ has been chosen and that all matrices and vectors are expressed as rows with respect to this basis. Define an operation on the vector $v=(c_1,\dots ,c_n)$ by the equation $v^{[p]}=(c_1^p,\dots ,c_n^p).$ Then by \eqref{pac} if $\phi$ has the matrix $B,$ the action of $\phi$ is described by $\phi(v)=v^{[p]}B.$ Let $\psi=A\phi A^{-1}$ where we are taking the simple composition of a linear operator, a $p$-linear operator and the inverse of the linear operator. We assert that the following diagram commutes:
\begin{equation}\label{eqv}
\begin{CD}
\OOr\otimes V @>\lambda_{\phi}>> V_{\phi}\\
@VV1\otimes A^{-1}V  @VV A^{-1}V \\
\OOr \otimes V @>\lambda_{\psi}>> V_{\psi}
\end{CD}
\end{equation}
The vertical arrows represent simple right matrix multiplication by $1\otimes A^{-1}$ and $A^{-1}$ respectively. The horizontal maps are those defined above labeled by their associated $p$-linear map. Then $(\lambda_{\phi}({\thetao}^r\otimes v))A^{-1}=(\phi^r(v))A^{-1}=(\dots ((v^{[p]}B)^{[p]}B)^{[p]}\dots B)^{[p]}A^{-1}=v^{[p^r]}B^{[p^{r-1}]}\cdots BA^{-1}.$ Descending by $1\otimes A^{-1}$ and applying $\lambda_{\psi}$ yields:
\begin{equation*}\begin{split}
\lambda&_{\psi}({\thetao}^r\otimes vA^{-1}) =\left(\dots \left(\left(\left(vA^{-1}\right)^{[p]}\right)A^{[p]}BA^{-1}\right)^{[p]}  \dots \right)^{[p]}A^{[p]}BA^{-1}\\
&=\left(v^{[p^r]}(A^{[p^r]})^{-1}\right)\left(A^{[p^r]}B^{[p^{r-1}]}(A^{[p^{r-1}]})^{-1}\right)\left(A^{[p^{r-1}]}B(A^{[p^{r-2}]})^{-1}\right) \\
&\hskip 1cm\dots\left(A^{[p^2]}B^{[p]}(A^{[p]})^{-1}\right)\left(A^{[p]}BA^{-1}\right)\\&=v^{[p^r]}B^{[p^{r-1}]}B^{[p^{r-2}]}\dots BA^{-1}.
\end{split}
\end{equation*}
This establishes the commutativity of \eqref{eqv}.  The commutativity of that diagram implies that $\operatorname{Ker}(\lambda_{\psi})=(\operatorname{Ker}(\lambda_{\phi}))A^{-1}.$ But then $\lie \phi=\operatorname{Ker}(\lambda_{\phi})$ and $\psi =A\phi A^{-1}$ so that this equation can be rewritten as $\lie {A\phi A^{-1}}=\lie \phi A^{-1}.$ That is, the correspondence between $p$-nilpot{-}ents on $V$ and $\OOr$-submodules of $\OF$ transversal to $V$ is equivariant with respect to the natural action of $\GL(n,k).$ 
\end{proof}
The reader is cautioned that the subgroup $\GL(n,k)\subseteq \GL(n,\OOr)$ is dependent on the choice of basis. It is also true that the conjugate of a $p$-linear map by an element of $\GL(n,\OOr)$ is not in general a $p$-linear map.

We conclude with some remarks on the relation of Lusztig's original construction to this construction. As we have established an analogous isomorphism, there is no special reason to give full details. We outline a procedure to show that it is the natural dual of our construction. Let $U$ be the linear dual of $V$ and let the value of $u$ on $v$ be $\langle u, v\rangle.$ Let $\mathfrak{U}=U\otimes_k\OOr$ and view it as a right $\OOr$-module. View $\mathfrak{U}$ as a $k$-vector space under right multiplication by $k$ and view $\OF=\OOr \otimes V$ as a left $\OOr$-module and a $k$   -vector space under left multiplication. Extend the natural pairing to these modules by the equation, $\langle u\otimes {\thetao}^r, {\thetao}^s\otimes v\rangle =\langle u,v\rangle \delta_{r+s,n-1},$ this last symbol being a Kronecker delta. This is an adjoint pairing, that is, $\langle ma, n\rangle =\langle m, an\rangle,\; a\in \OOr.$ Hence the orthogonal complement of an $\OOr$-submodule is an $\OOr$-submodule. Let $\mathfrak{U}_1= \sum_{i=0}^{n-2}U\otimes {\thetao}^i .$ Note that $\mathfrak{U}_1$ and $V$ are orthogonal complements. Thus submodules of $\OF$ of codimension $n$ transversal to $V$ are complementary to submodules of $\mathfrak{U}$ transversal to $\mathfrak{U}_1.$ Inverse $p$-linear maps are additive maps such that $\tau(a^pv)=a\tau(v).$ One shows that for $\mathfrak{M},$ any $n$-dimensional subspace of $\mathfrak{U}$ transversal to $\mathfrak{U}_1,$ there is a unique inverse $p$-linear map $\tau$ so that $\mathfrak{M}$ is the set of all elements of the form $u\otimes {\thetao}^{n-1}+\sum_{i=0}^{n-2}\tau^i(u)\otimes {\thetao}^{n-1-i}.$ Then one shows that the orthogonal complement of $\lie \phi$ is the $n$-dimensional submodule of $\mathfrak{U}$ associated to the inverse $p$-linear map $\tau$ which is uniquely determined by the equation $\langle u, \phi(v)\rangle^p =\langle \tau(u),v\rangle.$

\section{Line Bundles on Orbit Closures}\label{OC}
This section is devoted to the geometry of $\Lrn$ which is largely determined by its $\SL(n,\V)$-orbit structure. The orbit structure was essentally known to Hilbert and was the basis of the Birkhoff decomposition. One of the first explicit descriptions of it in modern terms is to be found in \cite{im}. 

We recall a very elementary result from commutative algebra. Let $R$ be a normal domain and let $\mathcal{M}$ denote the set of height one primes in $R$. For a finitely generated $R$-module $N$, let $N^*$  denote its linear dual. Then for any finitely generated torsion free module $N$, $N^{**}=\bigcap_{\mathfrak{p}\in \mathcal{M}}N_{\mathfrak{p}}$. Consequently, since $\Lrn$ is normal, the global sections in any coherent reflexive sheaf can be computed on any open set containing all the height one primes. Recall that $\mathbb{L}(r_1,\dots ,r_n)$ is the lattice with a basis $\{p^{r_1}e_1,\dots ,p^{r_n}e_n\}$.  Theorem 6 of subsection 3.5 of \cite{me} asserts that the smooth locus of $\Lrn$ is exactly the $\SL (n,\V)$-orbit of the lattice, $\mathbb{L}(nr,0,\dots, 0)$. Said otherwise, this is just the set of lattices $L$ such that $F/L\simeq \V/p^{nr}\V$. The complement of this orbit is of codimension two and it is the closure of the orbit of the lattice $\mathbb{L}(nr-1,1,0,\dots,0)$ (see the end of the proof of Proposition 10 in \S 3 of \cite{me}).

The elements of $\SL (n,\V)$ with subdiagonal elements of positive valuation constitute an Iwahori subgroup.  Write $B$ for this Iwahori subgroup  of $\SL (n,\V)$. Then the $B$-orbits in $\Lrn$ correspond exactly to the lattices $\{ \mathbb{L}(r_1,\dots, r_n):0\leq r_i \leq nr, \sum_ir_i=nr\}$. These $n$-tuples correspond exactly to the elements $p^r\operatorname{diag}(p^{r_1},\dots ,p^{r_n})$, which correspond to the dual of the character group of $\SL (n,\V)$, that is to say, its cocharacters which can be seen either as one-parameter subgroups or as characters of the Langlands dual. (The $\SL(n,\V)$-orbits correspond to the dominant cocharacters.) We summarize the results of \S 3.4 of \cite{me} in the following lemma. The proof is contained in several statements, equations and discussions in that section of \cite{me}.

\begin{lemma}\label{L:or}
Under the $\SL(n,\V)$-action or the $B$-action on $\Lrn$ there are finitely many orbits. Moreover:
\begin{enumerate}
\item The $\SL(n,\V)$-orbits are the orbits of the elements $\mathbb{L}(r_1,\dots ,r_n)$ where $0\leq r_i,\,  r_1\geq r_2 \geq \dots \geq r_n$ and $\sum_ir_i=nr$.
\item The $B$-orbits correspond to all of the lattices $\mathbb{L}(r_1,\dots ,r_n)$, $0\leq r_i,\,\sum_ir_i=nr$.
\item There is a unique dominant $\SL(n,\V)$-orbit, the orbit of the group  $\mathbb{L}(nr,0,\dots ,0)$. The complement of this orbit is of codimension two.
\item There is a unique dominant $B$-orbit. It is isomorphic to the affine $n(n-1)r$-space over $k$.
\item The complement of the unique dominant $B$-orbit is the closure of a unique codimension one orbit, the orbit of $\mathbb{L}(0, nr, 0, \dots ,0)$.
\end{enumerate}
\end{lemma}

The analogue of the result below in the classical case is well known. See for example \cite{pic}.

\begin{proposition}\label{P:cl}
Let $D$ denote the closure of the unique codimension one $B$-orbit in $\Lrn $. All of the $B$-orbits in $\Lrn$ are isomorphic to affine spaces 
$\mathbb{A}^N_k$ for some $N.$ The Picard group of $\Lrn$ is isomorphic to $\mathbb{Z}$ and its generator is the invertible sheaf $\V (D)$, the polar divisor of $D$.
\end{proposition}
\begin{proof}
The proof of this is standard. First note if an element of the group $\SL(n,\V)$ is congruent to $I_n$ modulo $p^{nr}$, it acts trivially on $\Lrn $. Let $\Cr$ denote this congruence subgroup. Then $\SL (n,\V)/\Cr$ is a finite dimensional algebraic group with a nontrivial unipotent radical and semisimple part equal to $\SL(n,k)$.  A standard maximal torus in $\SL(n,\V)$ is the set of matrices in it with multiplicative representatives (i.e. roots of unity) on the diagonal. Let $\Gamma_0$ denote the group of elements in $\SL(n,\V)$ which are congruent to $I_n$ modulo $p$.  The unipotent radical of $\SL(n,\V)/\Cr$ is $\Gamma_0/\Cr$. Notice that $B/\Cr$ is a solvable group.

Each lattice of the form $\mathbb{L}(r_1,\dots ,r_n)$ is stabilized by the maximal torus and of course by $\Cr$. Hence its orbit is an orbit of the unipotent radical of $B/\Cr$. All unipotent orbits over an algebraically closed field are affine spaces. This establishes the first assertion.

For the second note that since $\Lrn$ is normal, every Weil divisor is Cartier. Thus the Picard group is the divisor class group. Any divisor can be written in the form $mD+\sum_{i=1}^q t_iM_i$ where the $M_i$ are divisors with generic point in the dominant orbit and $t_i\in \mathbb{Z}$. Since the coordinate ring of the dominant orbit is a unique factorization domain, there is a rational function on it $f$, whose divisor is of the form $sD+\sum_{i=1}^qt_iM_i$. Hence every divisor is linearly equivalent to a multiple of $D$. On the other hand, $\Lrn$ is projective and so has a very ample divisor which by necessity is not torsion. That divisor is a multiple of $D.$ Hence $D$ is ample. It is worth noting that the same argument, which is frequently deployed in discussions of rational varieties, appears in \cite{me2} to prove a similar result for the dominant $\SL(n,\V)$-orbit.
\end{proof}
We turn to some elementary aspects of the structure of $\Lrn $ which will permit us to view $\Gamma (\Lrn,\V (qD))$, the module of sections as a module of sections on a more tractable subvariety. It is a representation of $\SL (n,\V)$ over $k.$  We will describe the dualizing sheaf explicitly.  We use $X_0$ to denote the dominant $\SL (n,\V)$-orbit in $\Lrn$. We begin with the geometric structure of the dominant orbit.  We recall that $\Gamma_{nr}$ denotes the set of matrices in $\SL (n,\V)$ congruent to $I_n$ modulo $p^{nr}$. Let $G_r=\SL (n,\V)/\Cr \simeq \SL (n,\V/p^{nr}\V )$. Then $\Cr$ operates trivially on $\Lrn$ and so it admits a $G_r$-action. There is a natural map from $G_r$ to $\SL(n,k)$ induced by the natural surjection, $\V/p^{nr}\V \rightarrow k$. This homomorphism induces a natural transitive left action of $G_r$ on $\mathbb{P}_k^{n-1}$.
Let $P$ denote the subgroup of $G_r$ consisting of elements whose first row is of the form $(a_1, pa_2,\dots ,pa_n)$. It is a maximal parabolic and the stabilizer of a point on $\mathbb{P}_k^{n-1}$.
\begin{theorem}\label{T:do}
The dominant orbit $X_0$ is a principal fiber space over $\mathbb{P}_k^{n-1}$. Let $P_r$ denote the subgroup of $G_r$ consisting of elements with first row $(a, 0,\dots ,0)$. Then $X_0$ may be written in the form $G_r\times^P P/P_r$. Moreover $P/P_r$ is isomorphic to the affine space of dimension $(n-1)(nr-1)$ over $k$, and the projection $X_0\rightarrow \mathbb{P}_k^{n-1}$ is an affine morphism.
\end{theorem}
\begin{proof}
As we have noted the dominant orbit is the $G_r$-orbit of \linebreak $\mathbb{L}(nr,0,\dots ,0)$. If $L$ is any other lattice which can be written in the form $N+p^{nr}v$, where $N$ is a rank $n-1$ direct summand of $F$ and $v$ is a vector which reduces to a cyclic generator of $F/N$, then $L$ is certainly in this orbit and clearly lattices in the dominant orbit are all of this form.

In general, let $(\bu _1,\dots ,\bu_n)$ be a set of column vectors spanning a lattice of discriminant $nr$ in $F$. The matrix $M\in \SL(n,\V)$ operates on this span by replacing the given basis with $M\bu _1,\dots ,M\bu _n$. These two bases span the same lattice if and only if there exist elements $a_{i,j}$ such that $M\bu _j=\sum_{i=1}^na_{i,j}\bu _j$. The $a_{i,j}$ form a matrix in $\GL (n,\V)$, and since $M\in \SL(n,\V)$, the matrix $(a_{i,j})$ is in $\SL(n,\V)$.

Let $U$ be the matrix with columns $\bu _j$ and let $A=(a_{i,j})$. According to the previous paragraph, $U$ and $MU$ have the same lattice as column space if and only if there is $A\in \SL(n,\V)$ such that $MU=UA$. In another context this would just mean that $M=UAU^{-1}$. Simple as this expression is, multiplication by a matrix involving negative powers of $p$ in Witt vectors would require taking $p$'th roots.  Let $D=\operatorname{diag}(p^{nr},1,\dots ,1)$. Then $M$ stabilizes the column span of $D$ if and only if for some $A\in \SL(n,\V),\;MD=DA$. If $M=(m_{i,j}),\; A=(a_{i,j})$, the condition is that $M\in \SL (n,\V)$ and $m_{1,j}=p^{nr}a_{1,j}$ with $j>1$ and $p^{nr}m_{1,1}=p^{nr}a_{1,1}$, that is, $m_{1,j}\in p^{nr}\V$ for $j>1$.  There is also a condition on the first column of $A$ but this does not impose any condition on $M.$ Thinking modulo $\Cr$ the stabilizer of $\mathbb{L}(nr,0,\dots ,0)$ in $G_r$ is the subgroup consisting of matrices with first row congruent to $(a_{1,1},0,\dots ,0)$ in $\SL(n,\V)/\Cr.$ Call this group $P_r$. Then $X_0=G_r/P_r$. 

Now $P_r$ is contained in the maximal parabolic $P$.  Then $G_r/P=\mathbb{P}_k^{n-1}$ and there is a natural $G_r$-equivariant map $G_r/P_r\rightarrow G_r/P$ with fiber $P/P_r$. That $G_r/P_r=G_r\times^PP/P_r$ follows trivially. We must describe $P/P_r$ as a $P$-space.  If $X\in P$, we may write it:
\begin{equation*}
 \left(
 \begin{matrix}
	x_{1,1} & p\boldsymbol{a} \\
        X_1     & X_2
  \end{matrix}
  \right).
\end{equation*}
In this expression, $x_{1,1}\in \V/p^{nr}\V,\:X_1$ is a column of height $n-1$ with entries in $\V/p^{nr}\V,\: X_2$ is an $n-1$ by $n-1$ matrix with entries in $\V/p^{nr}\V$ and $\boldsymbol{a}$ is a row of length $n-1$ with entries in  $\V/p^{nr-1}\V$. It can be factored uniquely as:
\begin{equation}\label{E:bc}
\left(
 \begin{matrix}
	x_{1,1} & p\boldsymbol{a} \\
        X_1     & X_2
  \end{matrix}
  \right)=
  \left(
  \begin{matrix}
         1  &  p\boldsymbol{a}X_2^{-1} \\
         0  &   I
  \end{matrix}
  \right)
  \left(
  \begin{matrix}
        x_{1,1}-p\boldsymbol{a}X_2^{-1}X_1  & 0 \\
         X_1                                & X_2
  \end{matrix}
  \right).
\end{equation}
Since the determinant of $X$ is $1$ and the determinant of the first factor in (5.1) is $1$, the determinant of the second factor is \emph{a fortiori} $1$. This factorization implies that $P\simeq U\times_kP_r$ where $U$ is the group of elements:
\begin{equation}\label{D:U}
\left(
 \begin{matrix}
	1 &p\boldsymbol{a} \\
	0 & I
 \end{matrix}
\right).
\end{equation}
The isomorphism is defined by the group product and it is an isomorphism only of varieties. It is clear that the set of elements $\{pa:\, a\in \V/p^{nr}\V \}$ is isomorphic to affine $nr-1$ space. Hence $U\simeq \mathbb{A}_k^{(n-1)(nr-1)}$.
\end{proof}
\begin{corollary}\label{C:is}
The isomorphism between $\mathbb{A}^{(n-1)(nr-1)}_k$ and $P/P_r$ is the map which sends $(a_2,\dots ,a_{n-1})$ to the $P_r$-coset of
\begin{equation}\label{M:par}
\left(
\begin{matrix}
	1 & pa_2 & pa_3 & \dots & pa_{n-1} \\
	0 &  1  & 0     & \dots &  0       \\
	\vdots & \vdots & \ddots & \vdots & \vdots \\
	0 & 0 &0 &\dots & 1
\end{matrix}
\right).
\end{equation}
In this expression, $a_i\in \V/p^{nr-1}\V$ with $a_i=\sum_{j=1}^{nr-2}\xi(x_{i,j})^{p^{-j}}p^j$ 
and $x_{i,j}$ is an indeterminate.  Write $p\boldsymbol{a}_{2,1}=(pa_2, \dots ,pa_{n-1})$.  Then the action of the matrix,
\begin{equation}\label{M:gm}
\left(
\begin{matrix}
	x_{1,1}& pX_{1,2} \\
	X_{2.1} & X_{2,2}
\end{matrix}
\right)
\end{equation}
on the coset of \eqref{M:par} is described by the expression
\begin{equation}\label{M:act}
\left(
\begin{matrix}
	1&p(x_{1,1}\boldsymbol{a}+X_{1,2})(X_{2,2}+pX_{2,1}\boldsymbol{a})^{-1}\\
	0 & I
\end{matrix}
\right).
\end{equation}
\end{corollary}
\begin{proof}
To see this just apply the factorization formula \eqref{E:bc} to the product of the matrix \eqref{M:gm} and the parametrized matrix \eqref{M:par}.
\end{proof}
Theorem~\ref{T:do} and the corollary following offer some hope of computing the global sections in line bundles on $\Lrn$. If $\mathcal{L}$ is a line bundle on $\Lrn$ it is certainly reflexive and so by Main Theorem~\ref{M:th} and the well known classical result on reflexive modules over normal domains, $H^0(\Lrn,\mathcal{L})=H^0(X_0,\mathcal{L}|X_0)$. The same \emph{does not} apply to higher cohomology. The proposition has a very interesting interpretation which clearly can be generalized to other generalized Schubert cells. It is to be noted that this is an equally valid and novel way of viewing Schubert cells in the affine Grassmannian and other affine infinite dimensional homogeneous spaces. The proofs that do not apply with no modification apply with obvious and minimal modification. This is not our concern in this work so we leave it for another study. The interpretation is:
\begin{remark}\label{R:M}
We may consider the $\Spec (\V/p^{nr}\V)$-scheme $\mathbb{P}^{n-1}_{\V/p^{nr}\V}$. The $\V/p^{nr}\V$-points of this scheme can be thought of as rank $n-1$ free $\V/p^{nr}\V$-direct summands of $(\V/p^{nr}\V)^n$ and $\X_0$ can be thought of as the Greenberg scheme associated to this scheme with respect to the ring scheme, $\V/p^{nr}\V$ (not the localized Greenberg functor of \cite{me}). The map $X_0\rightarrow \mathbb{P}_k^{n-1}$ can be regarded as the map which sends the $\V/p^{nr}\V$-direct summand $L$ to $L+pF/pF$ in $F/pF$.
\end{remark}

We write $\pi:X_0\rightarrow \mathbb{P}_k^{n-1}$ for the projection. The most useful consequence of these observations is:
\begin{lemma}\label{L:rs}
Let $\mathcal{F}$ be a reflexive coherent sheaf on $\Lrn $. Then $H^0(\Lrn,\mathcal{F})=H^0(\mathbb{P}^{n-1}_k,\pi_{*}(\mathcal{F}|X_0))$.
\end{lemma}
\begin{proof}
The sheaf $\mathcal{F}$ is reflexive and $\Lrn$ is normal by Main Theorem~\ref{M:th}. Since the complement of $X_0$ is of codimension $2$, \[H^0(\Lrn,\mathcal{F})=H^0(X_0,\mathcal{F}|X_0).\] On the other hand, $\pi$ is affine and so by the Leray spectral sequence, $H^i(\mathbb{P}_k^{n-1}, \pi_{*}(\mathcal{F}|X_0))=H^i(X_0,\mathcal{F}|X_0)$.
\end{proof}

\begin{remark}\label{R:uhp}
We take note of an amusing analogy suggested by \eqref{M:act}. We consider the maximal parabolic $P\subseteq SL(n, \V/p^{nr}\V )$. We write the element $X$ of $P$ in the form:
\begin{equation}\label{lfa}
X=
\left(
\begin{matrix}
 x_{1,1} & X_{1,2} \\
 X_{2,1} & X_{2,2}
\end{matrix}
\right).
\end{equation}
In this expression, $x_{1,1}$ is a unit in $\V/p^{nr}\V ,\: X_{1,2}$ is a row of length $n-1$ with entries in $p\V/p^{nr}\V ,\: X_{2,1}$ is a column with entries in $\V/p^{nr}\V $ and $X_{2,2}$ is an invertible matrix with entries in $\V/p^{nr}\V $. Define an action of $X$ on the row $u\in (p\V/p^{nr}\V)^{n-1}$ by the equation:
\begin{equation}\label{E:hpa}
X\cdot u=(x_{1,1}u+X_{1,2})(X_{2,1}u+X_{2,2})^{-1}.
\end{equation}
Notice that the matrix in the second parentheses is invertible because $X_{2,2}$ is invertible and the first term is in $(p\V/p^{nr}\V)^{n-1}$. It follows from \eqref{M:par} and \eqref{M:act} that this is an action. Then $X$ stabilizes $0$ only if $X_{1,2}X_{2,2}^{-1}=0$,
that is, if $X_{1,2}=0$. That is, the stabilizer of $0$ is
$P_r$ and the action of $P$ on $U=(p\V/p^{nr}\V)^{n-1}$ is a kind of $p$-adic modular action.
\end{remark}

For $j\geq 1$, let $\Gamma_j$ denote the group of matrices in $\SL (n,\V)$ congruent to $I_n$ mod $p^{j}$, and let $\overline{\Gamma}_j=\Gamma_j/\Gamma_{nr}\subseteq \SL(n,\V/p^{nr}\V)$.  Then $\overline{\Gamma}_1$ is the unipotent radical of the $k$-group $\SL(n,\V/p^{nr}\V)$ with semisimple part $\SL(n,k)$.

\section{The Restricted Lie Algebra of $\SL(n,\V)$; the Canonical Bundle}\label{RLA}
In order to understand the ample generator of the Picard group of $\Lrn$, we must describe the Lie algebras of the proalgebraic group $\SL(n,\V)$, and the subgroups of it which appear in the discussions above. These are group schemes over a field of positive characteristic and so one must contend with Frobenius twists and restricted Lie algebras. For a representation $\rho$ on $V$ of a group over $k$, its $\nu$'th Frobenius twist is the representation on $V,$ the representing transformations of which are just matrices with entries $p^{\nu}$'th powers of the entries of the corresponding representing transformation of $\rho.$ It is written $V^{[p^{\nu}]}.$ 

Restricted Lie algebras are Lie algebras with a formal $p$'th power. In Lie algebras of groups over fields of positive characteristic, the operation is just $p$'th power iteration of invariant vector fields. As we shall work closely with them we recall the definition. The Lie algebra $L$ over $k$ is \emph{restricted} if it is equipped with a formal $p$'th power $x\mapsto x^{[p]}$ such that $(ax)^{[p]}=a^px^{[p]}$ and such that $(x+y)^{[p]}=x^{[p]}+y^{[p]}+\sum_{j=1}^{p-1}s_j(x,y)$, where the $s_j$ are Lie polynomials defined as follows. Extend $L$ to a Lie algebra over the polynomial ring $k[T]$ and write the expression:
\begin{equation}\label{D:rl}
(ad(Tx+y))^{p-1}(x)=\sum_{i=1}^{p-1}is_i(x,y)T^i.
\end{equation}
Clearly, $s_i$ is a sum of terms of the form $a[z_1,[z_2,[\dots, [z_{p-1},x]\dots ]]]$, where $a\in \mathbb{Z}$ and each $z_j$ can be either $x$ or $y$. The $s_i$ are called the Jacobson polynomials. Since the Jacobson polynomials are Lie polynomials, if $\tau:L\rightarrow L'$ is a Lie algebra homomorphism,
\begin{equation}\label{lp}
\tau(s_i(x,y))=s_i(\tau(x),\tau(y)),\;i=1,\dots ,p-1.
\end{equation}

Note that $\SL(n,\V)$ contains its big cell in its algebraic Bruhat decomposition as an $\V$-group and that the $k$-tangent space to $e$ in this open subset is the Lie algebra of the $k$-group $\SL(n,\V)$. Moreover the big cell is stable under the conjugating action of the maximal $\V$-torus as well as the maximal $k$-torus. This observation allows us to describe the Lie algebra of $\SL(n,\V)$ as a representation of the maximal $k$-torus. This will be our starting point.

Write ${\V}^*$ for the group of units in $\V$ viewed as a $k$-scheme. A typical element of ${\V}^*$ can be written as $\xi(t)+\sum_{j=1}^{\infty}\xi(t_j)^{p^{-j}}p^j$. If we consider this parametrization briefly, it is apparent that ${\V}^*=\mathbb{G}_{m,k}\times U_1$, where $U_1$ is the subgroup of units congruent to $1$ modulo $p$.

Let $A_1$ denote the additive group of Witt vectors congruent to $0$ modulo $p.$ It is a proalgebraic group under addition. 
Note that $2^{3/2}>e.$ Consequently $\ln (2)>2/3$ and so for all $p, \;\ln(p)>2/3.$ Let $v_p$ denote the $p$-adic valuation. Then for all $j$ certainly $v_p(j)\leq \log_p(j)=\frac{\ln(j)}{\ln(p)}<3/2\ln(j).$ In particular $j-v_p(j)\geq j-3/2\ln(j)$ and this later expression clearly tends toward infinity. Hence for $u\in A_1$ the $p$-adic value of $u^j/j$ tends toward infinity. Thus the series $\ln(1+u)=\sum_{j=1}^{\infty}(-1)^{j-1}u^j/j$ is defined and convergent on $U_1=1+A_1.$ In particular the map $u\mapsto \ln(u)$ gives a morphism of groups $U_1\rightarrow A_1.$ Since $u^2$ is of value strictly greater than $u$ for $u\in A_1$ the map may be easily seen to be injective. It is also clear that $\ln(1+p^n\V)\subseteq p^n\V$ whence it induces a map of truncations which is clearly algebraic. It follows that $\ln$ is an albebraic isomorphism from $U_1$ to $A_1.$Thus we may write ${\V}^*=\mathbb{G}_{m,k}\times A_1$.

Consider the $\V$-torus of $\SL(n,\V)$. That is, we consider the set of $\V$-points of the maximal torus of $\SL(n,\V)$. Denote it $T(\V)$. Viewing the $n-1$ fundamental dominant weights as morphisms from the algebraic $k$-group $T(\V)$ to the algebraic $k$-group ${\V}^*$, they establish an isomorphism from $T(\V)$ to $({\V}^*)^{n-1}$, the $(n-1)$-fold product of ${\V}^*$. Since the fundamental dominant weights in this context are merely morphisms of algebraic groups rather than characters ($U_1$ is the unipotent radical of ${\V}^*$), we write these weights $\varpi_i$ for $i=1,\dots ,n-1$, and view them as group morphisms. We write $d\varpi_i$ for the associated map of Lie algebras. Writing $\pi:{\V}^*\rightarrow k^*$ for the natural map, we write $\omega_i=\pi \circ \varpi_i$. Write $T_k$ for the algebraic $n-1$ dimensional split torus. The upshot of our discussion is that $T(\V)=T_k\times A_1^{n-1}$.

Let $\Phi$ denote the set of roots of $\SL(n,\mathbb{Z})$, and let $\Phi_+$ denote the set of positive roots. We fix an ordering on $\Phi$.  Let $\Delta$ be a set of simple positive roots. For $\alpha \in \Phi$, let $U_{\alpha}$ denote the corresponding root subgroup scheme. Then $\SL(n,\V)$ contains each of the root subgroups $U_{\alpha}(\V)$. The dense Bruhat cell in this group scheme can be described as the product $\prod_{\alpha\in \Phi_+}U_{\alpha}(\V)\times T(\V) \times \prod_{\alpha\in \Phi_+}U_{-\alpha}(\V)$, where the product is in the chosen ordering.

The Lie algebra of $\SL(n,\V)$ is, of course a restricted $k$-Lie algebra. Hence its description includes descriptions of two operations, bracket and formal $p$'th power. Moreover, $\SL(n,\V)$ is a proalgebraic group, that is, an inverse limit of finite dimensional algebraic groups. Thus its Lie algebra is an inverse limit of finite dimensional restricted Lie algebras. Thus it is a topological vector space complete in a cofinite linear topology, that is, a topology in which $0$ has a basis of neighborhoods consisting of subspaces of finite codimension.

In addition it admits the adjoint representation of $\SL(n,\V)$. Since $\SL(n,\V)$ contains
$T(\V)$ and hence $T_k$, it also admits a decomposition into $T_k$-weights. We wish to describe each of these structures on this Lie algebra. To facilitate discussion, we shall write $\mathfrak{L}(G)$ for the restricted Lie algebra of the $k$-group $G$.

The topology is determined by a basis of neighborhoods of $0$. The neighborhoods of $0$ are just the Lie algebras of the congruence subgroups $\Gamma_j$. Let $\boldsymbol{\gamma}_j=\lie{\Gamma_j}$. Since the $\Gamma_j$ are normal, the $\boldsymbol{\gamma}_j$ are a basis of neighborhoods of $0$ consisting of ideals defining a separated topology under which $\lie{\SL(n,\V)}$ is complete.
Notice that there is an easily understood formalism associated to the notion of topological Lie algebras of proalgebraic groups. As part of this formalism we note that if the proalgebraic group $G$ is a product $G=H\times M$, then the Lie algebra is the product $\lie{G}=\lie{H}\oplus \lie{M}$.
\begin{definition}\label{D:ws}
We write $\ws{}$ for $\lie{\SL(n,\V)}$, the restricted Lie algebra of the $k$-group $\SL(n,\V)$. We write $\ws{r}$ for the restricted Lie algebra $\lie{\SL(n,\V/p^{nr}\V)}$. \end{definition}

\begin{lemma} There is an unrestricted Lie subalgebra, $\mathfrak{g}_0$ in $\ws{r}$ which maps isomorphically to $\lsl(n,k)$ under the natural map from \linebreak $\lie{\SL(n,\V)}$ to $\lsl(n,k)$induced by the canonical homomorphism $\V \mapsto k$.
\end{lemma}

\begin{proof}
By \cite{cl} or \cite{me2}, the Lie algebra $\lie {U_{\alpha}(\V)}$ is
a direct product $\prod_{j\geq 0}kX_i$ with $p$'th power operation $X_i^{[p]}=X_{i+1}.$ Hence $(\lie {U_{\alpha}(\V)})^{[p]}$ is $T$-stable and of codimension one in $\lie {U_{\alpha}(\V)}.$ It hence has a unique $T$-complement $kX_{\alpha}$ which must be of weight $\alpha$ and it follows that $X_{\alpha}^{[p^s]}$ is of weight $p^s\alpha$ because raising to the $p^s$'th power is a $p^s$-linear map. It is clear that $\lie {U_{\alpha}(\V)}$ is just $\prod_{j\geq 0}kX_{\alpha}^{[p^j]}.$ Let $X_{\alpha}^{(j)}=X_{\alpha}^{[p^j]}.$ Then $\lie{U_{\alpha}}(\V)$ is an Abelian Lie algebra with $p$'th power $(X_{\alpha}^{(r)})^{[p]}=X_{\alpha}^{(r+1)}$. By the remarks above, $\lie{T(\V)}=\lie{T_k}\oplus \lie{A_1^{n-1}}$ and the decomposition is unique.  We write $\mathfrak{h}$ for $\lie{T_k}$.The unipotent radical of $\SL(n,\V)$ is $\Gamma_0$ and the natural surjection $\SL(n,\V)\rightarrow \SL(n,k)$ allows us to identify the root data of $\SL(n,\V)$ with that of $\SL(n,k)$. Since $\mathfrak{h}$ is the Cartan subalgebra of $\mathfrak{s}\mathfrak{l}(n,k)$, it admits a basis $\{h_{\alpha}:\,\alpha \in \Delta\}$, where $h_{\alpha}^{[p]}=h_{\alpha}$ and $h_{\alpha}v=\check{\alpha}(\lambda)v$ whenever $v$ is a vector of weight $\lambda$ and where $\check{\alpha}$ is the coroot associated to $\alpha$.

There is a natural $T$-equivariant map
from $\ws{}$ to $\mathfrak{s}\mathfrak{l}(n,k)$ and it must take $X_{\alpha}$ to the corresponding vector in $\mathfrak{s}\mathfrak{l}(n,k).$ Consider $[X_{\alpha},X_{\beta}]$. This product is $T$-stable of $T$ weight $\alpha +\beta$ and, when the sum is in $\Phi$, it maps to the root vector of that weight in $\lsl(n,k)$. Hence it is equal to $X_{\alpha+\beta}$.  In a root system of type $A_n,$ a non-zero sum of two distinct roots is never a multiple of a root and all root strings are of length one.If $\alpha +\beta \notin \Phi,$ and it is not zero, $[X_{\alpha},X_{\beta}]$ would be of weight $\alpha +\beta$ and hence not a multiple of a root. All the weights which occur in $\ws{}$ are of the form $p^j\alpha$ for some $\alpha \in \Phi.$  Hence $\alpha +\beta$ does not occur as a weight in $\ws{}.$ Hence $[X_{\alpha},X_{\beta}]=0$. 
Now consider $[X_{\alpha},X_{-\alpha}]$. Its image in $\lsl(n,k)$ is the element $h_{\alpha}$ in the Lie algebra of the maximal torus. Hence we may write $[X_{\alpha},X_{-\alpha}]=h_{\alpha}+n$ for some $n\in \gamma_1$. Now $[X_{\alpha},X_{-\alpha}]$ is of $T$-weight $0$ and so $n$ is as well. Hence $h_{\alpha} +n$ is the Jordan decomposition of $[X_{\alpha},X_{-\alpha}]$. Consider $[[X_{\alpha},X_{-\alpha}],X_{\beta}]$. On the one hand this is $\check{\alpha}(\beta)X_{\beta}+[n,X_{\beta}]$. On the other it is of $T$-weight $\beta$ and $T$-stable and it maps to the $\check{\alpha}(\beta)$ multiple of the corresponding weight vector in $\lsl(n,k)$ and so $[X_{\alpha},X_{-\alpha}]=\check{\alpha}(\beta)X_{\beta}$. 

It follows that $\coprod_{\alpha \in \Phi}kX_{\alpha}\bigoplus \mathfrak{h}$ is a Lie subalgebra (not restricted) of $\ws{}.$ Denote it $\mathfrak{g}_0.$ The natural surjection of $\ws{}$ on $\mathfrak{s}\mathfrak{l}(n,k)$ restricts to an isomorphism of vector spaces on $\mathfrak{g}_0$ and so it is isomorphic to $\mathfrak{s}\mathfrak{l}(n,k)$ as a (nonrestricted) Lie algebra.\end{proof}

\begin{remark} Clearly, for large $p,\;\mathfrak{g}_0$ is a restricted Lie algebra. In fact for $p>3$, it would appear to be the case. I would guess that the only case of any difficulty is $p=2$.\end{remark}  
 
 By the observations above,
$\lie{T(\V)}=\mathfrak{h}\oplus \lie{A_1^{n-1}}$. Notice that it is Abelian. To write it explicitly, we write $\lie{A_1^{n-1}}=\prod_{\substack{1\leq i\leq n-1\\j>0}}kY_{i,j}$ with the formal $p$'th power $Y_{i,j}^{[p]}=Y_{i,j+1}$. Each $Y_{i,j}$ is of $T_k$-weight $0$. The following summarizes what can be said about $\ws{}$.

\begin{theorem}\label{W:rla}
Assume that $p>2$. Let $\mathfrak{g}=\mathfrak{s}\mathfrak{l}(n,k)$ endowed with the representation of $\SL(n,\V)$ obtained by composing the projection of $\SL(n,\V)$ on $\SL(n,k)$ with the adjoint representation. Let $\mathfrak{g}^{[p^r]}, r>0$ denote the $r$'th Frobenius twist of the representation $\mathfrak{g}.$ Let $\boldsymbol{\gamma}_j=\lie{\Gamma_j}$. Then the sequence of restricted Lie algebras, $\ws{} =\boldsymbol{\gamma}_0\supseteq \boldsymbol{\gamma}_1\cdots \supseteq \boldsymbol{\gamma}_j\supseteq\cdots$, is an exhaustive decreasing filtration of $\ws{}$ by ideals and a basis of neighborhoods of $0$ in a separated complete topology on $\ws{}$. Furthermore:
\begin{enumerate}
\item The associated graded of $\ws{}$ in this filtration, as a representation of $\SL(n,\V)$, is $\mathfrak{g}\oplus \coprod_{j\geq 1}\mathfrak{g}^{[p^j]}$.
\item  $[\ws{},\boldsymbol{\gamma}_j]=(0)$ for all $j\geq 1$.
\item Let $\mathfrak{g}_0$ be the lie algebra constructed above. Then $\mathfrak{g}_0\simeq \mathfrak{g}$ as an unrestricted Lie algebra and a representation.
\item The $T_k$-weight of $\mathfrak{h}$ and of the $Y_{i,j}$ is $0.$ 
\item The weight of $X_{\alpha}^j$ is $p^j\alpha$. 
\end{enumerate}
\end{theorem}
\begin{proof}
The discussion preceding the numbered items requires no proof. We begin with (1). It is clear that the quotient $\boldsymbol{\gamma}_j/\boldsymbol{\gamma}_{j+1}$ is just the Lie algebra of the quotient $\Gamma_j/\Gamma_{j+1}$. This is a linear algebraic group with an $\SL(n,k)$-action induced by conjugation. The elements of $\Gamma_j$ can be written in the form $1+\xi(M)^{p^{-j}}p^j+u_{j+1}$, where $M\in \mathfrak{s}\mathfrak{l}(n,k)$ and $u_{j+1}$ is a matrix all of whose entries are of valuation at least $j+1$ and the notation, while undefined, is self explanatory. Noting that $\Gamma_{j+1}$ admits a similar description with $j$ replaced by $j+1$, and passing to the quotient, it is clear that $\Gamma_j/\Gamma_{j+1}$ is isomorphic to the $j$'th Frobenius twist of the linear functions on the additive group of $k$-matrices of trace $0$. It follows that $\boldsymbol{\gamma}_j/\boldsymbol{\gamma}_{j+1}$ is isomorphic to $\mathfrak{g}^{[p^j]}$.

We proceed to (2) which is the key observation. Now $\SL(n,\V)\subseteq \GL(n,\V)$, and the same for the corresponding Lie algebras and so it suffices to establish this result for $\wg{}$, where the latter is the Lie algebra of the pro-$k$-group $\GL(n,\V)$. We write $\boldsymbol{\gamma}'_j,\: \Gamma'_j$ for the subobjects associated to $\GL(n,\V)$ parallel to the corresponding symbols for $\SL(n,\V)$. Write $\mathcal{M}^j$ for the $n\times n$ matrices congruent to $0$ modulo $p^{j-1}$.

First, suppose that $u\in\V$ is of valuation at least $j$ so that $u=\sum_{i\geq j}\xi(u_i)^{p^{-i}}p^i$. If $x=\sum_{i\geq 0}\xi(x_i)^{p^{-i}}p^i$, then the product $xu$ has as Witt coordinates polynomials with $u$ coefficients, but in the quantities $x_i^{p^j}$. This can be understood as a finite computation since the coefficients of a Witt vector at valuation $m$ are uniquely determined by the truncation modulo $p^{m+1}$. Thus we may assume that $u$ is a finite sum of terms $\xi(u_i)^{p^{-i}}p^i, \:i\geq j$. But then it suffices to consider terms of the form $x(\xi(u_i)^{p^{-i}}p^i)$ as the coefficients of $xu$ will be polynomials in the Witt coefficients of these elements. But $x(\xi(u_i)^{p^{-i}}p^i)=\sum_{r\geq 0}\xi(x_r^{p^i}u_i^{p^r})^{p^{-r-i}}p^{r+i}$. That is, the coefficients of $x(\xi(u_i)^{p^{-i}}p^i)$ are polynomials in the $x_r^{p^i},\;i\geq j$, that is to say, they are polynomials in the $x_r^{p^j}$.

Now suppose that $U=(u_{r,s})\in \mathcal{M}^j$ and that $X=(x_{r,s}),\, Y=(y_{r,s})$ are in $\mathcal{M}^0$.
Then the $(l,m)$-entry of $XUY$ is $\sum x_{l,r}y_{s,m}u_{r,s}$. It is hence a polynomial in $p^j$'th powers of the entries of $X$ and $Y$. In particular $XUX^{-1}$ has as its Witt coefficients rational functions in the $p^j$'th powers of the Witt coefficients of $X$. Consider the action of $\ws{}$ on the coordinate ring of $\Gamma_j$. First note that $\Gamma_j=1+\mathcal{M}^j$. If $U\in \mathcal{M}^j$ is a matrix with indeterminate Witt coefficients and $X$ is a general element of $\GL(n,\V)$, then the Witt coefficients of $X(1+U)X^{-1}$ are polynomials in $p^j$'th powers of the entries of $X$. Hence if a tangent vector at $e$ in $\GL(n,\V)$ is applied to such an expression by its operation on the $x_{i,j,m}$, the result will be null in all cases. That is, the operation of $\ws{}$ on the coordinate ring of $\Gamma_j$ is trivial. Since the adjoint action of $\ws{}$ on $\boldsymbol{\gamma}_j$ is the differential of its action on the coordinate ring of $\Gamma_j$, it follows that $[\ws{},\boldsymbol{\gamma}_j]=0$ for all $j\geq 1$. This establishes (2).

To prove (3) first consider all cases where the characteristic is at least three or when the cnaracteristic is at least two. By the computations above and by item (2), 
\[
[\ws{},\ws{}]\simeq [\ws{}/\boldsymbol{\gamma}_1,\ws{}/\boldsymbol{\gamma}_1]\simeq [\mathfrak{g}_0,\mathfrak{g}_0]\simeq \mathfrak{g}_0.
\]

Hence $\mathfrak{g}_0=[\ws{},\ws{}].$ This is clearly stable under the adjoint representation. For $\mathfrak{s}\mathfrak{l}(2,k)$ when $k$ is of characteristic two one establishes the same result by an awkward hand computation. 
\end{proof}

Before proceeding we need a certain analysis of central extensions of restricted Lie algebras. Let $0\rightarrow \mathfrak{m}\rightarrow \widetilde{\mathfrak{g}}\overset{\phi}\rightarrow \mathfrak{g}\rightarrow 0$ denote a central extension of restricted Lie algebras which is trivial as an extension of Lie algebras. I will call such an extension \emph{Lie trivial} and we shall often denote it $\phi$.

If $\phi:\widetilde{\mathfrak{g}}\rightarrow \mathfrak{g}$ is an extension of $\mathfrak{g}$ by $\mathfrak{m}$, there is a section $\sigma:\mathfrak{g}\rightarrow \widetilde{\mathfrak{g}}$ which is just a Lie algebra morphism. 

\begin{lemma}\label{W:rc}
Let $\phi:\widetilde{\mathfrak{g}}\rightarrow \mathfrak{g}$ be a Lie trivial central extension of $\mathfrak{g}$ by $\mathfrak{m}$. Let $\sigma:\mathfrak{g}\rightarrow \widetilde{\mathfrak{g}}$ denote a Lie algebra section of $\phi$. For $x\in
\mathfrak{g}$, let $\beta(x)=\sigma(x)^{[p]}-\sigma(x^{[p]})$. Then $\beta$ is a $p$-linear map from $\mathfrak{g}$ to $\mathfrak{m}$.
\end{lemma}
\begin{proof}
That the values of $\beta$ lie in $\mathfrak{m}$ follows from the facts that $\phi$ is a restricted Lie algebra morphism and that $\phi \circ \sigma= \id_{\mathfrak{g}}$. It is clear that it is 
$p$-linear. We prove that it is additive.
\begin{align}\label{rl:b}
\beta(x+y)-\beta(x)-\beta(y)&=\sigma(x+y)^{[p]}-\sigma((x+y)^{[p]})\\
&\quad -\sigma(x)^{[p]}+\sigma(x^{[p]})-\sigma(y)^{[p]}+\sigma(y^{[p]}) \notag \\
&=\sigma(x)^{[p]}+\sigma(y)^{[p]}+\sum_{i=1}^{p-1}s_i(\sigma(x),\sigma(y)) \notag \\ 
&\quad -\sigma(x^{[p]}+y^{[p]}+\sum_{i=1}^{p-1}s_i(x,y) ) \notag \\ &\qquad -\sigma(x)^{[p]}+\sigma(x^{[p]})-\sigma(y)^{[p]}+\sigma(y^{[p]}) \notag \\
&=\sum_{i=1}^{p-1}s_i(\sigma(x),\sigma(y))-\sum_{i=1}^{p-1}\sigma(s_i(x,y)). \notag
\end{align}
By Equation~\eqref{lp}, the final expression of \eqref{rl:b} is null since for each $i$, $s_i(\sigma(x),\sigma(y))=\sigma(s_i(x,y))$. That is, $\beta(x+y)=\beta(x)+\beta(y)$.
\end{proof}

What remains is to determine the automorphisms of a Lie trivial extension. To this end we define the \emph{$p$-center} of a restricted Lie algebra:
\begin{definition}
Let $L$ be a restricted Lie algebra. The $p$-\emph{center} of $L$ is the set $\{x\in L:\; [x,L]=0, x^{[p]}=0\}$.  Write $L_p$ for the $p$-center of $L$. When $L$ is Abelian it is called strongly Abelian if $L=L_p$.
\end{definition}

Having made this definition, we leave to the reader the exercise of proving that two $p$-linear maps $\beta$ and $\beta'$ from $\mathfrak{g}$ to $\mathfrak{m}$ yield isomorphic Lie trivial restricted extensions if and only if they differ by a map into the $p$-center of $\mathfrak{m}$. This is a slightly eccentric description  of the set of extensions (see \cite{gh}). Hochschild alters central extensions by a (nonrestricted) Lie cocycle so that he can regard all central extensions as extensions by strongly Abelian centers but in our case the center is endowed with an important canonically defined restricted structure.
\begin{theorem}\label{T:ad}
Assume that $p>2$. Notation is as in Theorem~\ref{W:rla}. Then there is an $\SL(n,\V)$-equivariant Lie (i.e. unrestricted) splitting $\sigma: \mathfrak{g}\rightarrow \ws{}$. The map $\beta$ associated to $\sigma$ as in Lemma~\ref{W:rc} is an $\SL(n,\V)$-equivariant morphism. Moreover,
\begin{enumerate}
\item The map $\beta$ is injective so that it can be viewed as an $\SL(n,\V)$-morphism from Frobenius twist $\mathfrak{g}^{[p]}$ into $\ws{}$.
\item There are canonical nontrivial $T_k$-weight spaces of $\ws{}$ which permit us to write:
\[
\ws{}=\mathfrak{h}\oplus  \prod_{\substack{1\leq i\leq n-1\\j>0}}kY_{i,j} \oplus \prod_{\substack{\alpha\in \Phi\\j\geq 0}}kX_{\alpha}^j.
\]
\item As a nonrestricted Lie algebra or purely as the adjoint representation of $\SL(n,\V)$, we may represent $\ws{}$ as the infinite product:
\[
\ws{}=\prod_{j\geq 0}\mathfrak{g}^{[p^j]}.
\]
\end{enumerate}
\end{theorem}
\begin{proof}
First recall how $\sigma$ is constructed. The natural surjection $\phi:\ws{}\rightarrow \mathfrak{g}$ restricts to an isomorphism on the derived algebra $[\ws{},\ws{}]=\mathfrak{g}_0$. The derived algebra is $\SL(n,\V)$-stable and the surjection is equivariant. Hence its inverse, that is to say, $\sigma$ is equivariant. The $p$'th power map, which is not linear, is, however, equivariant. By definition, $\beta(x)=x^{[p]}-\sigma(x^{[p]})$. Hence it is a difference of equivariant maps. This establishes the first statement of the theorem.

We proceed to the first numbered item of the theorem. Let $\alpha$ be any root of $\SL(n,\V)$. Let $U_{\alpha,\V}$ denote the corresponding root subgroup of $\SL(n,\V)$. Write $\mathfrak{u}_{\alpha}(\V)$ for its Lie algebra viewed as a Lie subalgebra of $\ws{}$. It admits a basis of elements $\{X_{\alpha,i}:\: i\geq 0\}$ where $X_{\alpha,i}^{[p]}=X_{\alpha,i+1}$. Now the derived algebra $\mathfrak{g}_0=\operatorname{im}(\sigma)$ contains only $X_{\alpha,0}$. Moreover in the classical Lie algebra $\operatorname{sl}(n,k)$ for the corresponding root vector, $X_{\alpha}^{[p]}=0$. It follows that $\sigma(\phi(X_{\alpha,0})^{[p]})=0$ and that $\beta(X_{\alpha,0})=X_{\alpha,1}$. Hence the image of $\sigma(\mathfrak{g})$ under $\beta$ is an $\SL(n,\V)$-module containing the $p$-twisted root vectors $X_{\alpha,1}$. The only $\SL(n,\V)$-submodule of $\mathfrak{g}^{[p]}$ containing these elements is $\mathfrak{g}^{[p]}$ itself. This establishes (1).

The last two items are self evident.
\end{proof}

We may use these results to describe the K\"ahler differentials on the smooth locus of $\Lrn$ as well as the canonical bundle on this homogeneous space. Since $\Lrn$ is locally a complete intersection the canonical bundle is a line bundle and so determined by its restriction to $\Lrn.$This is nothing but a straightforward application of the computations of the last section. The basic formula for the differentials of a homogeneous space is the following. Let $G$ and $H$ be respectively a linear algebraic group and a closed subgroup. Let the Lie algebra of $G$ (respectively $H$) be $\mathfrak{g}$ (respectively $\mathfrak{h}$). Then $\mathfrak{g}/\mathfrak{h}$ is a representation of $H$ under the adjoint action and the tangent bundle is the induced bundle of this representation. Consequently, $\Omega_{(G/H)/k}$ is the bundle induced by the dual, $(\mathfrak{g}/\mathfrak{h})^*.$ The computation of the canonical bundle follows easily.

\begin{proposition}\label{cang} The canonical bundle on $\Lrn$ is a line bundle. It is isomorphic to $\mathcal{L}(-n\frac{p^{nr}-1}{p-1}\lambda_1).$ 
\end{proposition}

\begin{proof}By the main theorem $\Lrn$ is locally a complete intersection. Hence by, for example, Theorem 7.11 on p. 245 of \cite{ag} it is a line bundle and so a tensor power of $\mathcal{L}(\lambda_1).$ We may compute it on the smooth locus which is the  homogeneous space $\SL(n,\V)/P_r.$ It is the highest exterior power of the sheaf of differentials. The differentials are the bundle induced by the linear dual of $\ws{}/\mathfrak{P}_r.$ This linear dual admits a $P_r$-stable filtration so that the successive quotients are $\boldsymbol{\omega}_0, \boldsymbol{\omega}_0^{[p]},\dots ,\boldsymbol{\omega}_0^{[p^{nr-1}]}.$ 

To see this let $P(j),\;j=0,1,\dots ,nr-1$ denote the set of matrices 
$\left(\begin{smallmatrix}x_{1,1}&x_{1,2}\\x_{2,1}&x_{2,2}\end{smallmatrix}\right)$ in $\SL(n,\V/p^{nr}\V)$ consisting of elements for which $x_{1,2}$ is an $n-1$ row with entries in $p^j\V/p^{nr}\V.$ Now these are all normal in $\SL(n,\V/p^{nr}\V)$ and the quotient 
$P(j)/P(j-1)$ is isomorphic as a group to the additive group $(p^j\V/p^{j+1}\V)^{n-1}.$ (In this notation $P_r$ would just be $P(nr)$ and 
$P(0)=\SL(n,\V/p^{nr}\V).$ Write $\mathfrak{P}(j)$ for the Lie algebra of $P(j).$ We proceed to compute the action of $P_r$ induced by the adjoint action of the stabilizer of the distinguished coset $P_r$ on the quotients $\mathfrak{P}(j)/\mathfrak{P}(j+1).$

Let an arbitrary element of $\SL(n,\V/p^{nr}\V)$ be represented by the matrix $U=\left(\begin{smallmatrix}u_{1,1}&u_{1,2}\\u_{2,1}&u_{2,2}\end{smallmatrix}\right)$ with $u_{1,1}\in \V/p^{nr}\V$ and the rest accordingly. An element of $P_r$ is of the form $X=\left(\begin{smallmatrix}x_{1,1}&0\\x_{2,1}&x_{2.2}\end{smallmatrix}\right).$ Its inverse is $$\left(\begin{matrix}x_{1,1}^{-1}&0\\-x_{1,1}^{-1}x_{2,2}^{-1}x_{2,1}&x_{2,2}^{-1}\end{matrix}\right).$$ Conjugating $U$ by $X,$ the $(1,2)$ entry of the conjugate of $U$ by $X$ is the expression $x_{1,1}u_{1.2}x_{2,2}^{-1}.$ Over $\V/p^{nr}\V$ this is the contragredient of the standard representation of $x_{2,2}\in \GL(n-1,\V/p^{nr}\V)$ twisted by the inverse of its determinant. 
The result of the conjugating action of $\left(\begin{smallmatrix}x_{1,1}&x_{1,2}\\x_{2,1}&x_{2,2}\end{smallmatrix}\right)$ on $\overline{u}\in P(j)/P(j+1)$ is $x_{1,1}\overline{u}x_{2,2}^{-1}.$ Now $\overline{u}\in (p^j\V/p^{j+1}\V )^{n-1}$ and we know that for $t\in k^*$ and a Witt vector in $p^j\V$ the result of multiplication on the initial term is $\xi(t)\cdot \xi(a)^{p^{-j}}p^j=\xi(t^{p^j}a)^{p^{-j}}p^j.$ The quotient of groups $P(j)/P(j+1)$ is just the additive group $p^j\V/p^{j+1}\V)^{n-1}.$ Consequently the result of conjugation by an element of $P_r$ on this quotient is just the $p^j$'th Frobenius twist of the contagredient of the standard representation twisted by the inverse of the determinant. The Lie algebra of this additive group will be the same additive group and the result of the restriction of $\operatorname{ad}(X)$ to this quotient of Lie algebras will just be the same representation. To state this precisely let $\overline{x}_{i,j}$ denote the residue class of $x_{i,j}$ modulo $p$ and let $\overline{u}$ denote an $n-1$ row with entries in $k.$ Let $\left(\begin{smallmatrix}x_{1,1}&0\\x_{2,1}&x_{2,2}\end{smallmatrix}\right)$ operate on this row by the formula $\overline{x}_{1,1}\overline{u}\overline{x}_{2,2}^{-1}.$ This defines a representation which we write $\mathfrak{u}.$ The adjoint representation of $P_r$ on $\ws{}/\mathfrak{P}_r$ thus admits a filtration with associated graded $\coprod_{j=0}^{nr-1}\mathfrak{u}^{[p^j]}.$ Hence the associated graded of $(\ws{}/\mathfrak{P}_r)^*$ in the filtration dual to the filtration above is just $\coprod_{j=0}^{nr-1}(\mathfrak{u}^{[p^j]})^*.$ Finally $ (\mathfrak{u}^{[p^j]})^*= \boldsymbol{\omega}_0^{[p^j]}.$

The top exterior power of $(\ws{}/\mathfrak{P}_r)^*$ induces the top exterior power of the bundle induced by this representation. This top exterior power is the tensor product of the top exterior powers of the terms in its associated graded. Now $\boldsymbol{\omega}_0$ is the representation of $P_r$ corresponding to the lower right $n-1\times n-1$ minor acting on its standard $n-1$-dimensional representation twisted by the class of the determinant of $x_{2,2}$ modulo $p.$ As a character on $P_r$ the determinant of $x_{2,2}$ is the inverse of the residue class of $x_{1,1}$ and this is the first fundamental weight of $\SL(n,\V).$ The top exterior power of an irreducible is trivial so that the weight of the twisted representation is just $-n\lambda_1.$  The Frobenius twist $\boldsymbol{\omega}_0^{[p^s]}$ corresponds to $p^sn\lambda_1.$  The weight of the full tensor product of the top exterior powers of the terms of the associated graded is $-n(1+p+p^2+\cdots +p^{nr-1})\lambda_1=-n\frac{p^{nr}-1}{p-1}\lambda_1.$ Hence in the usual convention for bundles induced by characters, this is the line bundle induced by the weight in the assertion above.
\end{proof}

Let us consider for a moment the infinite Grassmannian over $k.$ It has been extensively studied. It may be thought of as the space parametrizing special $k[[t]]$ lattices of rank $n$ in $k((t))^n.$
It is the homogeneous space over $\SL(n,k((t))$ with stabilizer equal to the maximal parahoric $\SL(n, k[[t]]).$ The generalized Schubert cell corresponding to the diagonal matrix with diagonal entries $t^{(n-1)r}, t^{-r}, \dots ,t^{-r}$ is a particular Schubert cell which we will denote $\mathbb{L}at_r^n(k((t))).$ There is a canonical ample generator of the Picard group corresponding to a character on the subgroup of $\SL(n,k[[t]])$ whose $1,1$ position is nonzero but in which the rest of the first column is zero. Denote that character $\tilde{\lambda}_1.$ The constructions of this section all apply to that situation. We refer to \cite{pic} for computations. 
\begin{corollary}The scheme $\Lrn$ is not isomorphic to the scheme $\mathbb{L}at_r^n(k((t))).$
\end{corollary}           
 \begin{proof} On $\mathbb{L}at_r^n(k((t)))$ the computations of Proposition \ref{cang} would imply that the canonical bundle was the $-n$'th power of the ample generator of the Picard group in contradistinction to \ref{cang}. Since both the ample generator of $\operatorname{Pic}$ as well as the canonical bundle are uniquely determined functorial objects which would be carried one to another by any isomorphism, an isomorphism is impossible.
\end{proof}

\appendix

\section{Deformations of lattices in $\Lrn$} 

 In \cite{me} there is a computation of the tangent space to a lattice in $\Lrn$. The statement there is not particularly well adopted to the computations in this paper. For this reason we have expanded the proofs and computations of \cite{me} and reformulated them for application in this paper. We conclude with an explicit construction of a class of deformations sufficient to show that $\Lrn$ is a variety. 
\subsection{Infinitesimal deformations of schemes of lattices over\\ schemes of rings}We wish to classify certain infinitesimal deformations of polynomial algebras. Let $M$ be affine $N$-space and let $S$ be a closed subscheme isomorphic to affine $R$-space. Let $k[M]$ and $k[S]$ be their respective coordinate rings. Let $J$ be the ideal defining $k[S]$. We may assume it to be generated by a system of parameters $y_1,\dots ,y_{N-R}$. We wish to give a complete description of the set of flat infinitesimal deformations of $S$ in $M$. Write $\widetilde{M}$ for the base extension of $M$ by $k[\epsilon]$, the dual numbers, and write $k_{\epsilon}[\widetilde{M}]$ for its coordinate ring. A flat deformation of $S$ will be a closed $\ke$-subscheme of $\widetilde{M}, \; \widetilde{S}$, with coordinate ring $\ke[\widetilde{S}]$ which is flat over $\ke$ and which maps surjectively onto $k[S]$. Let $\widetilde{J}$ be the kernel of the map $\ke[\widetilde{F}]\mapsto \ke[\widetilde{S}]$. If $z_1, \dots ,z_R$ is a set of indeterminates such that $k[S]=k[z_1,\dots ,z_R]$ then lifting these indeterminates to $\tilde{z}_1, \dots ,\tilde{z}_R$, flatness implies that the monomials in the $\tilde{z}_j$ are a $\ke$ basis of $\ke[\widetilde{S}]$. Hence there is a section $s:k[S]\rightarrow \ke[\widetilde{S}]$ and a natural isomorphism of $\ke$-algebras $\ke[\widetilde{S}]\simeq \ke\otimes_kk[S]$. The trivial deformation $\tilde{S}_{\epsilon}$ which is the quotient of $\ke[\widetilde{M}]$ by the ideal $\widetilde{J}_0=J\ke[\widetilde{M}]$. In consequence of these observations, there is a commutative diagram:

\begin{equation}\label{defd}
\begin{CD}
{} @. 0 @. 0 @. 0 @. {}\\
@. @VVV  @VVV  @VVV @.\\
0@>>>\epsilon J @>>> \epsilon k[M] @>>> \epsilon k[S] @>>>0\\
@. @VVV @VVV @VVV @.\\
0@>>>\widetilde{J} @>>>  \ke[\widetilde{M}] @>>> \ke[\widetilde{S}]@>>> 0\\
@.  @V\pi_0 VV  @V\pi VV  @VV\pi_1V  @.\\
0@>>> J @>>> k[M] @>>> k[S] @>>> 0\\
@. @VVV @VVV @VVV @.\\
{}@. 0 @. 0 @. 0 @. {}
\end{CD}
\end{equation}

\begin{proposition}\label{def} There is a bijective correspondence between infinitesimal deformations of $\widetilde{J}_0$ and sections in the normal sheaf $\mathcal{N}_{S/N}=\mathcal{H}\it{om}_{\V_M}(J/J^2, k[S])$. The correspondence sends the section $\delta \in \mathcal{N}_{S/N}$ to the ideal $\widetilde{J}_{\delta}=\{a+b\epsilon \in \ke[\widetilde{M}]:a\in J \quad b\equiv \delta(a)\operatorname{mod} (J)\}$. The conormal vector $\delta$ is called the classifying map of $\widetilde{J}_{\delta}$. 
\end{proposition}
\begin{proof} This is a standard result. We include some discussion to recall for the reader some detail with which we will work below. Since $\pi_0$ is the restriction of $\pi$ to $\widetilde{J}$, for each $u\in J$ there is some element $u_1\in k[M]$ such that $u+u_1\epsilon \in \widetilde{J}$. Then $u_1$ is determined modulo $J$ because $\operatorname{ker}\pi_0 =\epsilon J$. Choose a $k$-vector space section to $\pi_0$ and write it $s(u)=u+s_1(u)\epsilon$. Let $\delta(u)$ be the residue class of $s_1(u)$ in $k[S]=k[M]/J$. Clearly the ideal $\widetilde{J}$ is the sum of the vector spaces $\epsilon J$ and $s(J)$. The reader may deduce the remainder of the proof from these observations.\end{proof}

We now recall the notation of the earlier sections of this paper. Let $\V_r=\V/p^{nr}\V$ and let $F_r=\V_r^n$. Then $\Lrn$ is isomorphic to the subgroupschemes of $F_r$ of codimension $nr$ which are $\V_r$ submodules. Then we propose to describe the tangent space to $\Lrn$ at the lattice $L$ by applying proposition \ref{def} to three pairs of algebras, $k[F_r]$ and $k[L]$, then $k[F_r]\otimes_kk[F_r]$ and $k[L]\otimes_k k[L]$ and finally to $k[F_r]\otimes_kk[\V_r]$ and $k[L]\otimes_kk[\V_r]$. 

Let $\alpha:k[F_R]\rightarrow k[F_r\otimes_k[F_r]$ be the map of coordinate rings dual to addition and let $\mu:k[F_r]\rightarrow k[F_r]\otimes_kk[\V_r]]$ be the map corresponding to scalar multiplication on $F$ by $\V_r$. We will use the same symbols to denote these operations on all subobjects quotient objects and base extensions as well. Fix an $\V_r$submodule of $F_r$ of codimension $nr$ and write it $L$. Let the ideal defining it be $I$. Fix a normal vector $\delta \in \mathcal{N}_{L/F}$ and let $\widetilde{I}_{\delta}$ be the ideal corresponding. We wish to give sufficient conditions for $\widetilde{I}_{\delta}$ to be an ideal defining an $\V_r$ submodule scheme over $\ke$. We leave it to the reader to interpret our extended notation. It will define a subgroup scheme if: 

\begin{equation}\label{sgp}
\alpha(\widetilde{I}_{\delta})\subseteq \widetilde{I}_{\delta}\otimes_{\ke}\ke[\widetilde{F}_r]+ \ke[\widetilde{F}_r]\otimes_{\ke}\widetilde{I}_{\delta}
\end{equation}. 

It will define an $\V_r$ submodule if:

\begin{equation}\label{smod}
\mu(\widetilde{I}_{\delta})\subseteq \widetilde{I}_{\delta}\otimes_{\ke}\ke[\widetilde{\V}_r]
\end{equation}

Now consider $\mathfrak{I}=I\otimes k[F]+k[F]\otimes I\subseteq k[F]\otimes k[F]$. Using, for example, the fact that $I$ is generated by a system of parameters it is easy to see that $\mathfrak{I}/\mathfrak{I}^2 \simeq (I/I^2\otimes k[L])\oplus (k[L]\otimes I/I^2)$. In the case of $I\otimes_kk[F]$ it follows from nothing more than right exactness of the tensor product that $(I\otimes_kk[\V_r])/(I\otimes_kk[\V_r])^2=(I/I^2)\otimes_kk[F]$.

 One further ingredient in the following consists of natural maps on $I/I^2$ induced by $\alpha$ and $\mu$. First consider $\mathfrak{I}=I\otimes k[F]+k[F]\otimes I\subseteq k[F]\otimes k[F]$. Using, for example, the fact that $I$ is generated by a system of parameters it is easy to see that:

\begin{align}
\mathfrak{I}&/\mathfrak{I}^2 \simeq (I/I^2\otimes k[L])\oplus (k[L]\otimes I/I^2)\\ \mathfrak{I}&= I\otimes k[F]+k[F]\otimes I \notag
\end{align}

In the case of $I\otimes k[\V]_r$ it is clear that $(I\otimes k[\V]_r)/(I^2\otimes k[\V]_r)=(I/I^2)\otimes k[\V_r]$. Then since $\alpha$ takes $I$ to $I\otimes k[F]+k[F]\otimes I$ and $\mu$ takes $I$ to $I\otimes k[\V_r]$ there are natural maps of conormal bundles:

\begin{align}
\alpha^N:&I/I^2\mapsto (I/I^2)\otimes k[L]+k[L]\otimes (I/I^2)\\
\mu^N:&I/I^2\mapsto (I/I^2)\otimes k[L]
\end{align}

\begin{proposition}\label{coa}The classifying map for $\widetilde{I}_{\delta}\otimes_{\ke}\ke[\widetilde{F}]+\ke[\widetilde{F}]\otimes_{\ke}\widetilde{I}_{\delta}$ as an infinitesimal deformation of $I\otimes_kk[F]+k[F]\otimes_k I$ is $\delta^2=(\delta\otimes \id, \id \otimes \delta)$. The classifying map for $\widetilde{I}\otimes_{\ke}\ke[\widetilde{F}]$ is $\delta \otimes \id$.
\end{proposition}  

\begin{proof}It is clear that the classifying map for $\widetilde{I}_{\delta}\otimes_{\ke}\ke[\widetilde{F}_r]$ is none other than $\delta \otimes \id:I/I^2\otimes k[\V_r]\mapsto k[L]\otimes k[\V_r]$. Hence all that remains to be proven is the first assertion of the proposition.

Each element of $I\otimes k[F]+k[F]\otimes I=\mathfrak{I}$ is a sum of terms $u\otimes a+b\otimes v$ with $u,v\in I$.There are $u+m\epsilon$ and $v+n\epsilon$ in $\overline{I}_{\delta}$ lying above $u$ and $v$ respectively with $m\equiv \delta(u)\;\operatorname{mod}I$ and $n\equiv \delta(v)\;\operatorname{mod}I$. Consequently $(u+m\epsilon)\otimes a+b\otimes (v+n\epsilon)=u\otimes a+b\otimes v+(m\otimes a+b\otimes n)\epsilon$ is in $\widetilde{I}_{\delta}\otimes_{\ke}\ke[\widetilde{F}]+\ke[\widetilde{F}]\otimes_{\ke}\widetilde{I}_{\delta}$. Hence by Proposition \ref{def}  the class of $m\otimes a+b\otimes n\;\operatorname{mod}\mathfrak{I}^2$ is the value of the classifying map for $\widetilde{I}_{\delta}\otimes_{\ke}\ke[\widetilde{F}]+\ke[\widetilde{F}]\otimes_{\ke}\widetilde{I}_{\delta}$ on $u\otimes a+b\otimes v$. Applying equation (A.4)this is the value of $(\delta \otimes \id, \id \otimes \delta)$ on the class of $u\otimes a+b\otimes v$ in $(I/I^2)\otimes k[L]\oplus k[L]\otimes (I/I^2)$.
\end{proof}

\begin{proposition}\label{cd}Let $I$ define the lattice $L$ in $F$ and suppose that $\widetilde{I}_{\delta}$ is an infinitesimal deformation of $I$ with classifying map $\delta$. Then $\widetilde{I}_{\delta}$ defines a $\ke$ lattice in $\widetilde{F}$ if and only if the following diagrams commute:

\begin{equation}\label{ad}  
\begin{CD}
I/I^2@>\delta>>k[L] \\
@VV\alpha^NV     @VV\alpha V \\
(I/I^2)\otimes k[L]\oplus k[L]\otimes (I/I^2)@>(\delta \otimes \id,\id \otimes \delta)>>k[L]\otimes k[L]
\end{CD}
\end{equation}

\begin{equation}\label{mul}
\begin{CD}
I/I^2@>\delta>>k[L] \\
@VV\mu^NV      @VV\mu V \\
(I/I^2)\otimes k[\V_r]@>\delta \otimes \id>>k[L]\otimes k[L]
\end{CD}
\end{equation}

Written as equations these two diagrams assert that:

\begin{align}
\alpha \circ \delta &=(\delta \otimes \id, \id \otimes \delta)\circ \alpha^N \\
\mu \circ \delta &= (\delta \otimes \id)\circ \mu^N
\end{align}

\end{proposition}

\begin{proof} These diagrams are essentially diagrammatic representations of the assertions in Proposition \ref{coa}. Consider diagram \ref{ad}. By \ref{sgp} the classifying map $\delta$ will define a $\ke$-subgroup of $\widetilde{F}$ if $\alpha$ carries $\widetilde{I}_{\delta}$ into $\widetilde{I}_{\delta}\otimes_{\ke}\ke[\widetilde{F}_r]+ \ke[\widetilde{F}_r]\otimes_{\ke}\widetilde{I}_{\delta}$. Suppose $u+m\epsilon \in \widetilde{I}_{\delta}$. Then $u\in I$ and $m\equiv \delta(u)\operatorname{mod}I$. Now $\alpha(u+m\epsilon)=\alpha(u)+\alpha(m)\epsilon$. By Proposition \ref{coa} the classifying map for $\widetilde{I}_{\delta}\otimes_{\ke}\ke[\widetilde{F}_r]+ \ke[\widetilde{F}_r]\otimes_{\ke}\widetilde{I}_{\delta}$ is $(\delta \otimes \id,\id \otimes \delta)$. Hence $\alpha(u+m\epsilon)$ is in $\widetilde{I}_{\delta}$ if $\alpha(m)\equiv (\delta \otimes \id,\id \otimes \delta)(\alpha(u)) \operatorname{mod} \mathfrak{I}^2$. But $\alpha(u)$ reduces to $\alpha^N(\overline{u})$ and the image of $\alpha(u)$ under $(\delta \otimes \id,\id\otimes \delta)$ depends only on its class modulo $\mathfrak{I}^2$. The argument for diagram \ref{mul} is similar but the simpler nature of the maps involved makes it rather obvious.
\end{proof}  

Since $L$ is a normal algebraic subgroup of $F$ it can be represented as the fiber over $(0)$ in the projection $F\mapsto F/L$. Since $F/L$ is isomorphic to affine $nr$ space there is a system of parameters defining $(0)$ in $F/L$ which can be viewed as a system of parameters defining $L$ in $F$. That is there are $nr$ functions defining $0$ in $F/L$ which can be viewed as a system of parameters defining $L$ in $F$. They are invariant under both left and right translation by elements of $L$ and their representatives in $I/I^2$ can be viewed as invariant sections. (The $L$ bundle $I/I^2$ is $L$ homogeneous on $L$.) We will call such a system of parameters an {\it invariant system} of parameters defining $L$.

We wish to identify certain special invariant systems of parameters. To this end we apply the structure theory of finite modules over principal ideal domains. View the lattices as $\V$-free modules. Then if $L\subseteq F$ is a lattice of corank $nr$ then there is an ordered $F$-basis $f_1,\dots ,f_n$ and integers $\nu_1 \geq \nu_2 \geq \dots \geq \nu_n$ such that $\sum_{i=1}^n\nu_i=nr$ so that $L$ is the lattice with basis $p^{\nu_1}f_1, p^{\nu_2}f_2, \dots ,p^{\nu_n}f_n$. A typical element of $F$ can be written $\sum_{i=1}^n \sum_{j=0}^{\infty}\xi(x_{i,j})^{p^{-j}}p^jf_j$. It is clear that in this description of $F$ and $L$ a set of equations defining $L$ is:

\begin{equation}\label{psp}
x_{ij}=0\quad 0\leq j< \nu_i
\end{equation}

A system of parameters defining $L$ which is of this type will be referred to as a {\it principal invariant system} of parameters. For simplicity we will refer to the integers $\nu_1, \nu_2, \dots , \nu_n$ as the principal divisors of the lattice $L$. We wish to understand the maps $\alpha^N$ and $\mu^N$.

First we will consider $\alpha^N$. Since $\alpha$ is a subgroup under addition, $\alpha(I)\subseteq I\otimes k[F]+k[F]\otimes I$ and so composing $\alpha$ with the projection $\operatorname{id}\otimes \pi$ where $\pi:k[F]\mapsto k[L]$ is the restriction, we see that there is a natural coaction $\alpha^r:I\mapsto I\otimes k[L]$ corresponding to right translation. Taking commutativity into account, there is a symmetrically constructed map $\alpha^{\ell}:I\mapsto k[L]\otimes I$ corresponding to left translation by the negative of an element. (Recall that left translation usually involves the inverse.) That is negative left translation is the map $(\alpha^{\ell}_x(f))(m)=f(x+m)$ and since the group is commutative it is the same up to order as right translation.

We recall some basic facts concerning Witt polynomials. They are defined by the equations:
 
\begin{align}
(\sum_{i=0}^{\infty}&\xi(x_i)^{p^{-i}}p^i)+(\sum_{i=0}^{\infty}\xi(y_i)^{p^{-i}}p^i)&=\sum_{i=0}^{\infty}\xi(\Phi_i(x_0,\dots ,x_i;y_0,\dots y_i))^{p^{-i}}p^i \label{wap}\\
(\sum_{i=0}^{\infty}&\xi(x_i)^{p^{-i}}p^i)\cdot(\sum_{i=0}^{\infty}\xi(y_i)^{p^{-i}}p^i)&=\sum_{i=0}^{\infty}\xi(\Psi_i(x_0,\dots ,x_i;y_0,\dots y_i))^{p^{-i}}p^i \label{wmp}
\end{align}

In these equations it is important to keep track of weight. The indeterminates $x_i$ and $y_i$ are assigned weight $p^i$. Then $\Phi_i$ is of total weight $p^i$ That is each monomial in it in the $x_j$ and $y_j$ is of total weight $p^i$ in the two sets of indeterminates together. They are called the additive polynomials.

 The $\Psi_i$, known as the multiplicative polynomials are of total weight $p^i$ in the $x_j$ and the $y_j$ separately. That is every monomial is of weight $p^i$ in the $x_j$ and then separately in the $y_j$.

The group of additive homomorphisms of the lattice $L$ into $\mathbb{G}_{a,k}$ is an $\V$-module under the contragredient action $(a\cdot f)(x)=f(ax)$. 

\begin{definition}\label{ach} An additive character of the lattice $L$ is an algebraic additive homomorphism of $L$ into $\mathbb{G}_{a,k}=k^+$ which is $\V$-linear for the action on $k$ induced by the natural homomorphism $\V\mapsto k$. The set of additive characters on $L$ is written $\mathcal{A}^+(L)$. It is a group but it is also an $\V$-module under the contragredient action. \end{definition}

Notice that $\mathcal{A}^+(L)$ the group of additive characters is  a vector space over $k$ under the standard multiplication $(af)(x)=a(f(x))$. Since $p\V$ acts as zero on $\mathcal{A}^+(L)$ the contragredient $\V$ action endows $\mathcal{A}^+(L)$ with another vector space structure. Hence it is actually a $k,k$ bimodule. Its decomposition into a direct sum $L=\coprod_{i=1}^n\V/p^{\alpha_i}\V$ is actually a bimodule decomposition making it into a vector space of dimension $nr$ in two different ways.

The conormal bundle $I/I^2$ also has a bimodule structure. The left $k$ structure is simply its $k$-vector space structure as a quotient of $k$ vector subspaces of the $k$-algebra $k[F]$. For the second note that $\V$ operates on $k[F]$ by the action induced by its action on the $\V$-module $F$. In functional notation the action can be written $(a\cdot f)(x)=f(ax)$. This makes $k[F]$ an $\V$-module and since $L$ is an $\V$-submodule, both $I$ and $I^2$ are $\V$-submodules. Since $p$ annihilates $k[F]$ this makes $k[F],\;I$ and $I^2$ all into $\V/p\V=k$-vector spaces but the structure is not the standard one. To understand it note that the $\V$-lattice $F$ is a scheme of modules over the scheme of rings $\V$. If $\mu:k[F]\mapsto k[F]\otimes k[\V]$ is the map of coordinate rings corresponding to scalar multiplication $\V \times F\mapsto F$, the corresponding structure can be described as follows. Let $\mu(f)=\sum_{i=1}^r f_i\otimes f'_i$ then the action can be written $(a\cdot f)(am)=\sum_{i=1}^rf'_i(a)f_i(m)$.

\begin{proposition}\label{wtf} Let $L$ be a lattice in $F$ of codimension $R$ with elementary divisors $\alpha_1, \dots ,\alpha_n$ (nonincreasing order). Let $\{ x_{i,j}: 0\leq j\leq \alpha_i\}$ be a principal invariant system of parameters for $L$ in $F$. Let $\overline{x}_{i,j}$ denote the residue class of $x_{i,j}$ in $I/I^2$ where $I$ is the ideal defining $L$. Then:
\begin{enumerate}
  \item $\alpha^N(\overline{x}_{i,j})=\overline{x}_{i,j}\otimes 1+1\otimes \overline{x}_{i,j}$.
  \item $\mu^N(\overline{x}_{i,j})=\overline{x}_{i,j}\otimes f_{i,j}$ for some $f_{i,j}\in k[\V_r]$.
  \end{enumerate} 
\end{proposition} 

\begin{proof} The first two assertions follow from weight considerations. First note that $\alpha^N(x_{i,j})=\Phi_j(x_{i,0}\otimes 1,x_{i,1}\otimes 1\dots ,x_{i,j}\otimes 1;1\otimes x_{i,0}, 1\otimes x_{i,1},\dots ,1\otimes x_{i,j})$. Now $\Phi_j$ is a sum of monomials $M\otimes N$ where the total weight of the tensor product is $p^j$. Since $x_{i,j}$ is itself of weight $p^j$ the only terms involving $x_{i,j}$ are $x_{i,j}\otimes 1$ and $1\otimes x_{i,j}$. Every other term must be of weight $p^j$. Since each term is a tensor product $M\otimes N$ with $M$ and $N$ monomials only those monomials consisting of a single variable on each side of the tensor product is nonzero modulo $I^2\otimes I+I\otimes I^2$. Thus only terms $x_{i,r}\otimes x_{i,s}$ with both $r$ and $s$ strictly less than $j$ are possible. Such terms, of weight $p^r+p^s$, are never of weight $p^j$. Hence $\alpha^N(x_{i,j})=c_1x_{i,j}\otimes 1+c_2\otimes x_{i,j}$. Commutativity implies that $c_1=c_2$ and the fact that translation by $0$ is the identity implies that $c_1=1$.

Now consider $\mu^N$. Write a typical element of $\V$ as $\sum_{i=0}^{\infty}\xi(y_i)^{p^{-i}}p^i$ so that $k[\V]=k[y_0,y_1,\dots ]$. Then $\mu(x_{i,j})=\Psi_j(x_{i,0},\dots ,x_{i,j};y_0,\dots ,y_j)$. Write $\Psi_j(x_{i,0},\dots ,x_{i,j};y_0,\dots ,y_j)=\sum_{q=1}^R M_q\otimes h_q$ where the $M_i$ are monomials in the $x_{i,r}$ and the $h_q$ are polynomials in the $y_r$. Both the $M_q$ and the $h_q$ must be homogeneous of weight $p^j$. The only monomial in the $x_{i,r}$ homogeneous of weight $p^j$ and not in $I^2$ is $x_{i,j}$. Hence $\mu^N(\overline{x}_{i,j})=\overline{x}_{i,j}\otimes f_{i,j}$ where $f_{i,j}$ is the polynomial in the $y_i\;0\leq i\leq j$ of weight $p^j$ which occurs as the right coefficient of $x_{i,j}$.
 \end{proof}

Consider the symmetric algebra on $\mathfrak{n}_{L/F}$. We write it $k[\mathfrak{n}_{L/F}]$. As we noted the spectrum of this algebra is identified with the linear dual of $\mathfrak{n}_{L/F}$, that is the restricted Lie algebra $\mathfrak{F}/\mathfrak{L}$. This is an additive group with addition (the standard one) defined by the map $\alpha^N$. That is $\alpha^N$ can be thought of as inducing an algebra morphism from $k[\mathfrak{n}_{L/F}]$ to $k[\mathfrak{n}_{L/F}]\otimes k[\mathfrak{n}_{L/F}]$ and this map is the map of algebras describing the standard addition on $\mathfrak{F}/\mathfrak{L}$. Similarly the map $\mu^N:\mathfrak{n}_{L/F}\rightarrow \mathfrak{n}_{L/F}\otimes k[M]$ induces a homomorphism of algebras $\mu^N:k[\mathfrak{n}_{L/F}]\rightarrow k[\mathfrak{n}_{L/F}]\otimes k[\V]$. This can be thought of as the map of coordinate rings corresponding to an  action $\V\times_k(\mathfrak{F}/\mathfrak{L})\rightarrow \mathfrak{F}/\mathfrak{L}$.

\begin{theorem}Let $L\subseteq F$ be a codimension $nr$ lattice in $F$ with elementary divisors $\alpha_1\geq \alpha_2\geq \dots \geq \alpha_n$. Let $I$ be the ideal in $k[F]$ defining $F$. Let $\{ x_{i,j}: 0\leq j< \alpha_i\}$ be a principal invariant system of parameters for $L$. Let $\widetilde{I}_{\delta}$ be an infinitesimal deformation of the ideal $I$ defining a scheme of lattices of codimension $nr$ in $\ke[\widetilde{F}]$.
\begin{enumerate}
  \item The elements $\overline{x}_{i,j}$, the residue classes of the elements $x_{i,j}$ modulo $I^2$, constitute a basis for the space of invariant sections in the conormal bundle $I/I^2$.
 
  \item Let $\mathfrak{n}_{L/F}=(I/I^2)^L$ denote the space of translation invariant sections in the conormal bundle $I/I^2$ on $L$. Then the classifying map $\delta$ is uniquely determined by its restriction to the $nr$ dimensional $k$-vector space $\mathfrak{n}_{L/F}$.
  \item For each $\overline{x}_{i,j},\; \delta(\overline{x}_{i,j})\in k[L]$ is an $\V_r$ homomorphism from $L$ to the $\V_r$-module of additive characters $\mathcal{A}^+(L)$. Converseley every $\V$-module map  determines an infinitesimal module scheme deformation of $L$ in $F$. 
\end{enumerate}
\end{theorem}

\begin{proof}For (1) note that $I/I^2$ is a homogeneous bundle on $L$ and so $I/I^2\simeq \mathfrak{n}\otimes_kk[L]$ where the isomorphism is of homogeneous bundles and where the action of $L$ on the right is trivial on $\mathfrak{n}$ and by translation on $k[L]$. Under this isomorphism, $\mathfrak{n}$ is identified with $\mathfrak{n}_{L/F}$ the space of residue classes of $L$-translation invariant parameters.

Consider (2). Since $I/I^2 \simeq \mathfrak{n}_{L/F}\otimes_kk[L]$ by  extension of coefficients we note that $\Hom_{k[L]}(I/I^2, k[L])=\Hom_k(\mathfrak{n}_{L/F},k[L])$.

Now consider (3). Let $u_{i,j}=\delta(x_{i,j})$. By Proposition \ref{cd}, $\alpha(u_{i,j})=\alpha \circ \delta (x_{i,j})=(\delta \otimes \id,\id \otimes \delta)\circ \alpha^N(\overline{x}_{i,j})$ By Proposition \ref{wtf}, $\alpha^N(\overline{x}_{i,j})=\overline{x}_{i,j}\otimes 1+1\otimes \overline{x}_{i,j}$. Hence $\alpha ( \delta (x_{i,j}))=(\delta \otimes \id,\id \otimes \delta)\circ \alpha^N(\overline{x}_{i,j})= \delta(x_{i,j})\otimes 1+1\otimes \delta(x_{i,j})$. Recalling the definition of the $u_{i,j}$, this says that $\alpha(u_{i,j})=u_{i,j}\otimes 1+1\otimes u_{i,j}$. This is exactly the condition which forces $u_{i,j}$ to be an additive homomorphism from $L$ to $\mathbb{G}_{a,k}$.

We now give a geometric description of the classifying maps $\delta$ that define lattice deformations. For this we note that a principal invariant system of parameters can be thought of as a set of indeterminates generating the coordinate ring of the quotient group $F/L$. By definition they are also a system of parameters for the maximal ideal defining the identity. Write $k[F/L]=k[\{x_{i,j}:0\leq j<\alpha_i\}]$ where the $\alpha_i$ are as in Proposition A.5. Write $\mathfrak{m}$ for the ideal defining $0$ in $F/L$. Then the map of coordinate rings $k[F/L]\hookrightarrow k[F]$ corresponding to the projection $F\rightarrow F/L$ induces isomorphisms $\mathfrak{m}/\mathfrak{m}^2\simeq (\mathfrak{m}+I^2)/I^2 \simeq \mathfrak{n}_{F/L}$. Hence $\mathfrak{n}_{F/L}$ can be identified with the dual of the Lie algebra of the lattice $F/L$. Under this identification the conormal map in assertion (1) of Proposition A.5 is identified with the additive structure on $\Spec (k[\mathfrak{n}_{F/L}])$ (the symmetric algebra on $\mathfrak{n}_{F/L}$) which can be thought of as the Lie algebra of $F/L$. Since $\V$ operates as endomorphisms of $F/L$ it operates as endomorphisms of the Lie algebra of $F/L$. The conormal map in equation (2) of Proposition A.5 can be interpreted as a map extending to a coaction $k[\mathfrak{n}_{F/L}]\rightarrow k[\mathfrak{n}_{F/L}]\otimes_kk[\V]$ describing this $\V$-action on the Lie algebra of $F/L$. Write $\mathfrak{F}$ for the Lie algebra of $F$ and $\mathfrak{L}$ for that of $L$. Then the Lie algebra of $F/L$ is canonically equal to $\mathfrak{F}/\mathfrak{L}$.

 By Proposition \ref{wtf} $\mu^N(\overline{x}_{i,j})=\overline{x}_{i,j}\otimes f_{i,j}$. In particular $\mu(u_{i,j})=\mu(\delta(x_{i,j}))=\delta \otimes \id (\mu^N(\overline{x}_{i,j}))=\delta \otimes \id (\overline{x}_{i,j}\otimes f_{i,j})=u_{i,j}\otimes f_{i,j}$. That is $\mu(u_{i,j})=u_{i,j}\otimes f_{i,j}$. Now to write $\mu(g)=\sum_{i=1}^qg'_i\otimes g''_i$ means that the equation $\sum_{i=1}^q g'_i(m)g''_i(a)=g(am)$. That is $u_{i,j}(am)=f_{i,j}(a)u_{i,j}(m)$. That is $a\cdot u_{i,j}=f_{i,j}(a)u_{i,j}$ and $u_{i,j}(am)=f_{i,j}(a)u_{i,j}(m)$. Then $u_{i,j}(abx)=f_{i,j}(a)u_{i,j}(bx)=f_{i,j}(a)f_{i,j}(b)u_{i,j}(x)$. Hence $f_{i,j}$ is multiplicative. Similarly, $f_{i,j}(a+b)u_{i,j}(x)=u_{i,j}((a+b)x)=u_{i,j}(ax)+u_{i,j}(bx)=f_{i,j}(a)u_{i,j}(x)+f_{i,j}(b)u_{i,j}(x)=(f_{i,j}(a)+f_{i,j}(b))u_{i,j}(x)$. That is $f_{i,j}$ is both multiplicative and additive and so defines an embedding $k\hookrightarrow \Hom_k(\mathcal{A}^+(L),\mathcal{A}^+(L))$. Furthermore $u_{i,j}(am)=f_{i,j}(a)u_{i,j}(m)$. That is $u_{i,j}$ is additive and linear for the contragredient action.

Thus $\delta$ maps $\mathfrak{n}_{L/F}$ into the contragredient linear additive characters on $L$. It hence induces a morphism of affine schemes $\phi_{\delta}:L\mapsto \mathfrak{F}/\mathfrak{L}$. The map $\delta$ is linear for the natural vector space structure on $\mathcal{A}^+(L)$. Moreover Propositions A.2, A.3 and A.5 imply that $\mu(\delta(f))=(\delta \otimes \id) (\mu(f))$. Explicitly, if $u_{i,j}\in \mathfrak{n}_{L/F}$ is an invariant parameter then $\mu^N(u_{i,j})=u_{i,j}\otimes f_{i,j}$ This equation precisely asserts that diagram A.8 is commutative and this in turn simply says that the following diagram commutes:

\[
\begin{CD}
       \V \otimes \widetilde{L}@>\mu >> \widetilde{L}  \\
       @V\id \otimes \phi_{\delta} VV    @V\phi_{\delta} VV \\
        \V \otimes \mathfrak{F}/\mathfrak{L}@>\mu^N>>\mathfrak{F}/\mathfrak{L}\end{CD}
\]

Conversely assume $\delta$ defines an $\V$ map from $\mathfrak{n}_{L/F}$ into $\mathcal{A}^+(L)$. Since it is a $k$-linear from $I/I^2$ into $k[L]$, it defines an infinitesimal deformation of the ideal $I$. Since its image lies in $\mathcal{A}^+(L)$, Proposition A.5, assertion (1) implies that the deformation it defines is the ideal of a subscheme closed under addition. Since it is an $\V$ module map and since all the schemes under consideration are varieties of finite type over a field, it follows that equation A.10 holds. This in turn implies that $\delta$ is the classifying map of an ideal defining a $\ke$-lattice of corank $nr$ in $\widetilde{F}$. 
 \end{proof}
\begin{remark} In this proof a rarely noted phenomenon of basic significance occurs. To wit $\mu^N$ defines an exotic vector space structure on the space of additive characters. In general if $V$ is a vector space over $k$ and $v_1\dots ,v_n$ is an ordered basis, define an action of $k$ on $V$ by the equations $a\circ v_i=a^{p^\nu_i}v_i$. This makes $V$ into a vector space over $k$ with an exotic structure.\end{remark}

\begin{theorem}Write $\mathcal{D}ef(L;F)$ for the set of infintesimal deformations of the lattice $L$ of corank $nr$ in $F$. This is a subspace of $\Hom_{k[L]}(I/I^2,k[L])$. There is a bijective correspondence between $\mathcal{D}ef(L;F)$ and the space $\Hom_{\V}(L/pL, \mathfrak{F}/\mathfrak{L})$ of morphisms of $\V$-modules from $L/pL$ to $\mathfrak{F}/\mathfrak{L}$. These are morphisms in the category of schemes of modules over schemes of rings This is a vector space of dimension $(n-1)nr$ for a lattice with elementary divisors $\{nr,0,0  \dots ,0\}$ and it is of dimension $n^2r$ for all other elementary divisor types. \end{theorem}

\begin{proof}The infinitesimal deformations of the lattice $L$ in $F$ are in bijective correspondence with the classifying maps $\delta$ satisfying Equations A.9 and A.10 of Proposition A.3. These maps are uniquely determined by their restrictions to $\mathfrak{n}_{F/L}$. Hence we may view them as maps from $\mathfrak{n}_{F/L}$ to $k[L]$ which satisfy A.9 and A.10. So suppose that $\delta$ is one such map. Now $\mathfrak{n}_{F/L}$ is spanned by the residue classes $\overline{x}_{i,j}$. As we remarked above the algebra $k[\mathfrak{n}_{F/L}]$ with the coaddition induced by the map $\alpha^N$ and the $k[\V_r]$ coaction induced by $\mu^N$ can be viewed as the coaction  determining the  $\V_r$ module structure on $\mathfrak{F}/\mathfrak{L}$ corresponding to the differential of the multiplicative homothety. A map from $\mathfrak{n}_{F/L}$ into $k[L]$ canonically determines an algebra map $k[\mathfrak{n}_{F/L}]\rightarrow k[L]$ and so corresponds to a map of varieties $\delta^* :L\rightarrow \mathfrak{F}/\mathfrak{L}$. Condition A.9 of Proposition A.3 and assertion (1) of Proposition A.5 imply that this morphism is an additive morphism of group schemes. Now $\mathfrak{F}/\mathfrak{L}$ is a vector space over a field of characteristic $p$ and so $\delta^*$ must vanish on $pL$. Hence the assignment $\delta \mapsto \delta^*$ maps $\mathcal{D}ef(L;F)$ to $\Hom_{\V}(L/pL, \mathfrak{F}/\mathfrak{L})$ where the homomorphisms are taken to be the algebraic morphisms for the additive and multiplicative structures. We give the inverse map.

Now the groupscheme $L/pL$ is the polynomial algebra $k[x_{1,\alpha_1}, \dots ,x_{n,\alpha_n}]$ with the coaction $\alpha(x_{i,\alpha_i})=x_{i,\alpha_i}\otimes 1+1\otimes x_{i,\alpha_i}$. The functions $x_{i,\alpha_i}$ are to be thought of as $pL$ invariant functions on $L$. Write $\overline{L}^*$ for the $k$ span of the functions $x_{i,\alpha_i}$. Then $k[\overline{L}^*]$ (the symmetric algebra on $\overline{L}$) is the coordinate ring of $L/pL$. Since $pL$ is an $\V$-submodule, it follows that the coaction $\mu:k[L]\rightarrow k[L]\otimes k[\V_r]$ restricts to a coaction on $k[L/pL]$ and under this coaction it can be verified that $\mu(\overline{L}^*)\subseteq \overline{L}^*\otimes k[\V_r]$. A map $\gamma:L/pL\rightarrow \mathfrak{F}/\mathfrak{L}$ of algebraic modules corresponds to a map of coordinate rings $k[\mathfrak{F}/\mathfrak{L}]\rightarrow k[L/pL]$ that "commutes" with the structure data. But $k[\mathfrak{F}/\mathfrak{L}]=k[\mathfrak{n}_{L/F}]$ and $k[L/pL]=k[\overline{L}^*]$. Hence $\gamma$ is determined by its restriction $\gamma:\mathfrak{n}_{L/F}\rightarrow k[\overline{L}^*]\subseteq k[L]$. It will give rise to an additive morphism if $\alpha^N \circ \gamma(u)= (\gamma \otimes 1,1\otimes \gamma)\circ \alpha^N$. Lastly $\gamma$ will be an algebraic module map only if $(\gamma \otimes \id) \circ \mu^N = \mu \circ \gamma$. These conditions imply that $\gamma$ is a map satisfying equations A.9 and A.10.   \end{proof} 
    
\subsection{Constructing explicit deformations}  For every lattice $L$ of co-rank $nr$ in $F$ there is some ordered basis $f_1, f_2,\dots ,f_n$ of $F$ so that $p^{\alpha_1}f_1, p^{\alpha_2}f_2,\dots, p^{\alpha_n}f_n$ is a basis of $L$. In this section it will be convenient to assume that the $\alpha_i$ are non-decreasing in $i$. This is opposite to our usual convention. We will refer to the sequence of elementary divisors $\{\alpha_1,\dots, \alpha_n\}$ as the type of $L$. Since all lattices of a given type lie in one $\SL(n,\V)$ orbit it suffices to consider lattices for which the $f_i$ can be taken to be the standard basis vectors $e_i$. We show that every type is a specialization of the maximal type, $\{0,\dots, 0, nr\}$ and that the type $\{r,r,\dots, r\}$ is a specialization of all types. If $\{\alpha_1, \dots, \alpha_n\}$ is not one of the two extreme types then there is some pair $\alpha_q, \alpha_s;\;q<s$ so that $\alpha_q<r$ and $\alpha_s>r$. We begin with a matrix $M(t)$ with entries in the Witt vectors evaluated on the polynomial domain $k[t]$. First we show that for non-zero values of $t$ the column space of this matrix is of type $\{\alpha_1,\dots, \alpha_n\}$ and that at $t=0$ it is of type $\{\alpha_1,\dots, \alpha_q+1,\dots, \alpha_s-1,\dots, \alpha_n\}$.  Then, viewing $M(t)$ as a morphism of $k[t]$-group schemes from $\mathbb{A}^1\times F$ to itself, we explicitly construct the coordinate ring of a family as the image of the morphism and show it to be a smooth family of group subschemes of $\mathbb{A}^1\times F$.

We begin with the definition of $M(t)$. For $j\neq q,s$, let $M_j=p^{\alpha_je_j}$. Let $M_q=p^{\alpha_q+1}e_q$ and let $M_s=\xi(t)^{p^{-\alpha_q}}p^{\alpha_q}e_q+p^{\alpha_s-1}e_s$. Notice that since $t$ does not have any $p$'th roots $\xi(t)^{p^{-\alpha_q}}p^{\alpha_q}$ is simply a Witt vector with one component and not a $k[t]$ multiple of $p^{\alpha_q}$. When $t$ is replaced by a value in the perfect field $k$ it becomes a unit multiple of $p^{\alpha_q}$.

 When $t$ is set equal to $0$ the columns yield a lattice of the desired specialization type. When $t=c\neq 0$ then by the remark above $M_q-\xi(c^{-p^{-\alpha_q}})^{-p}pM_s=\xi(c^{-p^{-\alpha_q}})^{-p}p^{\alpha_s}e_s$. Consequently $M_q$ and $M_q-\xi(c^{-p^{-\alpha_q}})^{-p}pM_s$ span a lattice of type $\{\alpha_1, \dots ,\alpha_q \} $. Hence the reduced parts of the fibers of this family behave exactly as required. What has not been established is that the family is a flat family with reduced fibers. To do this let $\{x_{j,i},\;0<i\leq nr-1$ be a set of indeterminates and let  $\x_j= \sum_0^{nr-1}\xi(x_{j,i})^{p^{-i}}p^i$ be $n$ (truncated) Witt vectors with indeterminate coefficients and let $\boldsymbol{X}$ denote an $n$-vector with entries $\x_i$. Let $F_0(t)$ and $F_1(t)$ be two isomorphic copies of $F/p^{nr}F\times_k\Spec (k[t])$ which we will call the source and the target of $M(t)$. Let $X$ denote the set of indeterminates $x_{i,j}$ defined above and let $Y$ be a set of indeterminates $y_{i,j}$ with the same indices.
Write $F_0(t)=\Spec (k[t,X]$ and $F_1(t)=k[t,Y]$. We consider the vector $M(t)\boldsymbol{X}$ which has $n$ truncated Witt vector entries. The contravariant morphism of rings will be written $m^\#$. Then $m^\#$ sends $y_{j,i}$ to the $i$'th component of the $j$'th entry of $M(t)\boldsymbol{X}$. Let $\y_j=\sum_0^{nr-1}\xi(y_{j,i})^{p^{-i}}p^i$ and let $\boldsymbol{Y}$ be a vector with $j$'th entry $\y_j$. Then, the image of multiplication by $M(t)$ is defined by the equation:
\[
\boldsymbol{Y}=M(t)\boldsymbol{X}=\left(\begin{matrix}p^{\alpha_1}\x_1 \\ \vdots \\ p^{\alpha_q+1}\x_q+\xi(t)^{p^{-\alpha_q}}p^{\alpha_q}\x_s \\ \vdots \\ p^{\alpha_s-1}\x_s \\ \vdots \\ p^{\alpha_n}\x_n  \end{matrix}\right)
\] 

Computing the Witt components of the entries of this column and setting them equal to the $m^\#$ image of the corresponding components in $\boldsymbol{Y}$ yields: 

\begin{align*}
 m^\#(y_{j,i})&=0\; \text{if}\; i<\alpha_j,j\neq q,s \\
    m^\#(y_{j,i})&=x_{j,i-\alpha_j}^{p^{\alpha_j}}\; \text{if}\; i\geq \alpha_j \\
    m^\#(y_{s,i})&=0 \; \text{if} i<\alpha_s-1 \\ 
    m^\#(y_{s,i})&=x_{s,i-\alpha_s+1}^{p^{\alpha_s-1}}\; \text{if}\;
 i\geq \alpha_s-1 \\ m^\#(y_{q,i})&=0 \; \text{if}\; i<\alpha_q \\
m^\#(y_{q,\alpha_q})&=tx_{s,0}^{p^{\alpha_q}}\\
m^\#(y_{q,j})&=\Phi_j(0,\dots, 0, x_{q,0}^
{p^{\alpha_{q}+1}}\dots,\\ &\quad x_{q,j-\alpha_q-1}^{p^{\alpha_q+1}}; 0\dots ,0,0, x_{s,1}^{p^{\alpha_q}},\dots ,x_{s,j-\alpha_q}^{p^{\alpha_q}})\;\text{if}\; j>\alpha_q
\end{align*}

If $\Phi_j(u_0,\dots ,u_j, v_0,\dots v_j)=\Phi_j$ is one of the Witt polynomials associated to addition then $\Phi_j=u_j+v_j+H$ where $H$ is a plynomial stricly in varables of lower index and in which each monomial contains both $u$ and $v$ terms. Now $m^\#(y_{j,i})=0$ for $j\neq s$ and $i<\alpha_j$ while $m^\#(y_{s,i})=0$ for $i<\alpha_s-1$. The other variables except for $y_{q,\alpha_q}$ are sent to a set of algebraically independent polynomials. The variable $y_{q,\alpha_q}$ however satisfies a relation. We note that $m^\#(y_{q,\alpha_q}-ty_{s,\alpha_s-1})=tx_{s,0}^{p^{\alpha_q}}-t(x_{s,0}^{p^{\alpha_s-1}})^p=0$. Thus if $J$ is the kernel of $m^\#$ it is the ideal generated by the expressions:

\begin{align*}
y_{j,i}&\;i<\alpha_j\; j\neq s \\ y_{s,j}&\;j<\alpha_s-1 \\ \intertext{and} 
y_{q,\alpha_q}&-ty_{s,\alpha_s-1}^p
\end{align*}

The $k[t]$-endomorphism of $k[t,Y]$ which is the identity on all the $y_{j,i}$ but $y_{q,\alpha_q}$ and which sends
$y_{q,\alpha_q}$ to $y_{q,\alpha_q}-ty_{s,\alpha_s-1}^p$ is an automorphism and so it is clear that $k[t,Y]/J$ is just a ring of polynomials in $(n^2-n)r$ variables over $k[t]$. Thus the family of deformations $M(t)X$ is smooth (and hence flat) over $k[t]$.

\end{document}